\documentclass[11pt]{amsart}
\usepackage[dvipsnames,usenames]{color}
\usepackage{lipsum}
\usepackage{filecontents}
\usepackage[dvipsnames]{xcolor}
\usepackage{hyperref}
\usepackage{cleveref}

\newcommand\myshade{85}
\colorlet{mylinkcolor}{NavyBlue}
\colorlet{mycitecolor}{NavyBlue}
\colorlet{myurlcolor}{NavyBlue}

\hypersetup{
  linkcolor  = mylinkcolor!\myshade!black,
  citecolor  = mycitecolor!\myshade!black,
  urlcolor   = myurlcolor!\myshade!black,
  colorlinks = true,
}
\usepackage[utf8]{inputenc}
\usepackage{amsmath}
\usepackage{tabularx,booktabs}
\newcommand{\ra}[1]{\renewcommand{\arraystretch}{#1}}
\usepackage{adjustbox}
\usepackage{tikz-cd}
\usepackage{amssymb}
\usepackage{epsfig}
\usepackage{epstopdf}
\usepackage{ifpdf}
\usepackage{amsthm}
\usepackage{soul}
\usepackage{caption}
\usepackage{longtable}
\usepackage{multirow,array}
\usepackage{rotating}
\usepackage{indentfirst}
\usepackage{arydshln}

\newtheorem{prop}{Proposition}[section]
\newtheorem{lemma}[prop]{Lemma}
\newtheorem{cor}[prop]{Corollary}
\newtheorem{theorem}[prop]{Theorem}

\newtheorem{ques}[prop]{Question}
\newtheorem{conj}[prop]{Conjecture}

\theoremstyle{definition}
\newtheorem{definition}[prop]{Definition}
\newtheorem{ex}[prop]{Example}
\newtheorem{rmk}[prop]{Remark}

\newcommand{\Z}{\mathbb{Z}}
\newcommand{\Q}{\mathbb{Q}}

\newcommand{\nc}{\newcommand} 
\nc{\BL}{\mathbb L}
\nc{\on}{\operatorname} 
\nc{\CS}{\mathcal S}
\nc{\wh}{\widehat}
\nc{\im}{\operatorname{im}}

\newcommand{\supp}{\operatorname{supp}}

\newcommand{\braket}[2]{\left\langle #1 , #2 \right\rangle} 

\newcommand{\abs}[1]{\lvert #1 \rvert}

\author[W. Ballinger]{William Ballinger}
\address {Department of Mathematics, California Institute of Technology, Pasadena, CA 91125}
\email{wballing@caltech.edu}

\author[C. Hsu]{Chloe Ching-Yun Hsu}
\address {Department of Mathematics, California Institute of Technology, Pasadena, CA 91125}
\email{chhsu@caltech.edu}

\author[W. Mackey]{Wyatt Mackey}
\address {Department of Mathematics, Harvard University, Cambridge, MA 02138}
\email{wmackey@college.harvard.edu}

\author[Y. Ni]{Yi Ni}
\address {Department of Mathematics, California Institute of Technology, Pasadena, CA 91125}
\email{yini@caltech.edu}

\author[T. Ochse]{Tynan Ochse}
\address {Department of Mathematics, California Institute of Technology, Pasadena, CA 91125}
\email{tochse@caltech.edu}

\author[F. Vafaee]{Faramarz Vafaee}
\address {Department of Mathematics, California Institute of Technology, Pasadena, CA 91125}
\email{vafaee@caltech.edu}

\begin{document}

\title{The prism manifold realization problem}

\date{}

\maketitle
\begin{abstract}
The spherical manifold realization problem asks which spherical three-manifolds arise from surgeries on knots in $S^3$. In recent years, the realization problem for C, T, O, and I-type spherical manifolds has been solved, leaving the D-type manifolds (also known as the prism manifolds) as the only remaining case. Every prism manifold can be parametrized as $P(p,q)$, for a pair of relatively prime integers $p>1$ and $q$. We determine a complete list of prism manifolds $P(p, q)$ that can be realized by positive integral surgeries on knots in $S^3$ when $q<0$. The methodology undertaken to obtain the classification is similar to that of Greene for lens spaces.
\end{abstract}

\section{Introduction}\label{sec:intro}

There are many notions of simplicity for closed three-manifolds. Perhaps the simplest is that of a manifold with a finite fundamental group. One of the most prominent problems in three-manifold topology is to indicate the list of the simplest closed three-manifolds that can be realized by the simplest three-dimensional topological operations. Given that every closed three-manifold can be obtained by performing surgery on a link in $S^3$~\cite{Lickorish1962, Wallace1960}, the aforementioned realization problem may be stated as follows:
\begin{ques}\label{ques:Surg}
Which closed $3$--manifolds with finite fundamental groups can be realized by surgeries on nontrivial knots in $S^3$? 
\end{ques}

By the work of Thurston~\cite{Thurston1982}, any knot in $S^3$ is precisely one of a torus knot, a satellite knot or a hyperbolic knot. Moser classified all finite surgeries on torus knots~\cite{Moser1971}. Later, Boyer and Zhang showed that if surgery on a satellite knot $K\subset S^3$ results in a manifold with a finite fundamental group, then $K$ must be a cable of a torus knot~\cite[Corollary~1.4]{Boyer1996}. Such surgeries are classified by Bleiler and Hodgson in~\cite[Theorem~7]{Bleiler1996}. 
In regard to the surgery coefficient, Culler--Gordon--Luecke--Shalen \cite{Gordon1987} proved that any cyclic surgery must be integral. As further proved by Boyer and Zhang~\cite[Theorem~1.1]{Boyer1996}, the coefficient of any finite surgery is either $p$ or $p/2$, for some integer $p$. Li and Ni showed that if half integral surgery on a hyperbolic knot results in a manifold $Y$ with a finite fundamental group, then $Y$ is homeomorphic to $p/2$ surgery on either a torus knot or a cable of a torus knot~\cite{LiNi2014}. As a result, we henceforth restrict attention to integral surgeries. 

Using Perelman's Geometrization Theorem, closed three-manifolds with finite fundamental groups can be characterized as those three-manifolds that admit {\it spherical geometry}. A {\it spherical $3$--manifold} (also known as an {\it elliptic $3$--manifold}) has the form
\[
Y=S^3/G,
\]
where $G$ is a finite subgroup of $SO(4)$ that acts freely on $S^3$ by rotations. The center $Z=Z(G)$ of $G\cong\pi_1(Y)$ is necessarily a cyclic group. According to the structure of $G/Z$, spherical manifolds (besides $S^3$) are divided into five types: {\bf C} or cyclic, {\bf D} or dihedral, {\bf T} or tetrahedral, {\bf O} or octahedral, {\bf I} or icosahedral. In particular, if $G/Z$ is the dihedral group 
\[D_{2p}=\langle x,y|x^2=y^2=(xy)^p=1\rangle,\]
for some integer $p>1$, we get the {\bf D}--type manifolds. These manifolds are also known as the {\it prism manifolds}.


Greene~\cite{greene:LSRP} solved the integer surgery realization problem (that is, Question~\ref{ques:Surg} when the surgery coefficient is integral) for lens spaces, namely, the {\bf C}--type manifolds. Later, Gu~\cite{Gu2014} provided the solution for {\bf T}, {\bf O}, and {\bf I}--type manifolds. This leaves the {\bf D}-type manifolds as the only remaining case, and that is the theme of the present work.

There are much more prism manifolds than any other types of spherical manifolds. It is straightforward to check that for each integer $m>0$, there are only finitely many spherical manifolds $Y$ of other types with $|H_1(Y)|=m$. However, for each $m$ divisible by $4$, there are infinitely many prism manifolds with the order of the first singular homology equal to $m$. To justify, let $P(p,q)$ be the oriented prism manifold with Seifert invariants
\begin{equation}\label{eq:PrismSeifert}(-1;(2,1),(2,1),(p, q)),
\end{equation}
where $p>1$ and $q$ are a pair of relatively prime integers. 
These manifolds satisfy
\begin{equation}\label{eq:PrismHomology}
|H_1(P(p, q))|=4|q|.
\end{equation}
Therefore, any integer $p>1$ relatively prime to $q$ will give a prism manifold $P(p, q)$ with the desired order of the first singular homology. In regard to the realization problem, however, we still have a finiteness result. It was first proved by Doig in~\cite{Doig2013} that, for a fixed $|q|$, there are only finitely many $p$ for which $P(p, q)$ may be realized by surgery on a knot $K\subset S^3$. Later, Ni and Zhang \cite{Ni2016} proved an explicit bound for $p$ in terms of $q$:
$$p< 4|q|.$$

All the known examples of integral cyclic surgeries (lens space surgeries) come from Berge's primitive/primitive (or simply P/P) construction \cite{Berge}. There is a generalization of this construction to Seifert-fibered surgeries due to Dean \cite{Dean2003}, called the primitive/Seifert-fibered (or P/SF) construction. See Definition~\ref{def:P/SF}. The {\it  surface slope} Dehn surgery on a hyperbolic P/SF knot results in a Seifert fibered space. Berge and Kang, in~\cite{BergeKang}, classified all P/SF knots in $S^3$. Further, they specified the indices of the singular fibers of the Seifert fibered manifolds resultant from the surface slope surgeries on such knots. Since prism manifolds are Seifert fibered spaces over $S^2$ with three singular fibers of indices $(2,2,p)$, following from the work of Berge and Kang, we obtain a list of prism manifolds that are realizable by knot surgeries. See Table~\ref{BK Prism with slope sign}.

We are now in a position to state the main result of the paper:

\begin{theorem}\label{thm:Classification}
Given a pair of relatively prime integers $p>1$ and $q<0$, if the prism manifold $P(p,q)$ can be obtained by $4|q|$--Dehn surgery on a knot $K$ in $S^3$, then there exists a Berge--Kang knot $K'$ such that the $4|q|$--surgery on $K'$ results in $P(p,q)$, and
$K$ and $K'$ have isomorphic knot Floer homology groups.
Moreover, $P(p,q)$ belongs to one of the six families in Table~\ref{table:Types}.
\end{theorem}

\begin{rmk}
The six families in Table~\ref{table:Types} are divided so that each changemaker vector (see Definition~\ref{defn:changemaker}) corresponds to a unique family. However, a prism manifold $P(p,q)$ could belong to different families, and not just a unique one. We will address the overlaps between these families in Subsection~\ref{subsect:Overlap}. See Table~\ref{Overlap}.
\end{rmk}

The methodology undertaken to prove Theorem~\ref{thm:Classification} is inspired from that of Greene~\cite{greene:LSRP}. A prism manifold $P(p,q)$ with $q<0$ naturally bounds a negative definite four-manifold $X(p,q)$. See Section~\ref{sec:Background}. Suppose that $P(p,q)$ is realized by $4|q|$--surgery on a knot $K\subset S^3$. In particular, $P(p,q)$ bounds the two-handle cobordism $W_{4|q|}=W_{4|q|}(K)$, obtained by attaching a two-handle to $D^4$ along $K\subset S^3$ with framing $4|q|$. Note that the surgery coefficient is dictated by homology considerations: see Equation~\eqref{eq:PrismHomology}. The four-manifold $Z := X(p,q) \cup -W_{4|q|}$ is a smooth, closed, negative definite 4-manifold with $b_2(Z)=n+4$, where $n+3=b_2(X(p,q))$ for some $n\ge 1$. Donaldson's ``Theorem A" implies that the intersection pairing on $H_2(Z)$ is isomorphic to $-\mathbb Z^{n+4}$, negative of the standard $(n+4)$-dimensional Euclidean integer lattice~\cite{Donaldson1983}. Consequently, negative of the intersection pairing on $X(p,q)$, denoted $\Delta(p,q)$, embeds as a codimension one sub-lattice of $\mathbb Z^{n+4}$. For the prism manifold $P(p,q)$ to arise from a knot surgery, this already gives a restriction on the pair $(p,q)$. We then appeal to the innovative work of Greene that provides even more constraints on $\Delta(p,q)$. To state this essential step, we first need to make a combinatorial definition.  

\begin{definition}\label{defn:changemaker}
A vector $\sigma=(\sigma_0,\sigma_1,\dots,\sigma_{n+3})\in\mathbb Z^{n+4}$ that satisfies $0\le\sigma_0\le\sigma_1\le\cdots\le\sigma_{n+3}$ is a {\it changemaker vector} if for every $k$, with $0\le k\le\sigma_0+\sigma_1+\cdots+\sigma_{n+3}$, there exists a subset $S\subset\{0,1,\dots,n+3\}$
such that $k=\sum_{i\in S}\sigma_i$.
\end{definition}

The lattice embedding theorem of Greene~\cite[Theorem~3.3]{Greene2015} now reads as follows: if $P(p,q)$ with $q<0$ is realized by $4|q|$-surgery on $K\subset S^3$, then $\Delta(p,q)$ embeds into $\mathbb Z^{n+4}$ as the orthogonal complement of a changemaker vector $\sigma\in \mathbb Z^{n+4}$. 

By determining the pairs $(p, q)$ which pass this refined embedding restriction, we get the list of all prism manifolds that could possibly be realized by positive integral surgeries on knots. It only remains to verify that, corresponding to every prism manifold $P(p,q)$ in our list, there is a knot $K\subset S^3$ on which some surgery yields $P(p,q)$. Indeed, this is the case.

\begin{theorem}\label{thm:lattice}
Given a pair of relatively prime integers $p>1$ and $q<0$,
$\Delta(p,q)\cong(\sigma)^{\perp}$ for a changemaker vector $\sigma\in\mathbb Z^{n+4}$ if and only if $P(p,q)$ belongs to one of the six families in Theorem~\ref{thm:Classification}. Moreover, in this case, there exist a knot $K\subset S^3$ and an isomorphism of lattices
\[\phi: (\mathbb Z^{n+4},I)\to (H_2(Z),-Q_Z),\]
satisfying the property that $\phi(\sigma)$ is a generator of $H_2(-W_{4|q|})$. Here $I$ denotes the standard inner product on $\mathbb Z^{n+4}$ and $Q_Z$ is the intersection form of $Z= X(p,q) \cup -W_{4|q|}$.
\end{theorem}

\subsection{Prism manifolds $P(p,q)$ with $q>0$} 

As discussed, digging up the list of P/SF knots in $S^3$~\cite{BergeKang}, we obtain a family of knots with prism manifold surgeries. See Table~\ref{BK Prism with slope sign}.
The parameters $p$ and $|q|$ of the resulting $P(p,q)$ are explicitly given in \cite{BergeKang}. 
By specifying whether $q>0$ or $q<0$ (Lemma~\ref{lem:EulerPos}), we not only confirm that every manifold in Table~\ref{table:Types} is realizable, but also get a list $\mathcal P^+$ of prism manifolds with $q>0$ arising from surgeries on knots in $S^3$. See Table~\ref{table:Types+}.

In light of Theorem~\ref{thm:Classification}, we make the following conjecture.

\begin{conj}\label{conj1} Given a pair of relatively prime integers $p>1$ and $q>0$, if $P(p, q)$ can be obtained by $4q$--Dehn surgery on a knot $K$ in $S^3$, then $P(p,q)\in \mathcal P^+$.
\end{conj}

Theorem~\ref{thm:Classification} leaves open the integer surgery realization problem for manifolds $P(p,q)$ with $q>0$, and Conjecture~\ref{conj1} proposes the solution. A natural direction to pursue is to indicate the list of all knots in $S^3$ that admit surgeries to spherical manifolds. In~\cite{Berge}, Berge proposed a complete list of knots in $S^3$ with lens space surgeries. Indeed, Berge's Conjecture states that the P/P knots form a complete list of knots in $S^3$ that admit lens space surgeries.
All the known examples of spherical manifolds arising from knot surgeries will provide supporting evidence to the following conjecture:

\begin{conj}\label{conj2}
Let $K$ be a knot in $S^3$ that admits a surgery to a spherical manifold. Then $K$ is either a P/SF or a P/P knot.
\end{conj} 

We point out that Conjecture~\ref{conj2} implies Conjecture~\ref{conj1}. 

When $q>p$, $P(p,q)$ is the double branched cover of $S^3$ with branching locus being an alternating Montesinos link, thus it is the boundary of a sharp $4$--manifold \cite{OSzBrDoub}. Greene's strategy can still be used to study the realization problem in this case, but the lattices will not be of D-type. We plan to address this case in a future paper.

\subsection{Organization}

In Section~\ref{sec:Background}, we give the basic topological properties of prism manifolds, and explain how to reduce the realization problem to a problem about lattices. In Section~\ref{sec:DLattice}, we study the D-type lattices which are central in our paper. There is a natural {\it vertex basis} for a D-type lattice. Every vector in the vertex basis is {\it irreducible}. A classification of irreducible vectors is given in Proposition~\ref{irred}. In Section~\ref{sec:Changemaker}, we endow a {\it changemaker lattice} $(\sigma)^{\perp}$ with a {\it standard basis} $S$, and study the question when such a lattice is isomorphic to a D-type lattice. From the standard basis elements of a changemaker lattice we can form an intersection graph (see Definition~\ref{defn:InterGraph}). We collect many structural results about this graph. 

Section~\ref{sec:Blocked} addresses some technical lemmas regarding the non-existence of certain edges in the intersection graph associated to a changemaker lattice. The elements of a standard basis $S$, viewed as an ordered set, are of three types: {\it tight}, {\it just right}, and {\it gappy} (Definition~\ref{stbasis}). As it turns out, the classification of changemaker lattices that are isomorphic to D-type lattices relies highly on the properties of one specific element $v_{f-1}$ of the standard basis: whether it is tight, just right or gappy, together with its placement in $S$. Accordingly, we will do a case by case analysis to enumerate the possible standard bases for such a lattice. This occupies Sections~\ref{sec:a0=2}--\ref{sec:f=3}. Section~\ref{sec:pandq} is devoted to converting these standard bases into vertex bases to get a list of pairs $(p,q)$ corresponding to the D-type lattices found in Sections~\ref{sec:a0=2}--\ref{sec:f=3}.

In Section~\ref{sec:BergeKangPrism} we tabulate all Berge--Kang's P/SF knots that admit prism manifold surgeries. See Table~\ref{BK Prism with slope sign}. By comparing the set of prism manifolds $P(p,q)$ resultant from this list when $q<0$, to that obtained from the lattice embedding constraints, we get that the two sets coincide. Thus we finish the proofs of Theorems~\ref{thm:Classification} and~\ref{thm:lattice}.

\subsection*{Acknowledgements}

Y.N. was partially supported by NSF grant number DMS-1252992 and an Alfred P. Sloan Research Fellowship. W.B., C.H., W.M. and T.O. were supported by Caltech's Summer Undergraduate Research Fellowships (SURF) program. W.B. also wishes to thank Samuel P. and Frances Krown for their generous support through the SURF program. We are grateful to John Berge for sending us the preprint \cite{BergeKang} and some useful programs.


\section{Background}\label{sec:Background}
In this section, we start with recalling some basic facts about prism manifolds, then provide a concise strategy to translate the prism manifold realization problem into a lattice theory question. Meanwhile, the necessary background from Heegaard Floer homology will be cited. 

\subsection{Prism manifolds} 

It is well known that every spherical manifold is a Seifert fibered space, that is, a three-manifold with a surgery diagram as depicted in Figure~\ref{fig:SFS}.\footnote{We only consider Seifert fibered spaces whose base orbifold has genus zero.} The data 
\begin{equation}\label{eq:SeifertData}
(e;(p_1,q_1),(p_2,q_2),\dots,(p_r,q_r))
\end{equation}
are called the {\it Seifert invariants},
where $e$ is an integer, and $(p_1,q_1),\dots,(p_r,q_r)$ are pairs of relatively prime integers such that $p_i>1$. The rational number 
$$e_{\mathrm{orb}}:=e+\sum_{i=1}^r\frac{q_i}{p_i}$$
is called the {\it orbifold Euler number}. The oriented homeomorphism type of a Seifert fibered space is determined by the multi-set 
$$\Big\{\frac{q_i}{p_i}-\Big\lfloor\frac{q_i}{p_i}\Big\rfloor\:\Big|\:i=1,\dots,r\Big\},$$
together with $e_{\mathrm{orb}}$. 
It is elementary to verify that if a Seifert fibered space is a rational homology sphere, it must be the case that $e_{\mathrm{orb}}\ne0$, and
\begin{equation}\label{eq:FirstHomologyG}
|H_1(Y)|=p_1p_2\cdots p_r|e_{\mathrm{orb}}|. 
\end{equation}

\begin{figure}[t] 
\begin{picture}(340,100)
\put(90,0){\scalebox{0.7}{\includegraphics*
{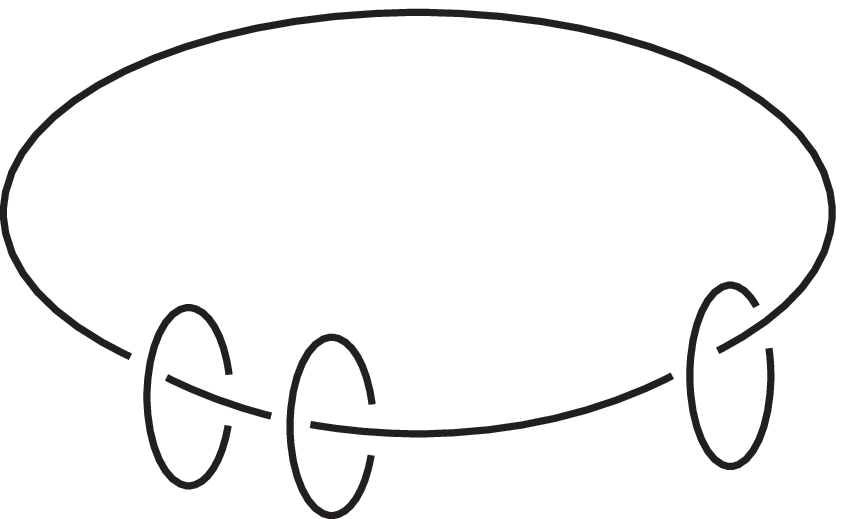}}}

\put(94,90){$e$}

\put(102,5){$-\frac{p_1}{q_1}$}

\put(165,0){$-\frac{p_2}{q_2}$}

\put(190,5){\dots}

\put(215,5){$-\frac{p_r}{q_r}$}

\end{picture}
\caption{\label{fig:SFS}The surgery diagram of a Seifert fibered space}
\end{figure}


For a pair of relatively prime integers $p>1$ and $q$, a prism manifold $P(p, q)$ is a Seifert fibered space over $S^2$ with three singular fibers of indices $(2,2,p)$, having the Seifert invariants 
\[
(-1;(2,1), (2,1), (p, q)).
\]
It is well known that $P(p,q)$ has exactly two Seifert fibrations, and the above one is the only fibration over an orientable base orbifold \cite{Scott}. As a result, the orbifold Euler number of the above Seifert fibration, which is $\frac pq$, is a topological invariant for $P(p,q)$. Hence $P(p_1, q_1) \cong P(p_2,q_2)$ if and only if $(p_1,q_1)=(p_2,q_2)$. Here ``$\cong$" denotes orientation preserving homeomorphism.

Following from their Seifert fibered presentations, prism manifolds enjoy the symmetry 
\[
P(p,-q) \cong -P(p,q),
\]
where $-P(p,q)$ is the manifold $P(p,q)$ with opposite orientation. The fundamental group of $P(p,q)$ has presentation
\begin{equation}\label{eq:pi1}
\pi_1(P(p,q)) = \langle x,y|xyx^{-1}=y^{-1},x^{2|q|}=y^p\rangle.
\end{equation}
The center of this group is a cyclic group of order $2|q|$ generated by $x^2$. 
It follows from (\ref{eq:pi1}) that $H_1(P(p,q))$ is cyclic if and only if $p$ is odd. Hence if $P(p,q)$ is obtained by surgery on a knot in $S^3$, $p$ must be odd. 

\begin{lemma}\label{lem:Nonhyp}
Suppose that $P(p,q)$ is obtained by $4|q|$--surgery on a knot $K\subset S^3$. If $K$ is a torus knot, then $(p,q)=(2k+1,k)$ or $(2k+1,-k-1)$ for some $k>0$, and $K$ is $T(2k+1,2)$. If $K$ is a satellite knot, then either $(p,q)=(2k+1,9k+4)$  for some $k>0$, and $K$ is the $(12k+5,3)$--cable of $T(2k+1,2)$, 
or $(p,q)=(2k+1,-9k-5)$ for some $k>0$, and $K$ is the $(12k+7,3)$--cable of $T(2k+1,2)$.
\end{lemma}
\begin{proof}
When $K$ is a torus knot, it follows from \cite{Moser1971} that $K=T(2k+1,2)$ and $|q|=k$ or $k+1$. Since the indices of the singular fibers are $2,2,2k+1$, $p=2k+1$. 
To determine the sign of $q$, we can use Lemma~\ref{lem:EulerPos}. Note that the slope of the Seifert fiber of the complement of $K$ is $4k+2$. If $|q|=k$, since $0<4k<4k+2$, the orbifold Euler number of the resulting manifold is positive. If $|q|=k+1$, since $0<4k+2<4k+4$, the orbifold Euler number of the resulting manifold is negative.

When $K$ is a satellite knot, $K$ must be a cable of a torus knot~\cite[Corollary~1.4]{Boyer1996}, then we can use the classification of finite surgeries on such knots in~\cite[Theorem~7]{Bleiler1996}. So $K$ is the $(12k+6\mp1,3)$--cable of $T(2k+1,2)$, $|q|=9k+4$ or $9k+5$. When $|q|=9k+4$, the $(36k+16)$--surgery on the $(12k+5,3)$--cable of $T(2k+1,2)$ is the same as the $\frac{36k+16}9$--surgery on $T(2k+1,2)$. 
Since the resulting manifold has a singular fiber of index $2k+1$, $p=2k+1$. To determine the sign of $q$, we can use Lemma~\ref{lem:EulerPos} as in the last paragraph. Similarly, we can deal with the case  $|q|=9k+5$.
\end{proof}

\begin{rmk}
In \cite{Ni2016}, it is proved that if $P(2k+1,k)$ or $P(2k+1,-k-1)$ can be obtained by positive surgery on a knot $K\subset S^3$, then $K$ must be $T(2k+1,2)$.
\end{rmk}


\begin{figure}[t]
\begin{picture}(340,100)
\put(90,0){\scalebox{0.7}{\includegraphics*
{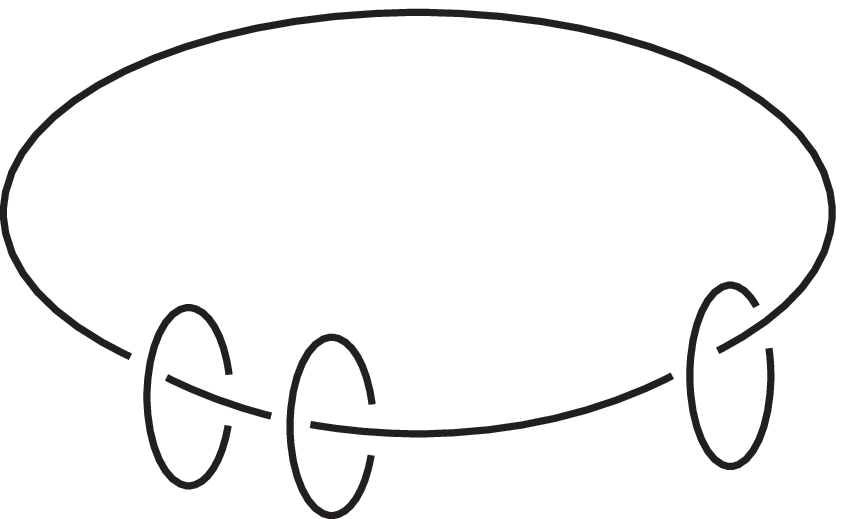}}}

\put(84,90){$-a_0$}

\put(102,5){$-2$}

\put(165,0){$-2$}

\put(215,5){$-\frac{p}{q'}$}

\end{picture}
\caption{\label{fig:SFS2}A surgery diagram of $P(p,q)$}
\end{figure}

One key step in the proof of Theorem~\ref{thm:Classification} is that every prism manifold $P(p,q)$ with $q<0$ bounds a negative definite four-manifold. 
Writing 
\[
k=\left\lfloor\frac{q}p\right\rfloor<0,\quad q'=q-kp>0,
\]
then $P(p,q)$ has an equivalent Seifert fibration with Seifert invariants
\[(-1+k;(2,1),(2,1),(p,q')).\]
Correspondingly, we get a surgery diagram for $P(p,q)$ as in Figure~\ref{fig:SFS2}, where $a_0=1-k\ge2$. Having $p/q'>1$, expand $p/q'$ in a continued fraction 
\begin{equation}\label{eq:ContFrac}
\frac p{q'}=[a_1,\dots,a_n]^-:=a_1-\frac1{a_2-\displaystyle\frac1{a_3-\displaystyle\frac1{\ddots-\displaystyle\frac1{a_n}}}},
\end{equation}
where the $a_i$ are integers satisfying $a_i\ge2$.
Equivalently, we can write
\begin{equation}\label{eq:q/pContFrac}
	\frac{-q}{p} = [a_0 - 1, a_1, \dots, a_n]^-.
\end{equation}

Let $X(p,q)$ be the four-manifold that $P(p,q)$ bounds, obtained by attaching two-handles to $D^4$ instructed by the framed link in Figure~\ref{fig:PlumbingD}. More precisely, each unknot component in Figure~\ref{fig:PlumbingD} denotes a disk bundle over $S^2$ with Euler number specified by its coefficient. The manifold $X(p,q)$ is obtained from plumbing these disk bundles together: two disk bundles are plumbed if the corresponding unknot components are linked. 

Let
\[Q_X: H_2(X) \times H_2(X) \to \mathbb{Z}\] denote the intersection pairing on $X=X(p,q)$. The second homology of $X$ has rank $n+3$, generated by elements $x_*$, $x_{**}$, $x_0$, $\dots$, $x_n$.
Note that $x_*$ and $x_{**}$ correspond to the vertices with weights $-2$ in Figure~\ref{fig:PlumbingD}.

\begin{lemma}\label{lem:QXdefinite}
$X(p,q)$ is a negative definite four-manifold.
\end{lemma}
\begin{proof}
We will show that $-Q_X$ is positive definite. Given a vector $v\in H_2(X)$, for each $i = 1,\dots, n$, it is easy to check that $-Q_X(v,v)$ is an increasing function of the $a_i$. In particular it suffices to prove the claim when each $a_i$ satisfies $a_i = 2$. Proceeding by induction on $b_2(X) = n+3$ with $n\ge 1$, we get that
		\[\mathrm{det}(-Q_X) = 4.\]
Since all principal minors are positive by induction, $-Q_X$ is positive definite.
\end{proof}

\begin{figure}[t]
\begin{picture}(340,170)
\put(90,0){\scalebox{0.7}{\includegraphics*
{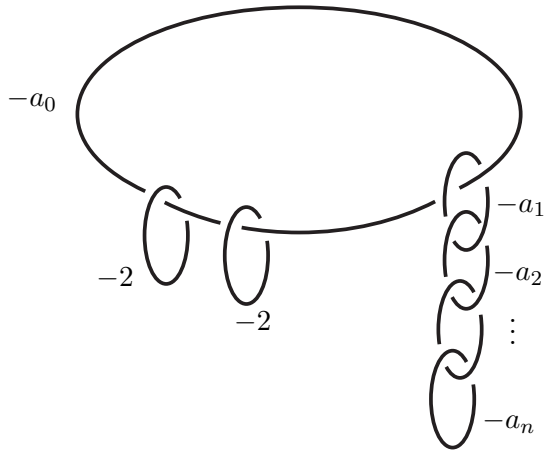}}}

\put(64,130){$-a_0$}

\put(98,62){$-2$}

\put(150,45){$-2$}

\put(249,90){$-a_1$}

\put(248,64){$-a_2$}

\put(254,40){$\vdots$}

\put(244,8){$-a_n$}
\end{picture}
\caption{\label{fig:PlumbingD}An integral surgery diagram of $P(p,q)$}
\end{figure}

\begin{figure}[t]
\begin{center}
\includegraphics[scale=.4]{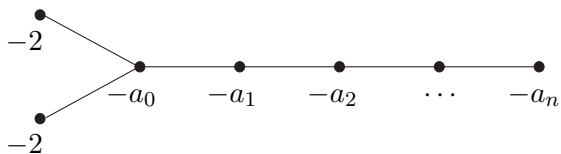}
\put(-204,28){$-2$}
\put(-204,-11){$-2$}
\put(-166,8.5){$-a_0$}
\put(-128,8.5){$-a_1$}
\put(-90,8.5){$-a_2$}
\put(-46,8.5){$\cdots$}
\put(-14.135,8.5){$-a_n$}
\caption{\label{fig:Tree}A negative definite plumbing diagram of $P(p,q)$}
\end{center}
\end{figure}

%
	
	

	
	


\subsection{The realization problem: From correction terms to lattice theory}

In what follows, we will present the methodology we apply to prove Theorem~\ref{thm:Classification}. One main ingredient is the {\it correction terms} in Heegaard Floer homology.

In \cite{OSzAbGr}, Ozsv\'ath and Szab\'o defined the correction term $d(Y, \mathfrak t)$ that associates a rational number to an oriented rational homology sphere $Y$ equipped with a Spin$^c$ structure $\mathfrak t$. They showed that this invariant obeys the relation 
\[
d(-Y,\mathfrak t) = -d(Y, \mathfrak t),
\]
where $-Y$ is the manifold $Y$ with the reversed orientation. If $Y$ is boundary of a negative definite four-manifold $X$, then
\begin{equation}\label{eq:CorrBound}
 c_1(\mathfrak s)^2 + b_2(X)\le 4d(Y, \mathfrak t),
\end{equation}
for any $\mathfrak s \in \text{Spin}^c(X)$ that extends $\mathfrak t \in \text{Spin}^c(Y)$.

\begin{definition}
A smooth, compact, negative definite $4$--manifold $X$ is {\it sharp} if for every $\mathfrak t \in \text{Spin}^c(Y)$, there exists some $\mathfrak s\in \text{Spin}^c(X)$ extending $\mathfrak t$ such that the equality is realized in Equation (\ref{eq:CorrBound}).
\end{definition}

The manifold $X=X(p,q)$ is an example of a sharp four--manifold. In order to prove this, it will be profitable to view the plumbing diagram of $X$, depicted in Figure~\ref{fig:Tree}, as a {\it weighted graph}, that is a graph equipped with an integer-valued function $m$ on its vertices. 
Given a weighted graph $G$, let $V$ be the abelian group freely generated by the vertices of $G$. Define a quadratic form $$Q_G: V\otimes V\to\mathbb Z$$ as follows. For each vertex $v$, $Q_G(v\otimes v)=m(v)$; for 
each pair of distinct vertices $v$ and $w$, $Q_G(v\otimes w)$ is $1$ if $v$ and $w$ are connected by an edge, and $0$ otherwise.

\begin{definition}
A weighted graph $G$ is said to be a {\it negative definite graph} if: 

$\bullet$ $G$ is a disjoint union of trees, and

$\bullet$ the quadratic form associated to $G$ is negative definite.

The {\it degree} of a vertex $v$, denoted $d(v)$, is the number of edges incident to $v$. A vertex $v$ is said to be a {\it bad vertex} of the weighted graph
if $$m(v)>-d(v).$$
\end{definition}
Given a weighted graph $G$, we can get a four-manifold $X_G$ obtained from the plumbing construction instructed by $G$. In~\cite{OSzPlumbed}, Ozsv\'ath and Szab\'o showed that if $G$ is a negative definite weighted graph with at most two bad vertices, then $X_G$ is sharp. In summary, using Lemma~\ref{lem:QXdefinite},

\begin{lemma}\label{lem:sharp} X(p,q) is a sharp four-manifold.
\end{lemma}

We end this subsection by presenting how the integer surgery realization problem for prism manifolds translates to a lattice theory question.

\begin{definition}
A {\it lattice}
is a finitely generated free abelian group $L$ together with a positive definite quadratic form 
$$\langle\cdot,\cdot\rangle: L\times L\to \mathbb R.$$
The lattice is {\it integral} if the value of the quadratic form is in $\mathbb Z$. 
\end{definition}

Throughout this paper, we will only consider integral lattices.

\begin{definition}\label{def:DType}
Suppose $p>1$ and $q<0$ are a pair of relatively prime integers.
The {\it D-type lattice} $\Delta(p,q)$ is the lattice freely generated by elements 
\begin{equation}\label{eq:VertexBasis}
x_*, x_{**}, x_0, x_1, \cdots , x_n    
\end{equation}
 with inner product given by 
\begin{equation}\label{eq:InnerProduct}
    \langle x_i, x_j\rangle = \begin{cases}
    -1, & \{i,j\}\text{ is either } \{*, 0\}, \text{ or } \{**, 0\},\\
    -1, & |i-j| = 1, 0 \le i, j \le n, \\
    2, & i = j \in \{*, **\}, \\
    a_i, & 0 \le i = j \le n,\\
    0, &\text{other cases,}
    \end{cases}
\end{equation}
where the coefficients $a_0, \dots, a_n$ satisfy $a_j \ge 2$ and also (\ref{eq:q/pContFrac}).
We then call (\ref{eq:VertexBasis}) a {\it vertex basis} of $\Delta(p,q)$.
\end{definition} 

The inner product space $(H_2(X), -Q_X)$ equals $\Delta(p,q)$, where $X=X(p,q)$ is the four-manifold with Kirby diagram as in Figure~\ref{fig:PlumbingD}. See also Figure~\ref{fig:Tree}. Now suppose that $4|q|$--surgery along a knot $K\subset S^3$ produces $P(p,q)$ with $q<0$. Let $W_{4|q|}$ denote the associated two-handle cobordism from $S^3$ to $P(p,q)$, capped off with $D^4$. Form the closed, oriented, smooth, four-manifold $Z = X(p, q) \cup -W_{4|q|}$. Since $b_1(S^3_{4|q|}(K))=0$, it follows that $$b_2(W)=b_2(X)+b_2(-W_{4|q|})=n+4.$$ The $4$--manifold $-W_{4|q|}$ is negative definite. Using this together with Lemma~\ref{lem:QXdefinite} and the fact that $H_2(X)\oplus H_2(-W_{4|q|})\hookrightarrow H_2(Z)$, it follows at once that $Z$ is negative definite. Using Lemma~\ref{lem:sharp}, the following is immediate from~\cite[Theorem~3.3]{Greene2015}.

\begin{theorem}\label{changemakerlatticeembedding}
Suppose $P(p, q)$ with $q<0$ arises from positive integer surgery on a knot in $S^3$. The~D-type lattice $\Delta(p, q)$ is isomorphic to the orthogonal complement $(\sigma)^\perp$ of some changemaker vector $\sigma \in \Z^{n+4}$.
\end{theorem}
Using techniques that will be developed in the next sections in tandem with Theorem~\ref{changemakerlatticeembedding}, we will find a classification of all D-type lattices $\Delta(p,q)$ isomorphic to $(\sigma)^{\perp}$ for some changemaker vector $\sigma$ in $\mathbb Z^{n+4}$. If the corresponding prism manifold $P(p,q)$ is indeed arising from surgery on a knot $K \subset S^3$, we are able to compute the Alexander polynomial of $K$ from the values of the components of $\sigma$. Giving an algorithmic method to compute the coefficients of the polynomial occupies the rest of this section.

Let $\Sigma$ be the closed surface obtained from capping off a Seifert surface for $K$ in $W_{4|q|}$. It is straightforward to check that the class $[\Sigma]$ generates $H_2(W_{4|q|})$. It follows from Theorem~\ref{changemakerlatticeembedding} that, under the embedding $H_2(X)\oplus H_2(-W_{4|q|})\hookrightarrow H_2(Z)$, $[\Sigma]$ gets mapped to a changemaker vector $\sigma$. Let $\{e_0, e_1, \cdots, e_{n+3}\}$ be the standard orthonormal basis  for $\mathbb Z^{n+4}$, and write
\[
\sigma = \sum_{i=0}^{n+3} \sigma_i e_i.
\] 
Also, define the \textit{characteristic covectors} of $\mathbb{Z}^{n+4}$ to be 
\[
\text{Char}(\mathbb Z^{n+4})=\left \{ \left.\sum_{i=0}^{n+3}\mathfrak c_i e_i \right| \mathfrak c_i\text{ odd for all } i\right \}.
\]
We remind the reader that, writing the Alexander polynomial of $K$ as 
\begin{equation}\label{eq:AlexanderPolynomial}
\Delta_K(T)= b_0 + \sum_{i>0}b_i(T^i+T^{-i}),
\end{equation}
the $k$-th {\it torsion coefficient} of $K$ is
\[
t_k(K)= \sum_{j\ge 1}jb_{k+j},
\]
where $k\ge 0$. With the preceding notation in place, the following lemma is immediate from \cite[Lemma~2.5]{Greene2015}.
\begin{lemma}\label{lem:AlexanderComputation}
The torsion coefficients satisfy
\[
t_i(K)=
\left\{
\begin{array}{cl}
\displaystyle\min_{\mathfrak c}\frac{\mathfrak c^2-n-4}8, &\text{for each $i\in\{0,1,\dots,2|q|\}$,}\\
&\\
0,&\text{for $i>2|q|$.}
\end{array}
\right.
\]
where $\mathfrak c$ is subject to 
\[
\mathfrak c\in\mathrm{Char}(\mathbb Z^{n+4}),\quad\langle\mathfrak c,\sigma\rangle+4|q|\equiv2i\pmod{8|q|}.
\]
For $i>0$, 
\[
b_i=t_{i-1}-2t_i+t_{i+1},\quad\text{for }i>0,
\]
and
\[b_0=1-2\sum_{i>0}b_i,\]
where the $b_i$ are as in~\eqref{eq:AlexanderPolynomial}.
\end{lemma}


\section{D-type lattices}\label{sec:DLattice}

Let $L$ be a lattice.
Given $v\in L$, let $|v|=\langle v, v\rangle$ be the {\it norm} of $v$.
Following Greene \cite{greene:LSRP}, say that an element $\ell \in L$ is {\it reducible} if $\ell = x + y$ for some nonzero $x, y \in L$ with $\langle x, y \rangle \ge 0$, and {\it irreducible} otherwise. We say $\ell$ is {\it breakable} if $\ell = x + y$ with $|x|, |y| \ge 3$ and $\langle x, y \rangle = -1$, and {\it unbreakable} otherwise. 

The main goal of this section will be to characterize the irreducible and unbreakable elements of a D-type lattice $\Delta(p,q)$. Since any isomorphism of lattices must send irreducible elements to irreducible elements, and similarly unbreakables to unbreakables, this will let us constrain the form of an isomorphism between $\Delta(p,q)$ and a changemaker lattice.

The pairing graph, also introduced in \cite{greene:LSRP}, will be one of the main tools we use to study lattices:

\begin{definition}
Given a lattice $L$ and a subset $V\subset L$, the {\it pairing graph} is $\wh{G}(V) = (V, E)$, where $e = (v_i, v_j) \in E$ if $\langle v_i, v_j \rangle \neq 0$.
\end{definition}

\begin{prop} \label{basesirred}
For each $i\in\{*, **, 0, ..., n\}$, $x_i$ is an irreducible element of $\Delta(p,q)$, .
\end{prop}
\begin{proof}
Take some $w \in \Delta(p,q)$ with $w \neq 0, x_i$. Write
\begin{equation*}
    w = w_*x_* + w_{**}x_{**} + \sum_{j=0}^n w_j x_j,
\end{equation*}
and consider $\braket{w}{x_i - w}$. This is
\begin{equation*}
    a_i w_i(1 - w_i) - \sum_{j \neq i}a_j w_j^2 + (\text{terms not involving the }a_j).
\end{equation*}
Here the $a_j$ are defined in (\ref{eq:q/pContFrac}) when $j\ge0$, and $\braket{x_j}{x_j} = a_j$. We also let $a_{*}=a_{**}=2$. The above expression is non-increasing with respect to each $a_j$. Therefore, to show that $\braket{w}{x_i - w}$ is negative, it suffices to show this when all of the $a_j$ are 2. In this case, $\Delta(p,q)$ is isomorphic to the standard $D_{n+3}$ lattice, the lattice of elements of $\Z^{n+3} = \langle e_{-2}, e_{-1}, e_0, \dots, e_n \rangle$ the sum of whose coefficients is even. This isomorphism sends
\begin{align}
     x_* &\mapsto e_{-2} + e_{-1}\nonumber \\
     x_{**} &\mapsto -e_{-2} + e_{-1}\nonumber \\
     x_j &\mapsto -e_{j-1} + e_j & j\ge 0.\label{eq:Embedding}
\end{align}
If $\pm e_i \pm e_j \in \Z^{n+3}$ is written as a sum of two other elements of $\Z^{n+3}$ with nonnegative pairing, these must be $\pm e_i$ and $\pm e_j$. However, these are not in $D_{n+3}$. Therefore, the image of $x_i$ is an irreducible element of $D_{n+3}$, so $x_i$ is an irreducible element of $\Delta(p,q)$.
\end{proof}

\begin{cor}\label{indecomposable}
The lattice $\Delta(p,q)$ is indecomposable, namely, $\Delta(p,q)$ is not the direct sum of two nontrivial lattices.
\end{cor}
\begin{proof}
Suppose that $\Delta(p,q) \cong L_1 \oplus L_2$. Then each $x_i$, being irreducible, must be in either $L_1$ or $L_2$. However, any element of $L_1$ has zero pairing with any element of $L_2$. Since $x_i \sim x_{i+1}$, $\hat{G}(\{*, **, 0, ..., n\})$ is connected. This means that all of the $x_i$ are in the same part of the decomposition, and the other is trivial.
\end{proof}

This gives us some information about irreducibility in the D-type lattices, but we want something more complete. Another important class of irreducible elements is the intervals:
\begin{definition}
For $A \subset \{*, **, 0, ..., n\}$, let $[A] := \sum_{i \in A} x_i$. We say $x \in \Delta(p,q)$ is an \emph{interval} if $x = [A]$ for some $A$ where $\wh{G}(A)$ is connected and $A$ does not contain both $*$ and $**$. 
\end{definition}

\begin{prop}
Intervals are irreducible.
\end{prop}
\begin{proof}
As in the proof of Proposition~\ref{basesirred} we may assume that all of the $a_i=2$, and consider the image of an interval $x$ in $\Z^{n+3}$ under the embedding (\ref{eq:Embedding}). (Note that the reduction to the case $a_i = 2$ only works because all the coefficients $c_i$ of $x= \sum c_i x_i$ are $0$ or $\pm 1$). Since $x$ is an interval, the image of $x$ is again of the form $\pm e_i \pm e_j$, which is irreducible as in Proposition~\ref{basesirred}.
\end{proof}

For linear lattices, this is essentially the end of the story - Greene proves that every irreducible is either an interval or the negation of an interval. The situation here is somewhat richer, but is similar in spirit: every irreducible is either an interval, or can be obtained from one by applying some involution of the lattice. To prove this, we first need to know a few involutions of $\Delta(p,q)$.

\begin{definition}\label{defn:Reflection}
Let $m$ be the smallest index for which $a_m \ge 3$. For any $0 \le j \le m$, let $y_j = x_* + x_{**} + 2\sum_{0 \le i < j} x_i$, and let $\tau_j: \Delta(p,q) \to \Delta(p,q)$ be the reflection in the subspace spanned by $y_0, \dots, y_j$. More explicitly, $\tau_j$ acts as follows:
\begin{equation*}
    \tau_j(x_i) = \begin{cases}
    x_{**} & i = * \\
    x_* & i = ** \\
    x_i & 0 \le i < j \\
    -x_*-x_{**} - 2x_0 - \cdots - 2x_{j-1} - x_j & i = j \\
    -x_i & i > j
    \end{cases}
\end{equation*}
\end{definition}

\begin{lemma} \label{taubasic}
Each $\tau_j$ satisfies $\braket{\tau_j(x)}{\tau_j(y)} = \braket{x}{y}$ and $\tau_j^2(x) = x$. Furthermore, for any $v, w \in \Delta(p,q)$, $\braket{v}{w} \equiv \braket{\tau_j(v)}{w} \pmod2$.
\end{lemma}
\begin{proof}
Each $\tau_j$ is a reflection, which gives the first two properties. For the last property, it suffices to check that $\braket{\tau_j(v_i)}{v_k} \equiv \braket{v_i}{v_k}$ for all $i, j, k$, which is entirely straightforward.
\end{proof}

The following special case of Lemma~\ref{taubasic} will be used very heavily in what follows:

\begin{lemma} \label{parity}
If $x \in \Delta(p,q)$, then $\braket{x}{x_*} \equiv \braket{x}{x_{**}} \pmod2$.
\end{lemma}
\begin{proof}
This follows from 
Lemma~\ref{taubasic} since $x_{**}=\tau_0(x_*)$.
\end{proof}

Since the $\tau_j$ preserves the pairings on $\Delta(p,q)$, $\tau_j(x)$ is irreducible whenever $x$ is. This gives some non-interval examples of irreducible elements of $\Delta(p,q)$. For example,
\begin{equation*}
    x_* + x_{**} + x_0 + \dots + x_i = -\tau_0(x_0 + \dots + x_i)
\end{equation*}
and (as long as $m \ge 2$), 
\begin{equation*}
    x_* + x_{**} + 2x_0 + 2x_1 + x_2 + \dots + x_i = -\tau_2(x_2 + \dots + x_i)
\end{equation*}
are both irreducible. However, these are essentially the only examples:

\begin{prop} \label{irred}
If $x \in \Delta(p,q)$ is any irreducible element, either $x$ or $-x$ is an interval, or there is some $j$ so that either $\tau_j(x)$ or $-\tau_j(x)$ is an interval.
\end{prop}
\begin{proof}
Suppose $x = \sum c_i x_i$ is irreducible. Replacing $x$ by $-x$ if necessary, we can assume that $c_0 \ge 0$ and there is at least one $c_i>0$. Let $A\subset \{*,**,0,\dots, n\}$ be the set of indices $i$ with $c_i < 0$, and let $y = \sum_{i \in A} c_i x_i$. Clearly, $\braket{y}{x-y}$ is a sum of terms of the form $-c_ic_{i'}$ with $c_i<0$, $c_{i'}\ge0$, so is nonnegative. Since $x$ is irreducible, then, $y = 0$, so actually all of the $c_i$ are nonnegative.

Now, let $B \subset \{*,**,0,\dots, n\}$ be the set of indices $i$ with $c_i \neq 0$. The pairing graph $\wh{G}(B)$ must be connected, since otherwise $x$ could be written as a sum of two nonzero elements that pair to zero.

If $c_i \in \{0,1\}$ for all $i$, $x = [B]$. Since $\wh{G}(B)$ is connected, either $x$ is  already an interval, or $B = \{*,**,0,\dots,j\}$ for some $j$, but then $-\tau_0(x) = x_0+\cdots+x_j$ is an interval.

If $c_i>1$ for some $i$, let $z = [B]$, and consider the pairing $\braket{z}{x-z}$. If $0 \not \in B$, then let $j_0 = \min B$ and $j_1 = \max B$, so
\begin{equation*}
    \braket{z}{x-z} = \left(\sum_{j_0 \le i \le j_1} (a_i - 2)(c_i - 1)\right) + c_{j_0} + c_{j_1} - 2.
\end{equation*}
Since $c_i \ge 1$ for each $i \in B$, and $a_i \ge 2$ for each $i$, this is nonnegative, contradicting the irreducibility of $x$. Therefore, $0 \in B$, so $c_0 \ge 1$. Since $\wh{G}(B)$ is connected, then, $c_i \ge 1$ for all $0 \le i \le \max B$. 

Now, let $j$ be the largest index with $c_j>1$. Suppose there were some largest $k$, $0 \le k < j$, with $c_k < c_j$, and let $w = x_{k+1} + \cdots + x_j$. Then
\begin{equation*}
    \braket{w}{x-w} = \left(\sum_{k+1 \le i \le j} (a_i - 2)(c_i - 1)\right) + \left(c_{k+1} - c_k\right) + \left(c_j - c_{j+1}\right)-2.
\end{equation*}
Since all of the $c_i$ are at least $1$, each of these terms is nonnegative, so this cannot happen. Therefore, $c_i\ge c_j > 1$ for any $0\le i<j$.

Finally, let $w_0 = x_0 + \cdots + x_j$, and let $w_1 = x_* + x_{**} + w_0$. Consider:
\begin{align*}
    \braket{w_0}{x-w_0} &= \left(\sum_{0 \le i \le j}(a_i - 2)(c_i - 1)\right) - \left(c_* + c_{**} - c_0\right) + \left(c_j  - 2\right) - c_{j+1}, \\
    \braket{w_1}{x-w_1} &= \left(\sum_{0 \le i \le j}(a_i - 2)(c_i - 1)\right) + \left(c_* + c_{**} - c_0\right) + \left(c_j - 2\right) - c_{j+1}.
\end{align*}
The first term of each of these sums is nonnegative and, for at least one of these sums, the second term is nonnegative. Moreover, $c_j \ge 2$, so the third term is also nonnegative. The last term $c_{j+1}$ is $0$ or $1$. Therefore, for both of these be negative, the first three terms must be zero, and the last one must be negative. Since $c_i > 1$ for $i \le j$, the first term can only be zero when $a_i = 2$ for $0 \le i \le j$, so $j < m$, where $m$ is as in Definition~\ref{defn:Reflection}. For the remaining terms to be zero or negative, $c_j = 2$, $c_0 = c_* + c_{**}$, and $c_{j+1} = 1$.

Since $j < m$, we can consider $-\tau_{j+1}(x) = \sum c_i' x_i$. For $i > j$, $c_i' = c_i \le 1$, and $c_j' = 2c_{j+1} - c_j = 0$. Note that $c'_{j+1}=c_{j+1}=1\ne0$. So $c_i'$ vanishes for $i < j$ as well, since $-\tau_{j+1}(x)$ is still irreducible. Therefore, $-\tau_{j+1}(x) = \sum_{i > j} c_i x_i$, which is an interval.
\end{proof}


We will also want to know which irreducible elements of $\Delta(p,q)$ are breakable. Since negation and the $\tau_j$ preserve breakability, it will suffice to know this for intervals:
\begin{lemma}\label{lem:TwoIndBr}
    An interval $x = [A]$ is breakable if there are at least two indices $i,k \in A$ with $a_i,a_k \ge 3$.
\end{lemma}
\begin{proof}
Suppose that $i, k \in A$ with $a_i, a_k \ge 3$. Fix some index $J$ with $i \le J < k$, and let $T_0 = \{j \in A | j \le J\}$ and $T_1 = \{j \in A | J < j\}$. Since $i \in T_0$ and $k \in T_1$, $|[T_0]| \ge 3$ and $|[T_1]| \ge 3$. Also, since $x$ is an interval, $j \in A$ for all $i \le j \le k$, so $J \in T_0$ and $J+1 \in T_1$. Therefore, $\braket{[T_0]}{[T_1]} = -1$, and $x$ is breakable.
\end{proof}

\begin{definition}\label{defn:z_j}
When $[A_j]$ is unbreakable and has norm at least $3$, let $z_j \in A_j$ be the unique element with $|z_j| \ge 3$.
\end{definition}

Finally, let us determine when two D-type lattices are isomorphic. 
\begin{prop}\label{pp}
If $\Delta(p, q) \cong \Delta(p', q')$, then $p = p'$ and $q = q'$. 
\end{prop}

\begin{proof}
Suppose $L$ is a lattice isomorphic to $\Delta(p,q)$ for some $p,q$. To recover $p$ and $q$ from $L$, it suffices to recover the ordered sequence of norms $(|x_0|,|x_1|,\dots,|x_n|)$ of the vertex basis, or equivalently to recover the weighted tree $\mathcal T$ corresponding to $\Delta(p,q)$. As in the proof of \cite[Proposition~3.6]{greene:LSRP}, we will first recover the vertices with weight at least $3$, and then fill in the vertices of weight $2$.  Let $R \subset L$ be the sublattice generated by the vectors of norm $2$ in $L$. It follows from Proposition~\ref{irred} that $R$ is generated by vertices with weight $2$.

Let $V_0$ be the set of irreducible, unbreakable elements of $L$ of norm at least $3$, and let $V$ be the quotient of $V_0$ by the equivalence $v \sim u$ whenever either $v - u \in R$ or $v + u \in R$. By Proposition~\ref{irred} and Lemma~\ref{lem:TwoIndBr}, each irreducible, unbreakable element of $L$ corresponds (up to sign and applications of the $\tau_j$) to an interval containing a unique vertex of weight at least $3$. We claim that the set $V$ is in bijection with the vertices in $\mathcal T$ of weight at least $3$. In fact,
if $v-u$ or $v+u$ can be written as a sum of vectors of norm $2$, the corresponding high weight vertices must be the same: if we write $v = \sum v_ix_i$ and $\pm \tau_j(v) = \sum v_i'x_i$ (where $\pm\tau_j(v)$ is an interval), then the coefficient $v'_i$ is nonzero for exactly one $i$ with $|x_i| \ge 3$. Since applying any $\tau_j$ does not change the coefficient $v_i$ when $|x_i| \ge 3$ except by a sign, the original coefficient $v_i$ is also nonzero (in fact, $\pm 1$) for exactly one $i$ with $|x_i| \ge 3$. Similarly, there is exactly one $j$ with $|x_j| \ge 3$ with $u_j$ nonzero and equal to $\pm 1$, where $u = \sum u_jx_j$. Therefore, one of $v-u$ or $v+u$ can be written as a sum of vertices of weight $2$ if and only if the high weight vertices are the same. This finishes the proof of the claim.
Also, if $v \in L$ is a representative of some class in $V$, then $|v|$ is the weight of the corresponding vertex.

Let $W$ be the set of indecomposable sublattices of $R$. Each element of $W$ corresponds to a connected subgraph of $\mathcal T$, all of whose vertices have weight $2$. As a lattice, each element $w \in W$ is isomorphic to either the root lattice $A_{n_w}$ for some $n_w$, or to $D_{n_w}$ if the corresponding subgraph contains all of $x_*,x_{**},x_0,$ and $x_1$. This last case happens for at most one $w \in W$. Now, form a bipartite graph $\mathcal B$ as follows: the vertex set is $V \cup W$, and there is an edge between $v \in V$ and $w \in W$ if, for some representative $\tilde{v} \in L$ of $v$ and some element $\tilde{w} \in w$, $\braket{\tilde{v}}{\tilde{w}} \neq 0$. This happens if and only if there is an edge in $\mathcal T$ between the vertex corresponding to $v$ and some vertex of the subgraph corresponding to $w$. In $\mathcal B$, each $w \in W$ neighbors at most two vertices $v \in V$, or at most one if the corresponding sublattice is isomorphic to $D_{n_w}$. 

We say an edge connecting $v\in V$ and $w\in W$ is special, if $w$ is isomorphic to $A_3$, and there exists a representative $\tilde{v}$ of $v$ such that there exist exactly $4$ vectors $\tilde w\in w$ satisfying $|\tilde w|=2$ and $\braket{\tilde v}{\tilde w}=-1$. It is easy to check that an edge is special if and only if $a_0=2$, $a_1>2$, $w$ is generated by $x_*,x_{**},x_0$, and $v$ corresponds to $x_1$.

The graph $\mathcal B$ contains all of the information about how the blocks of vertices of weight $2$ fit together with the other vertices in $\mathcal T$. Since the part of $\mathcal T$ corresponding to one of those blocks must be the Dynkin diagram of the corresponding sublattice, reconstructing $\mathcal T$ is straightforward. The vertex set of $\mathcal T$ is a copy of $V$, together with a copy of $V_w$ for each $w \in W$, where $V_w$ is a set containing $\operatorname{rank}(w)$ vertices. The edges in $\mathcal T$ are constructed as follows.
If the sublattice $w$ is isomorphic to $A_{n_w}$, connect the vertices in $V_w$ together in a path $P_w$. For each non-special edge in $\mathcal B$ connecting $v$ and $w$, connect $v$ to one of the ends of $P_w$. For each $w$, at most two vertices in $V$ need to be connected in this way, so they can always be connected so that no end of $P_w$ has valency $3$ in $\mathcal T$. If there is a special edge between $v$ and $w$, connect the middle vertex of $P_w$ to $v$.
If the sublattice $w$ is isomorphic to $D_{n_w}$, connect $n_w - 2$ of the vertices in $V_w$ into a path, connect the remaining two to one end of the path, and if there is a $v \in V$ neighboring $w$ in $\mathcal B$, connect it to the other end. 
Finally, for $v,v' \in V$, put an edge between them if there are representative $\tilde{v}, \tilde{v}'$ with $\braket{\tilde{v}}{\tilde{v}'} \neq 0$ and $v,v'$ do not have a common neighbor in $\mathcal B$. For the weights, put a $2$ on any vertex in $V_w$, and $|\tilde{v}|$ on a vertex $v \in V$, where $\tilde{v}$ is a representative of the class $v$.
\end{proof}

\section{Changemaker lattices}\label{sec:Changemaker}

According to Theorem~\ref{changemakerlatticeembedding}, whenever $P(p, q)$ comes from positive integer surgery on a knot, $\Delta(p,q)$ is isomorphic to the orthogonal complement $(\sigma)^\perp$ for some changemaker vector $\sigma \in \Z^{n+4}$. A lattice is called a {\it changemaker lattice} if it is isomorphic to the orthogonal complement of a changemaker vector.
In this section, we will prove some basic structural results about D-type lattices which are also changemaker lattices.

Write $(e_0,e_1, \dots, e_{n+3})$ for the orthonormal basis of $\Z^{n+4}$, and write $\sigma = \sum_i \sigma_i e_i$. 
Since $\Delta(p,q)$ is indecomposable (Corollary~\ref{indecomposable}), $\sigma_0\ne0$, otherwise $(\sigma)^\perp$ would have a direct summand $\mathbb Z$. So $\sigma_0=1$.

We will need several results from \cite[Section~3]{greene:LSRP} about changermaker lattices:

\begin{definition}[Definition~3.11 in \cite{greene:LSRP}] \label{stbasis}
The {\it standard basis} of $(\sigma)^\perp$ is the collection $S = \{v_1, \dots, v_{n+3}\}$, where
\begin{equation*}
    v_j = \left(2e_0 + \sum_{i = 1}^{j - 1} e_i\right) - e_j
\end{equation*}
whenever $\sigma_j = 1 + \sigma_0 + \cdots + \sigma_{j-1}$, and
\begin{equation*}
    v_j = \left(\sum_{i \in A} e_i\right) - e_j
\end{equation*}
whenever $\sigma_j = \sum_{i \in A} \sigma_i$, with $A \subset \{0, \dots, j-1\}$ chosen to maximize the quantity $\sum_{i \in A} 2^i$.  A vector $v_j \in S$ is called \emph{tight} in the first case, \emph{just right} in the second case as long as $i+1 \in A$ whenever $i < j-1$ and $i \in A$, and \emph{gappy} if there is some index $i$ with $i \in A$, $i < j-1$, and $i+1 \not \in A$. In this situation, call $i$ a \emph{gappy~index} for $v_j$.
\end{definition}

\begin{definition}
For $v \in \Z^{n+4}$, $\supp v = \{i | \braket{e_i}{v} \neq 0\}$ and $\supp^+ v = \{i | \braket{e_i}{v} > 0\}$.
\end{definition}

\begin{lemma}[Lemma~3.12~(3) in \cite{greene:LSRP}] \label{gappy3}
If $|v_{k+1}|=2$, then $k$ is not a gappy index for any $v_j \in S$.
\end{lemma}
\begin{proof}
This follows from the maximality of the set $A$: if this did not hold, we could remove $k$ from $A$ and add $k+1$, increasing the sum $\sum_{i \in A} 2^i$ but leaving the sum $\sum_{i \in A} \sigma_i$ unchanged.
\end{proof}

\begin{lemma}[Lemma~3.13 in \cite{greene:LSRP}] \label{lem:irred}
Each $v_j \in S$ is irreducible.
\end{lemma}

\begin{lemma}[Lemma~3.15 in \cite{greene:LSRP}]\label{lem:BrIsTight}
If $v_j \in S$ is breakable, it is tight.
\end{lemma}

\begin{lemma}[Lemma~3.14~(2) in \cite{greene:LSRP}]\label{lem:SumIrr}
If $v_t \in S$ is tight, $j > t$, and $v_j = e_t + e_{j-1} - e_j$, then $v_t + v_j$ is irreducible.
\end{lemma}

\begin{lemma}\label{newirreducibility}
If $v_t \in S$ is tight, $j > t$, and $v_j = e_0+...+e_{t-1}+(\sum_{i\in A}e_i) - e_j$ for some $A \subseteq \{t+1,...,j-1\}$, then $v_j - v_t$ is irreducible.
\end{lemma}
\begin{proof}
For contradiction, suppose there exists $x,y$ with
\[v_j - v_t  = -e_0+e_t+(\sum_{i\in A}e_i) - e_j= x + y\] such that $\braket{x}{y} \geq 0$. Write $x = \sum x_ie_i$ and $y = \sum y_ie_i$. Since $\forall i$ $|x_i+y_i|\leq 1$, $x_iy_i \leq 0$ for all $i$. Since $\braket{x}{y}=\sum x_iy_i\geq0$, we must have $x_iy_i=0$ for each $i$. Observe only $e_0$ and $e_j$ have negative coefficients in $v_j-v_t$. If $x_0 = x_j = 0$ (resp. $y_0 = y_j = 0$), then since $x$ (resp. $y$) is in $\langle \sigma\rangle ^\perp$ and has non-negative coefficients, $x = 0$ (resp. $y = 0$). If $x_0=0$ and $x_j = -1$, since $\sum \sigma_iy_i = 0$, $y = e_s-e_0$ for some $0 < s < t$. Since $y=v_j-v_t-x$, $x_s=-1$ and $x_sy_s \ne 0$, a contradiction. If $x_0=-1$ and $x_j = 0$, we get a contradiction just as in the previous case.
\end{proof}


If we have a lattice $L$ isomorphic to both a D-type lattice $\Delta(p,q)$ and a changemaker lattice $(\sigma)^\perp$, it has both the vertex basis $x_*, x_{**}, x_0, \dots, x_n$ and the standard basis $v_1, \dots, v_{n+3}$. Since each of the $v_i$ is irreducible, we can use Proposition~\ref{irred} to constrain its expression in terms of the $x_i$. In some cases, we can say more:

\begin{lemma} \label{modify}
If $\Delta(p,q)$ is isomorphic to $(\sigma)^\perp$ for some changemaker $\sigma \in \Z^{n+4}$, then there is an isomorphism $\phi: \Delta(p,q) \xrightarrow{\sim} (\sigma)^\perp$ that sends $x_{*}$ to $v_1$ and $x_{**}$ to $v_f$, for some index $f$. Furthermore, if $a_0 = 2$, then $f = 3$ and we can also choose $\phi$ to send $x_0$ to $v_2$.
\end{lemma}
\begin{proof}
Let $\phi_0:\Delta(p,q) \to (\sigma)^\perp$ be an isomorphism.
Both $x_*$ and $x_{**}$ have norm $2$, and $\braket{x_*}{x_{**}} = 0$. The only elements of $(\sigma)^\perp$ of norm $2$ have the form $e_i - e_j$, for indices $i$ and $j$ with $\sigma_i = \sigma_j$. Since the entries of $\sigma$ are nondecreasing, actually $\sigma_i = \sigma_k = \sigma_j$ for any $\min(i,j) \le k \le\max(i,j)$, so in particular $\sigma_i = \sigma_j = \sigma_m = \sigma_{m+1}$, for $m = \min (i,j)$. We define an automorphism $\psi$ of $\Z^{n+4}$ as follows. If $|j-i|=1$, let 
\[
\psi(e_i)=e_m,\quad \psi(e_j)=e_{m+1},\quad \psi(e_l)=e_l \text{ when }l\notin\{i,j\}.
\]
If $|j-i|>1$, $\psi(e_i)$ and $\psi(e_j)$ are as above, and
\[
\psi(e_{m+1})=e_{\max(i,j)},\quad \psi(e_l)=e_l \text{ when }l\notin\{i,j,m+1\}.
\]
Since $\psi$ fixes $\sigma$, it restricts to an automorphism $\overline{\psi}$ of $(\sigma)^\perp$. Consider $\overline{\psi}\circ\phi_0$, which sends $\phi_0^{-1}(e_i - e_j)$ to the standard basis vector $v_{m+1}$. Since $\braket{x_*}{x_{**}} = 0$, this process can be done separately for $x_*$ and $x_{**}$, making them both sent to standard basis vectors $v_e, v_f$. By precomposing with the automorphism of $\Delta(p,q)$ that switches $x_*$ and $x_{**}$, we can assume $e < f$, and then Lemma~\ref{parity} ensures that $e = 1$, because otherwise $\braket{v_{e-1}}{x_*} = -1$ while $\braket{v_{e-1}}{x_{**}} = 0$.

When $a_0 = 2$, $x_0$ has norm $2$ and pairing $-1$ with both $x_*$ and $x_{**}$, so $x_0=e_1-e_{f-1}$ or $-e_0+e_f$. It follows that $\sigma_0 = \sigma_1 = \cdots = \sigma_{f-1} = \sigma_f$. Therefore, $v_2 = e_1 - e_2$, so $\braket{v_2}{v_1} = -1$. By Lemma~\ref{parity}, $\braket{v_2}{v_f}$ is odd. This can happen only if $f = 3$, so $v_f = e_3 - e_2$, and $x_0=e_1-e_{2}$ or $-e_0+e_3$. Since $\sigma_0 = \sigma_1 = \sigma_2 = \sigma_3$, any permutation of $\{0,1,2,3\}$ induces a permutation of $\{e_0,e_1,e_2,e_3\}$, hence gives an automorphism of $(\sigma)^\perp$. If $x_0=e_1-e_{2}$, we are done. If $x_0=-e_0+e_3$, we can precompose $\phi_0$ with the automorphism exchanging $x_*$ and $x_{**}$, and postcompose $\phi_0$ with the automorphism given by the permutation $(02)(13)$.
\end{proof}

From now on, let $L = (\sigma)^\perp$ be a changemaker lattice isomorphic to some D-type lattice, and identify $x_*, x_{**}, x_0, \dots, x_n$ with elements of $L$ according to an isomorphism chosen as in Lemma~\ref{modify}.
Since $\sigma_0=1$, by Lemma~\ref{modify}, we see that 
\begin{equation}\label{eq:sigma1}
\sigma_1=\sigma_0=1.    
\end{equation}

\begin{lemma}\label{lem:j-1}
For any $v_j\in S$, we have $j-1\in\supp v_j$.
\end{lemma}
\begin{proof}
This is just the second part of \cite[Lemma~3.12~(1)]{greene:LSRP}, plus the fact that $v_1=e_0-e_1$.
\end{proof}


\begin{lemma}\label{a1whena0=2}
If $a_0 = 2$, $a_1 \ge 3$.
\end{lemma}
\begin{proof}
Suppose $a_0 = 2$. Note that $x_1$ satisfies $\braket{x_1}{x_0} = -1$ and $\braket{x_1}{x_*} = \braket{x_1}{x_{**}} = 0$. Therefore, if we write $x_1 = \sum_{i = 0}^{n+3} c_ie_i$, $c_0 = c_1 = c_2 - 1 = c_3 - 1$. Also, $x_1$ is not in the span of $x_0, x_*$, and $x_{**}$, so there is some $j > 3$ with $c_j \neq 0$. Therefore, at least three of the $c_i$ are nonzero, so $a_1 = |x_1| \ge 3$.
\end{proof}

In particular, this means that $m \le 1$ in the notation of Definition~\ref{defn:Reflection}. So
we only ever need to use the involution $\tau_j$ for $j \in \{0,1\}$ to make irreducible elements into intervals.

\begin{lemma} \label{pairequal}
If $v_j$ is a standard basis vector for $j \neq 1, f$ and $v_j$ is not tight, then $\braket{v_j}{v_1} = \braket{v_j}{v_f} \in \{-1,0\}$. If $v_j$ is tight, $j = f-1$ and $\braket{v_j}{v_1} = -\braket{v_j}{v_f}=1$.
\end{lemma}
\begin{proof}
Using Lemma~\ref{gappy3}, $\supp v_j \cap \{0,1\}\ne\{0\}$, so $\braket{v_j}{v_1}\in\{0,-1\}$ unless $v_j$ is tight. Similarly, either $j \neq f-1$ and $\supp v_j \cap \{f-1,f\}\ne\{f-1\}$, or $j = f-1$ and $\supp v_j \cap \{f-1,f\} = \{f-1\}$. In either case, $\braket{v_j}{v_f}\in\{0,-1\}$. By Lemma~\ref{parity}, these two values have the same parity, so they must be equal if $v_j$ is not tight. If $v_j$ is tight, $\braket{v_j}{v_1}$ is $1$, so $\braket{v_j}{v_f}$ is $-1$. Therefore, $j = f-1$.
\end{proof}

\begin{cor}\label{cor:OneTight}
There is at most one tight vector in $S$. If $a_0 = 2$, there is no tight vector.
\end{cor}


\begin{prop}\label{prop:OnlyTau}
If $a_0 \ge 3$, for any standard basis vector $v_j$, either $v_j$ is an interval or $\tau_0(v_j)$ is. If $a_0 = 2$, the same conclusion holds with $\tau_1$ instead of $\tau_0$. Furthermore, this interval contains $*$ or $**$ if and only if $j$ is $1$ or $f$ or $v_j$ is tight. If $a_0 = 2$, then this interval contains $0$ if and only if $j = 2$.
\end{prop}
\begin{proof}
First, suppose that $a_0 \ge 3$, so $m = 0$ and Proposition~\ref{irred} says that one of $v_j, -v_j, \tau_0(v_j)$, or $-\tau_0(v_j)$ is an interval. It remains only to show that we do not ever need to use a negation. If $j \in \{1,f\}$, $v_j \in \{x_*, x_{**}\}$ is an interval. Therefore, suppose $j \neq 1,f$. If $v_j$ is tight, by Lemma~\ref{pairequal} $\braket{v_j}{y_0}=0$, where $y_0=x_*+x_{**}$ is defined in Definition~\ref{defn:Reflection}. Since $\tau_0$ is the reflection in the line spanned by $y_0$, $\tau_0(v_j) = -v_j$ and $-\tau_0(v_j) = v_j$, and we are done because we do not need to use the negation or $-\tau_0$. 
If $v_j$ is not tight, Lemma~\ref{pairequal} says that $\braket{v_j}{x_*} =\braket{v_j}{x_{**}} \in\{ 0,-1\}$. If $\braket{v_j}{x_*} = 0$, we get $\braket{v_j}{y_0} = 0$. We are done by the same argument as before. If $\braket{v_j}{x_*} = \braket{v_j}{x_{**}} = -1$, we have
\begin{equation*}
    \braket{\tau_0(v_j)}{x_*} = \braket{\tau_0(v_j)}{\tau_0(x_{**})} = \braket{v_j}{x_{**}} = -1
\end{equation*}
and similarly $\braket{\tau_0(v_j)}{x_{**}} = -1$. Therefore, if $-v_j$ or $-\tau_0(v_j)$ were an interval $[A]$, it would satisfy
\begin{equation*}
    \braket{[A]}{x_*} = \braket{[A]}{x_{**}} = 1.
\end{equation*}
This can only happen if $*,**,$ and $0$ are all in $A$, so $[A]$ is not an interval.

If $a_0 = 2$, $(v_1,v_2,v_3) = (x_*,x_0,x_{**})$, all of which are intervals, so we can assume $j > 3$. By Corollary~\ref{cor:OneTight}, $v_j$ is not tight, so Lemma~\ref{pairequal} implies that $\braket{v_j}{v_1} = \braket{v_j}{v_3}$. Since $|v_1| = |v_2| = |v_3| = 2$, Lemma~\ref{gappy3} implies that $\supp v_j \cap \{0,1,2,3\}$ is either $\{0,1,2,3\}$, $\{2,3\}$, or $\emptyset$. In any case, $\braket{v_j}{v_1} = \braket{v_j}{v_3} = 0$. Therefore, $\tau_0(v_j) = -v_j$. This means that, for $j > 3$, one of $v_j, -v_j, \tau_1(v_j),$ or $-\tau_1(v_j)$ must be an interval. We can assert that 0 is never in this interval, since $v_j$ has zero pairing with $v_1 = x_*$ and $v_3 = x_{**}$. If 1 is not in this interval, then also $\braket{v_j}{x_0} = 0$, so $\braket{v_j}{y_1} = 0$. Since $\tau_1$ is the reflection in the plane spanned by $y_0,y_1$, $\tau_1(v_j) = -v_j$, and we are done. It remains only to show that if $[A]$ is an interval containing $1$ but not $0$, then neither $-[A]$ nor $-\tau_1([A])$ can be a standard basis vector, which holds because in this case
\begin{equation*}
    \braket{-[A]}{x_0} = \braket{-\tau_1([A])}{x_0} = 1,
\end{equation*}
and by Lemma~\ref{gappy3} for any non-tight standard basis vector this pairing is either 0 or $-1$.
\end{proof}

From now on, let $\tau$ denote $\tau_0$ if $a_0 \ge 3$, or $\tau_1$ if $a_0 = 2$. For each standard basis vector $v_j$, let $[A_j]$ be an interval for which either $v_j = [A_j]$ or $v_j = \tau([A_j])$, and let $\epsilon_j = 1$ if $v_j = [A_j]$, and $-1$ otherwise. When $a_0\ge3$, for all $j$ except $1$, $f$, $\epsilon_j$ is uniquely determined; when $a_0 = 2$, for $j\ne1,2,3$, $\epsilon_j$ is uniquely determined. However, $\epsilon_1$ and $\epsilon_f$ can be freely chosen since $\tau$ exchanges $v_1$ and $v_f$. To resolve this ambiguity, we will choose $[A_1]$ and $[A_f]$ to make $\epsilon_1 = -\epsilon_f = \epsilon_{f-1}$, and, when $a_0 = 2$, chose $\epsilon_2 = 1$. This makes $A_1 = A_f$, and, when $v_{f-1}$ is tight, $A_1 = A_f = \{*\}$ whenever $* \in A_{f-1}$, and $\{**\}$ otherwise. By Proposition~\ref{prop:OnlyTau}, this means that either $*$ or $**$ is in none of the $A_j$, which gives the following fact:

\begin{lemma} \label{lem:IntervalProd}
If $A_j$ and $A_k$ are disjoint, $\braket{[A_j]}{[A_k]}$ is either $0$ or $-1$. Otherwise, $\braket{[A_j]}{[A_k]}$ is either $|[A_j\cap A_k]| - 1$ or $|[A_j \cap A_k]| - 2$, where $|[A_j \cap A_k]|$ is either $|[A_j]|$ or $2$ if $[A_j]$ is unbreakable.
\end{lemma}
\begin{proof}
For simplicity in this proof, call whichever of $*$ or $**$ occurs in some of the $A_j$ $-1$. If $A_j$ and $A_k$ are disjoint, either there is no index $i$ with $x_i \in A_j$ and $x_{i \pm 1} \in A_k$, in which case they pair to $0$, or there is exactly one, in which case they pair to $-1$. If $A_j$ and $A_k$ do have nonempty intersection, either $\min A_j = \min A_k$ or $\max A_j = \max A_k$, in which $\braket{[A_j]}{[A_k]}$ is $|[A_j\cap A_k]| - 1$, or this does not hold, in which case $\braket{[A_j]}{[A_k]} = |[A_j\cap A_k]| - 2$. To compute $|[A_j \cap A_k]|$, note that for any $A \subset \{-1,0,\dots,n\}$ with $\hat{G}(A)$ connected,
\begin{equation}\label{eq:IntervalNorm}
|[A]| = 2 + \sum_{i \in A} (a_i - 2),
\end{equation}
so removing indices with $a_i = 2$ from $A$ does not change $|[A]|$. If $[A_j]$ is unbreakable, by Lemma~\ref{lem:TwoIndBr} $A_j$ has at most one element $i$ with $a_i \ge 3$, so $|[A_j\cap A_k]|$ is either $|[A_j]|$, if $A_k$ contains the index for which $a_i \ge 3$, or $2$, if it does not.
\end{proof}

\begin{definition}
Following Greene \cite{greene:LSRP}, say that two intervals $[A_j]$ and $[A_k]$ are {\it distant} if they do not intersect and $\braket{[A_j]}{[A_k]} = 0$, that they are {\it consecutive} and write $A_j\dagger A_k$ if they do not intersect and $\braket{[A_j]}{[A_k]} = -1$, and that they share a {\it common endpoint} if they intersect and $\braket{[A_j]}{[A_k]} = |[A_j\cap A_k]| - 1$. If $A_j$ and $A_k$ share a common endpoint and $A_j\subset A_k$, write $A_j\prec A_k$. Say that two intervals {\it abut} if they are either consecutive or share a common endpoint.
Write $A_j\pitchfork A_k$ if $A_j\cap A_k\ne\emptyset$ and they do not share a common end point.
\end{definition}

\begin{definition}\label{defn:InterGraph}
Let the {\it intersection graph} $G(T)$, where $T \subset S$ is a subset of the standard basis, be the graph with vertex set $T$, and an edge between $v_i$ and $v_j$ (write $v_i\sim v_j$) whenever $[A_i]$ and $[A_j]$ abut. If $v_i\sim v_j$ and $i<j$, we say $v_i$ is a {\it smaller neighbor} of $v_j$.
\end{definition}

For the intersection graph to be a useful concept, we need to somehow relate abutment of intervals to pairings in the lattice. First, we need to know how the pairings between intervals relate to the pairings between standard basis vectors:

\begin{lemma} \label{tausign}
For any two standard basis vectors $v_i, v_j$ with $\{i,j\} \neq \{1,f\}$, $\braket{v_i}{v_j} = \pm \braket{[A_i]}{[A_j]}$. Furthermore, if both $v_i$ and $v_j$ have norm at least $3$, then $\braket{v_i}{v_j} = \epsilon_i\epsilon_j\braket{[A_i]}{[A_j]}$.
\end{lemma}
\begin{proof}
If $\epsilon_i = \epsilon_j$, $\braket{v_i}{v_j} = \braket{[A_i]}{[A_j]}$ since $\tau$ is a reflection. Otherwise,
\begin{equation*}
\braket{[A_i]}{[A_j]} = \braket{\tau(v_i)}{v_j} = \braket{v_i}{\tau(v_j)},
\end{equation*}
so the desired results will hold as long as either $\tau(v_i) = -v_i$ or $\tau(v_j) = - v_j$. Using Proposition~\ref{prop:OnlyTau}, the only intervals (corresponding to standard basis vectors) that are not simply negated by $\tau$ are $x_*$, $x_{**}$, intervals with left endpoint $0$ and, if $a_0 = 2$, intervals with left endpoint 1. Note that any interval with left endpoint $0$ is not tight since the tight vector pairs differently with $x_*$ and $x_{**}$.

First, consider the case $a_0 \ge 3$. If one of $[A_i]$ or $[A_j]$ is either $x_*$ or $x_{**}$ and the other is an interval with left end $0$ (which in particular is neither tight nor $x_*$ or $x_{**}$), then the result follows from the facts that $\tau(x_*) = x_{**}$ and Lemma~\ref{pairequal}. If both $[A_i]$ and $[A_j]$ are intervals starting at $0$, then neither one is tight, by Lemma~\ref{lem:BrIsTight} both are unbreakable. Therefore, since $a_0 \ge 3$ and $0$ is in both $A_i$ and $A_j$, every other $k \in A_i \cup A_j$ has $a_k = 2$. Assume $\braket{v_i}{v_j} \ne \pm \braket{[A_i]}{[A_j]}$, it follows from (\ref{eq:IntervalNorm}) that $|v_i| = |v_j| = a_0$. Note that
\[
\braket{v_i}{v_j} = \braket{\tau([A_i]}{[A_j]} = \braket{[A_i]}{\tau([A_j])},
\]
and $\tau([A_i])=-[A_i]-x_*-x_{**}$,
we can compute that $\braket{v_i}{v_j}$ is either $3 - a_0$, if $A_i \neq A_j$, or $2 - a_0$ if $A_i = A_j$. We can check that neither of these can occur: for any two standard basis vectors, $\braket{v_i}{v_j} \ge -1$, so $a_0$ is either 3 or 4, and $\braket{v_i}{v_j}$ is either $0$ or $-1$. By Lemma~\ref{pairequal}, $\braket{v_i}{v_1} = \braket{v_j}{v_1} = \braket{v_i}{v_f} = \braket{v_j}{v_f} = -1$, so $1 \in \supp^+ v_i \cap \supp^+ v_j$. Therefore,  $\braket{v_i}{v_j} = 0$, and $a_0 = 3$. However, using Lemma~\ref{lem:j-1}, the only possible standard basis vectors of norm 3 that have pairing $-1$ with both $v_1$ and $v_f$ are $e_1 + e_{f-2} - e_{f-1}$ and $e_1 + e_f - e_{f+1}$, but these have pairing $1$ with each other. This contradiction shows that there cannot be two standard basis vectors corresponding to intervals with left endpoint $0$. Therefore, given any two standard basis vectors of norm at least $3$ (so neither one is $x_*$ or $x_{**}$), one of them must be negated by $\tau$, which proves the last statement in this case.


When $a_0 = 2$, the situation is similar. Using Proposition~\ref{prop:OnlyTau}, the only standard basis intervals that are not negated by $\tau$ are $x_*$, $x_{**}$, $x_0$, and intervals that start at $1$. The vector $x_0$ is fixed by $\tau$. For any interval $[A]$ with left end $1$, $\tau([A])=-[A]-x_*-x_{**}-2x_0$, so
$\braket{x_*}{[A]} = \braket{x_*}{\tau([A])} = 0$, and this also holds for $x_{**}$. It only remains to rule out the case in which $[A_i]$ and $[A_j]$ are both intervals with left end $1$. In this case, again $v_i$ and $v_j$ are unbreakable, so $1$ is the only index $k \in A_i \cup A_j$ with $a_k \ge 3$, $|v_i| = |v_j| = a_1$, and $\braket{v_i}{v_j}$ is either $3 - a_1$ or $2 - a_1$, always $\le0$. However,
\begin{equation*}
\braket{v_i}{x_*} = \braket{v_i}{x_{**}} = \braket{v_j}{x_*} = \braket{v_j}{x_{**}}=0
\end{equation*}
and
\begin{equation*}
\braket{v_i}{x_0} = \braket{v_j}{x_0} = -1,
\end{equation*}
so $\supp^+(v_i) \cap \{0,1,2,3\} = \supp^+(v_j) \cap \{0,1,2,3\} = \{2,3\}$. This forces $\braket{v_i}{v_j} \ge 1$, a contradiction. Again, since $x_*$, $x_{**}$, and $x_0$ all have norm $2$, at least one of any pair of standard basis vectors of norm at least $3$ is negated by $\tau$, and the conclusion follows.
\end{proof}

Because all of the $A_j$ are subintervals of one interval and the pairings between the $[A_j]$ are, up to sign, the same as the pairings between standard basis vectors, many of the results from \cite{greene:LSRP} about the intersection graph carry over unchanged to this situation. Most importantly,  \cite[Lemma 4.4]{greene:LSRP} holds unchanged:

\begin{lemma}[Lemma~4.4 in \cite{greene:LSRP}] \label{smallpair}
If $v_i$ and $v_j$ are unbreakable standard basis vectors of norm at least $3$, then $|\braket{v_i}{v_j}| \le 1$, with equality if and only if $[A_i]$ and $[A_j]$ are consecutive and $\braket{v_i}{v_j} = -\epsilon_i\epsilon_j$.
\end{lemma}

The proof of Lemma~\ref{smallpair}, which will not be repeated here, is identical to the one Greene gives. The overall strategy is similar to the one used in Lemma~\ref{tausign}: showing that, for unbreakable standard basis vectors $v_i, v_j$ of norm at least $3$, if $A_i \cap A_j$ contains an index $k$ with $a_k \ge 3$ then the pairing $\braket{v_i}{v_j}$ will be too large given the form of the standard basis vectors. The proof actually shows 
\begin{cor}\label{uniqueheavy}
If $v_i, v_j$ are unbreakable standard basis vectors of norm at least $3$, then $z_i \neq z_j$, where $z_i$ is defined in Definition~\ref{defn:z_j}.
\end{cor} 

\begin{cor}\label{cor:AdjacencyNon0}
If $\{i,j\}\ne\{1,f\}$ and $v_i\sim v_j$, then $\braket{v_i}{v_j}\ne0$. If $v_i\not\sim v_j$ and $v_i,v_j$ are both unbreakable, then $\braket{v_i}{v_j}=0$.
\end{cor}
\begin{proof}
If $v_i\sim v_j$, then $\braket{[A_i]}{[A_j]}$ is equal to either $-1$ or $|[A_i\cap A_j]| - 1\ne0$. It follows from Lemma~\ref{tausign} that $\braket{v_i}{v_j}\ne0$.

If $v_i\not\sim v_j$, by Lemma~\ref{tausign} we only need to show $\braket{[A_i]}{[A_j]}=0$. If $A_i,A_j$ are distant, we are done. If $A_i\pitchfork A_j$, by Lemma~\ref{lem:IntervalProd} 
\[\braket{[A_i]}{[A_j]}=[A_i\cap A_j]-2.\]
By Corollary~\ref{uniqueheavy}, $A_i\cap A_j$ does not contain any vertex with norm $\ge3$, so $[A_i\cap A_j]=2$ and our conclusion holds.
\end{proof}

In particular, when $S$ contains no breakable vectors, $G(S)$ is almost the same as $\hat{G}(S)$, with the only differences being that $G(S)$ has an edge between $v_1$ and $v_f$ while $\hat{G}(S)$ does not.

\begin{cor} \label{1fblock}
At most one unbreakable vector other than $v_1$ or $v_f$ neighbors $v_1$ in $G(S)$. The same holds for $v_f$, and, when $a_0 = 2$, $v_2$.
\end{cor}
\begin{proof}
If $v_j \sim v_1$, or $v_j \sim v_f$, then using Proposition~\ref{prop:OnlyTau} and Corollary~\ref{cor:AdjacencyNon0} we get either $j \in \{1,f\}$ or $0 \in A_j$. If $a_0 \ge 3$ and $v_j$ is unbreakable, then $z_j = 0$, so no other unbreakable standard basis vector corresponds to an interval containing $0$. When $a_0 = 2$, by Proposition~\ref{prop:OnlyTau} and Corollary~\ref{cor:AdjacencyNon0}, $v_2$ is the only neighbor of $v_1$ or $v_3$, and when $v_j \sim v_2$ for some $j \not \in \{1,f\}$ and $v_j$ is unbreakable, $z_j = 1$. 
\end{proof}

Once we have Lemma~\ref{smallpair}, almost all of the remaining results of \cite[Section~4]{greene:LSRP} carry over.
In particular, we have versions of \cite[Lemmas 4.8 and 4.10]{greene:LSRP}, with identical proofs:

\begin{definition}
A {\it claw} in a graph $G$ is a set of four vertices $(v;x,y,z)$ with $v$ adjacent to $x$, $y$, and $z$, and no two of $x, y,$ and $z$ adjacent to each other.
\end{definition}

\begin{lemma}[Lemma~4.8 in \cite{greene:LSRP}]
$G(S)$ has no claws.
\end{lemma}

Let $\bar{S}$ be the set of unbreakable elements of $S$, and let $S_g=\{v_1,v_2,\dots,v_g\}$.

\begin{definition}
A triple $(v_i,v_j,v_k) \in (\bar{S})^3$ is a {\it heavy triple} if each of $v_i, v_j, $ and $v_k$ have norm at least 3, and any two of them are connected by a path in $G(\bar{S})$ that does not pass through the third. If the heavy triple $(v_i,v_j,v_k)$ spans a triangle, we say it is a {\it heavy triangle}.
\end{definition}

\begin{lemma}[Lemma~4.10 in \cite{greene:LSRP}]
No triple is heavy.
\end{lemma}
\begin{proof}
Suppose that $(v_i,v_j,v_k)$ is heavy, with $z_i < z_j < z_k$. Then any path from $v_i$ to $v_k$ in $G(\bar{S})$ would contain an unbreakable interval containing $z_j$, contradicting Corollary \ref{uniqueheavy}.
\end{proof}

The last of Greene's results that we will need to characterize the D-type lattices that are isomorphic to changemaker lattices is a description of the cycles that can occur in $G(S)$. However, this will require more modification since Greene's proof relies on the the intervals $[A_i]$ being linearly independent, which fails in this situation since $[A_1] = [A_f]$. Luckily, this is the only problem:

\begin{lemma} \label{indepintervals}
$[A_2], [A_3], \dots, [A_{n+2}]$, and $[A_{n+3}]$ are linearly independent.
\end{lemma}
\begin{proof}
Consider the map $\pi: L \otimes \Q \to L \otimes \Q$ given by $\pi(x) = x - \tau(x)$. For any $x\in L \otimes \Q$, $\pi(\tau(x)) = -\pi(x)$, so in particular $\pi([A_i]) = \pm \pi(v_i)$ for each standard basis vector $v_i$.

When $a_0 \ge 3$, $\tau$ is the reflection in a line, so $\im \pi$ has dimension $n+2$. Since $(v_1,v_2, \dots, v_{n+3})$ spans $L \otimes \Q$, $(\pi(v_1), \pi(v_2), \dots, \pi(v_{n+3}))$ spans $\im \pi$. However, $\pi(v_1) = -\pi(v_f)$, so actually $\pi(v_2), \dots, \pi(v_{n+3})$ suffice. Since $\im \pi$ has dimension $n+2$, this means that $\pi(v_2), \dots, \pi(v_{n+3})$ are linearly independent. Since $\pi([A_i]) = \pm \pi(v_i)$ for each $i$, a linear dependence among $[A_2], \dots, [A_{n+3}]$ would induce one among $\pi(v_2), \dots, \pi(v_{n+3})$, and the conclusion follows.

When $a_0 = 2$, $\tau$ is the reflection in a plane, so $\im(\pi)$ has dimension $n+1$, and a basis is $(\pi(v_3), \pi(v_4), \dots, \pi(v_{n+3}))$. The same argument as before gives that $[A_3],[A_4], \dots, [A_{n+3}]$ are linearly independent. It remains only to be seen that $[A_2] = v_2$ is not in the span of $([A_3], \dots, [A_{n+3}])$. Note that $[A_3] = v_1$ and, for $j > 3$, $\tau(v_j) =-v_j$ or $-v_1 - 2v_2 - v_3 - v_j$, so
\begin{equation*}
    \operatorname{span}([A_3], \dots, [A_{n+3}]) \subset \operatorname{span}(v_1,v_1+2v_2+v_3, v_4, \dots, v_{n+3}),
\end{equation*}
and the right side does not contain $v_2$. 
\end{proof}

Once we have this, a proof identical to that of Greene's Lemma 3.8 gives

\begin{lemma}
Any simple cycle in $G(\{v_2,v_3, \dots, v_{n+3}\})$ induces a complete subgraph.
\end{lemma} 

Actually, since $[A_1] = [A_f]$, for any $j \not \in \{1,f\}$, $v_j \sim v_1$ if and only if $v_j \sim v_f$, so the same statement holds for $G(\{v_1,v_2,\dots,v_{f-1},v_{f+1},\dots, v_{n+3}\})$.

\begin{lemma}\label{lem:CycleTwoCases}
For any simple cycle $C \subset G(S)$, one of the following holds:
\begin{itemize}
\item The vertex set $V(C)$ induces a complete subgraph of $G(S)$
\item $C$ has exactly four vertices, of which three are $v_1, v_f,$ and $v_{f-1}$, and $v_{f-1}$ is breakable
\end{itemize}
\end{lemma}

\begin{proof}
If $C$ contains exactly $3$ vertices, $C$ is a complete subgroph. So we assume $C$ contains at least $4$ vertices.
If either $v_1$ or $v_f$ is not in $C$, the previous lemma says that $C$ induces a complete subgraph. Otherwise, $v_1,v_f \in V(C)$. Since any neighbor (other than $v_1$) of $v_f$ also neighbors $v_1$, $v_1$ must have two neighbors $v_j,v_k$ also in $C$, which will also neighbor $v_f$. By Corollary~\ref{1fblock}, one of these, say $v_k$, is breakable. By Lemma~\ref{lem:BrIsTight} and Lemma~\ref{pairequal}, $k = f-1$, and neither $v_1$ nor $v_f$ has any other neighbors. This means that the vertices of $C$ other than $v_1, v_f, v_j,$ and $v_{f-1}$ all lie on a path $\gamma$, for which one end neighbors $v_j$ and the other neighbors $v_{f-1}$. It remains to see only that $\gamma$ is empty. Since $v_f$ neighbors both $v_j$ and $v_{f-1}$, if $\gamma$ were nonempty we could form a new simple cycle going from $v_j$ to $v_{f-1}$ along $\gamma$, then to $v_f$ and back to $v_j$. This does not contain $v_1$, so induces a complete subgraph. This means that any vertex of $\gamma$ neighbors $v_f$, so $\gamma$ is empty. 
\end{proof}

\begin{lemma}\label{lem:CycleLength}
Any cycle in $G(S)$ has three vertices, unless it contains a breakable vector in which case it can contain $4$.
\end{lemma}
\begin{proof}
By Corollary~\ref{1fblock} and Lemma~\ref{lem:CycleTwoCases}, we can assume that $V(C)$ does not contain both $v_1$ and $v_f$ and that $V(C)$ induces a complete subgraph. If $v_{f-1}$ is breakable and in $V(C)$, then $V(C) \setminus \{v_{f-1}\}$ will still induce a complete subgraph, so it suffices to assume that $V(C)$ contains no breakable vector. If $V(C)$ had more than two vectors of norm at least $3$, any three of them would form a heavy triple. Since $V(C)$ induces a complete subgraph, any two vectors of norm $2$ in $V(C)$ are of the form $v_i = e_{i-1} - e_i$ and $v_{i+1} = e_i - e_{i+1}$, for some index $i$. However, any other standard basis vector $v_j$ that neighbored both $v_i$ and $v_{i+1}$ would have to have a gappy index at either $i-1$ or $i$, which cannot happen by Lemma~\ref{gappy3} because $v_i$ and $v_{i+1}$ both have norm $2$. Therefore, $V(C)$ has at most one vector of norm $2$, and at most two vectors of norm at least $3$, so has at most $3$ vectors overall. Since $C$ is a cycle, it must have exactly $3$.
\end{proof}

The main use we will have for this characterization of cycles is the following lemma that restricts the possible forms of gappy vectors, which is \cite[Lemma~7.1]{greene:LSRP}. As usual, the proof will not be repeated.

\begin{lemma}[Lemma~7.1 in \cite{greene:LSRP}]\label{gappyform}
Suppose that $v_g \in S$ is gappy, and that $S$ contains no breakable vector. Then $v_g$ is the unique gappy vector, $v_g = -e_g + e_{g-1} + \cdots + e_j + e_k$ for some $k + 1 < j < g$, and $v_k$ and $v_{k+1}$ belong to distinct components of $G(S_{g-1})$.
\end{lemma}


\section{Blocked vectors}\label{sec:Blocked}

In this section, we define blocked vectors and prove some lemmas about blocked vectors. These lemmas will be useful in the later sections.

\begin{definition}
Let $m\in\{1,2,\dots,n+2\}$. A vector $v_i \in S_m$ is {\it$(m,N)$-blocked} if $v_i \not \sim v_j$ for any $m < j \leq N$. Otherwise, $v_i$ is {\it $(m,N)$-open}. When $N = n+3$ and $S_N = S$, we simply say $v_i$ is {\it $m$-blocked} or {\it $m$-open}.
\end{definition}

\begin{lemma} \label{mblock3}
When $v_{f-1}$ is unbreakable, $v_1$ and $v_f$ are $f$-blocked. 
\end{lemma}
\begin{proof}
If $a_0 \geq 3$, by Corollary~\ref{uniqueheavy} $A_{f-1}$ is the unique unbreakable interval containing $x_0$, so $v_1,v_f$ are $f$-blocked. If $a_0 = 2$, then $f = 3$ and $|v_1|=|v_2|=|v_3| = 2$ by Lemma~\ref{modify}. Since $0,1,2$ cannot be gappy indices, no vector other than $v_2$ can neighbor $v_1,v_3$ at the same time, so $v_1,v_3$ are $3$-blocked by Lemma~\ref{pairequal}.
\end{proof}

Below, we introduce the notion of an $(m,N)$-blocking neighbor. We will see in Lemma~\ref{mblock1} that a standard basis vector is $(m,N)$-blocked if it has two $(m,N)$-blocking neighbors that do not abut each other.

\begin{definition}\label{defn:blockingneighbor}
Let $v_i,v_j \in S_m$ be such that $v_j \sim v_i$. We say $v_j$ is an {\it$(m,N)$-blocking neighbor of $v_i$} if one of the following holds:
\begin{enumerate}
\item $v_j$ is $(m,N)$-blocked.
\item $|v_i| \ge 3$, $|v_j|\ge 3$, both unbreakable. $v_\ell$ is unbreakable for $m < \ell \le N$. 
\item $j = i \pm 1$, and $|v_i| = |v_j| = 2$.
\item $j = i \pm 1$, $|v_{\max (i, j)}| = 2$, and $v_{\min(i,j)}$ is unbreakable. $v_\ell$ is $(m,N)$-blocked for $\ell < \min(i,j)$, and $v_\ell$ is unbreakable for $m < \ell \leq N$.
\end{enumerate}
When $N = n+3$ and $S_N = S$, we simply say $v_j$ is an {\it $m$-blocking neighbor of $v_i$}.
\end{definition}

\begin{lemma}\label{premblock1}
Suppose $v_j$ is an $(m,N)$-blocking neighbor of $v_i$. If $v_i \sim v_k$ for some $m < k \leq N$, then $v_j \not\sim v_k$.	
\end{lemma}
\begin{proof}
We prove this case by case.
\begin{enumerate}
\item Since $v_j$ is $(m,N)$-blocked, by definition $v_k \not\sim v_j$.
\item For contradiction, suppose $v_k \sim v_j$. Since $v_k \sim v_i$ and $v_k \sim v_j$, $|v_k| \geq 3$. Otherwise, if $|v_k| = 2$, $v_k$ has only one smaller neighbor. Since we assume $v_i,v_j,v_k$ are unbreakable, $(v_i,v_j,v_k)$ is a heavy triangle.
\item Suppose $v_k \sim v_j$ for contradiction. Assume $j = i+1$. This assumption is made without loss of generality because the situation $v_k \sim v_i,v_j$ and $v_i \sim v_j$ is symmetric in $i$ and $j$. 
Since $v_k \sim v_j$ and $v_j = e_i-e_{i+1}$ ($i > 0$), exactly one of $i$ and $i+1$ is in $\supp v_k$. Since $|v_j| = 2$, $j-1 = i$ cannot be a gappy index by Lemma~\ref{gappy3}. Therefore, $i \notin \supp v_k$ and $i+1 \in \supp v_k$.
Since $v_k \sim v_i$ and $v_i = e_{i-1}-e_i$, $i-1 \in \supp v_k$. This contradicts Lemma~\ref{gappy3} since $|v_i| = 2$.
\item Again, suppose $v_k \sim v_j$ for contradiction. Assume $j = i+1$ without loss of generality, since the situation is symmetric in $i$ and $j$. From item 3, we may assume $|v_i| \geq 3$ and $|v_{j}| = 2$. 
As in item 3, $i \notin \supp v_k$ and $i+1 \in \supp v_k$. Let $\ell=\min\supp v_k$. Since $v_k \sim v_i$ and $i \notin \supp v_k$, $\ell < i$. 

If $\ell=0$, let $s$ be the first gappy index of $v_k$, then $s<i$, we claim that $v_k\sim v_{s+1}$. Since $\braket{v_k}{v_{s+1}}>0$, if $v_k\not\sim v_{s+1}$, it follows from Lemmas~\ref{lem:IntervalProd} and \ref{tausign} and Corollary~\ref{cor:AdjacencyNon0} that $v_{s+1}$ is breakable, $A_k\pitchfork A_{s+1}$ and $\epsilon_k=\epsilon_{s+1}$. Thus up to applying $\tau$, $v_k-v_{s+1}$ becomes $[A_k]-[A_{s+1}]$, which is reducible since it is the signed sum of two distant intervals, a contradiction to Lemma~\ref{newirreducibility}.
Since $v_k\sim v_{s+1}$ and $s+1\le i$, $s=i-1$ by our assumption. 
Since $|v_i| \geq 3$, we must have $\braket{v_i}{v_k} \geq 2$. However, since $v_i$ and $v_k$ are both unbreakable, this contradicts Lemma~\ref{smallpair}.

If $\ell>0$, then $\braket{v_{\ell}}{v_k}=-1$. 
Since $v_k \not\sim v_\ell$, it follows from Corollary~\ref{cor:AdjacencyNon0} that $v_{\ell}$ is breakable, so $\ell=f-1$. By Lemma~\ref{tausign} and Lemma~\ref{lem:IntervalProd}, $|v_k|=3$.  Using Lemma~\ref{lem:j-1}, we get $v_k = -e_k + e_{k-1} + e_{\ell}$. Since $\ell < i < m< k$, $k > \ell+2$ so $k-1 > \ell+1$. Since $v_{\ell+1} =v_f= -e_{\ell+1} + e_\ell$, $\braket{v_k}{v_{\ell+1}}=1$, so $v_k\sim v_{\ell+1}$. If $\ell+1 < i$, then $v_{\ell+1}$ is $m$-blocked by assumption, a contradiction. If $\ell+1 = i$, $|v_i|=2$, a contradiction to our assumption that $|v_i|\geq3$.
\qedhere
\end{enumerate}
\end{proof}

\begin{lemma}\label{mblock1}
If $v_i \in S_m$ has two $(m,N)$-blocking neighbors $v_{j_1}, v_{j_2} \in S_m$ such that $v_{j_1} \not\sim v_{j_2}$, then $v_i$ is $(m,N)$-blocked.	
\end{lemma}
\begin{proof}
Suppose $v_k \sim v_i$ for some $m < k \leq N$. By Lemma~\ref{premblock1}, $v_{j_1} \not\sim v_k$ and $v_{j_2} \not\sim v_k$. However, this would imply $\{v_i; v_{j_1},v_{j_2},v_k\}$ is a claw.
\end{proof}



\begin{lemma} \label{mblock2}
Suppose that $v_k$ is just right for all $k>m$, and that $v_j \in S_m$ is unbreakable. Then $v_i \in S_m$ is $m$-blocked if either of the following holds
\begin{enumerate}
\item $\supp v_j$ has at least $4$ elements $\ge i$.
\item $\supp v_j$ has at least 3 elements $\ge i$, and $v_j$ is m-blocked.
\end{enumerate}
\end{lemma}
\begin{proof}
In the first case, any $v_k$ neighboring $v_i$ with $k>m$ would have $\min\supp v_k\le i$, hence $|\supp v_k\cap \supp v_j|\ge4$ and
$\braket{v_k}{v_j} \ge 2$, contradicting Lemma~\ref{smallpair}. In the second case if $v_i\sim v_k$ for some $k>m$, $\braket{v_k}{v_j} \ge 1$, so $v_k \sim v_j$, a contradiction.
\end{proof}

\begin{lemma} \label{2sforever}
Suppose $v_i$ is $(m,N)$-blocked for all $i < m$. If $S_{m-1}$ contains at least one unbreakable vector of norm $\ge 3$, then $|v_j| = 2$ for all $m < j \leq N$. 

\end{lemma}
\begin{proof}
We induct on $N-m$. The base case $N-m = 0$ is trivial. Inductively, suppose the lemma holds whenever $N-m \leq k$ for any non-negative integer $k$. The following shows that the lemma also holds when $N-m = k+1$. 

Since $v_i$ is $(m,N)$-blocked for all $i < m$, either $\min\supp v_{m+1} = 0$ or $|v_{m+1}| = 2$. Since $S_{m-1}$ contains some unbreakable vector $v_i$ of norm $\ge 3$, $\min\supp v_{m+1} = 0$ would imply $v_{m+1} \sim v_i$, but $v_i$ is $m$-blocked. Therefore, $|v_{m+1}| = 2$.  

Since $S_{m-1}$ contains a high norm vector, $v_m$ has some smaller neighbor $v_s$, which is $(m,N)$-blocked by assumption. By Definition~\ref{defn:blockingneighbor}~(1)(4), $v_s$ and $v_{m+1}$ are $(m+1,N)$-blocking neighbors of $v_m$. Since $|v_{m+1}| = 2$ and $s < m$, $v_s \not\sim v_{m+1}$. Hence, by Lemma~\ref{mblock1}, $v_m$ is $(m+1)$-blocked. 

Since $N-(m+1)\leq k$ and $v_i$ is $(m+1,N)$-blocked for all $i<m+1$, the inductive assumption implies $|v_j| = 2$ for all $m+1<j\leq N$. Since we have also shown $|v_{m+1}| = 2$, this finishes the induction proof.
\end{proof}

\begin{lemma} \label{opensguaranteed}
If $G(S_m)$ is disconnected, every connected component has at least one $m$-open vector.
\end{lemma}
\begin{proof}
Otherwise, a component of $G(S_m)$ would still be a component in $G(S_{n+3})$, disconnected from the rest of $G(S_{n+3})$. However, this contradicts Corollary~\ref{indecomposable}, which says the lattice is indecomposable and $G(S_{n+3})$ is connected.
\end{proof}

\begin{lemma} \label{2ssoon}
Suppose $S$ contains no breakable vector. If $S_m$ contains at least one vector of norm $\ge 3$, and contains exactly one $m$-open vector $v_j$ with $j < m$, then $v_{m+1}$ neighbors $v_j$ and has no other smaller neighbors, and $|v_k| = 2$ for all $k > m+1$.
\end{lemma}
\begin{proof}
Since $S_m$ has a high norm vector, $v_{m+1}$ has some smaller neighbor, which must be $v_j$. By Lemma~\ref{opensguaranteed}, $G(S_m)$ is connected, so $v_j$ neighbors some $v_i \in S_m$. Since $v_j$ is the only $m$-open vector in $S_m$, $v_i$ is $m$-blocked, and $v_i$ is an $m$-blocking neighbor of $v_j$ by Definition~\ref{defn:blockingneighbor}~(1). Since $v_{m+1} \sim v_j$ and $j < m$, $|v_{m+1}| \ge 3$. By Definition~\ref{defn:blockingneighbor}~(2), $v_{m+1}$ is an $(m+1)$-blocking neighbor of $v_j$. By Lemma~\ref{mblock1}, $v_j$ is $(m+1)$-blocked. The rest follows by Lemma~\ref{2sforever}.
\end{proof}


\section{$a_0=2$}\label{sec:a0=2}

In this section, we assume $a_0 = 2$. By Lemma~\ref{modify}, $f = 3$, and $v_2 = e_1 - e_2$. 

\begin{lemma}\label{lem:v2sblock}
If $v_s \sim v_2$ for $s>3$, then $v_2$ is $s$-blocked.	
\end{lemma}
\begin{proof}
By Lemma~\ref{modify} and Lemma~\ref{a1whena0=2}, $v_2 = x_0$ and $|x_1| = a_1 \geq 3$. If $v_s \sim v_2$, then by Corollary~\ref{uniqueheavy} $A_s$ is the unique interval containing $x_1$. Therefore, $v_1,v_3,v_s$ are the only neighbors of $v_2$.
\end{proof}

\begin{lemma}
The vector $v_4$ is just right and $|v_4| \in \{3,5\}$.
\end{lemma}
\begin{proof}
By Lemma~\ref{mblock3}, $v_1,v_3$ are $3$-blocked. We can then use Lemmas~\ref{gappy3} and~\ref{lem:j-1} to get our conclusion.
\end{proof}

\begin{lemma}
Suppose $|v_4| = 3$. Unless $4 = n+3$, $|v_5| \in \{2,6\}$ and $|v_j| = 2$ for $j >5$.
\end{lemma}
\begin{proof}
By Lemma~\ref{mblock3}, $v_1$ and $v_3$ are $3$-blocked. Lemma~\ref{lem:v2sblock} implies that $v_2$ is $4$-blocked since $v_4 \sim v_2$. If $v_5$ exists, to avoid pairing with $v_1,v_2,v_3$, either $|v_5| = 2$ or $|v_5| = 6$. In either case, $v_4$ is $5$-blocked by Lemma~\ref{mblock1} and Definition~\ref{defn:blockingneighbor} (1),(4) or (1),(2). By Lemma~\ref{2sforever}, $|v_j| = 2$ for all $j > 5$.
\end{proof}

\begin{lemma}
Suppose $|v_4| = 5$. There is an index $s$ for which $v_s$ is just right, $|v_s|=s-1$, and $|v_j| = 2$ for $4 < j < s$. Either $s = n+3$, or $|v_{s+1}| = 3$ and $|v_j| = 2$ for $j > s+1$.
\end{lemma}
\begin{proof}
Note that $G(S)$ is connected while $G(S_4)$ has two components with vertex sets $\{v_1,v_2,v_3\}$ and $\{v_4\}$. Since $v_1$ and $v_3$ are $4$-blocked, $v_2$ must be $4$-open.
Let $s > 4$ be the (unique) index with $v_s \sim v_2$.
Since $|v_2| = 2$, $1$ cannot be a gappy index for $v_s$. Hence, $v_s \sim v_2$ implies $\{1,2\}\cap \supp v_s=\{2\}$, which then implies $3 \in \supp v_s$. To avoid having pairing 2 with $v_4$ (which would contradict Lemma~\ref{smallpair}), $4 \in \supp v_s$ and $v_s \sim v_4$. 

For contradiction, suppose there exists $|v_j| \ge 3$ with $4 < j < s$ chosen minimally. Since $|v_5|=\cdots=|v_{j-1}|=2$, we have $v_4\sim v_5\sim\cdots\sim v_{j-1}$. Since $v_j$ does not neighbor $v_1, v_2,v_3$, and $|\braket{v_4}{v_j}|\le1$ by Lemma~\ref{smallpair}, we must have $\min\supp v_j\ge4$. Thus, $v_j$ pairs nontrivially with one of $v_4,v_5,\dots,v_{j-1}$. In other words, in $G(S_{s-1})$, $v_j$ is in the same connected component as $v_4$. Since $|v_i| = 2$ for $4 < i < j$, any $i$ cannot be a gappy index for $4 \leq i < j-1$. Since $2,3,4 \in \supp v_s$, $\{2,3,...,j-1\} \subset \supp v_s$. Hence, $\braket{v_s}{v_j}\ge|v_j|-2$, which in turn implies $v_s \sim v_j$, creating a heavy triple $(v_4, v_j, v_s)$ and resulting in a contradiction. 

Therefore, $|v_j| = 2$ for $4 < j < s$. It follows from Lemma~\ref{gappy3} that $v_s$ is just right, and $|v_s| = s-1$.
So all the adjacency relations in $G(S_s)$ are
\[
v_1\sim v_2, \quad v_1\sim v_3, \quad v_2\sim v_3,\quad v_2\sim v_s\sim v_4\sim v_5\sim\cdots\sim v_{s-1}.
\]
By Lemma~\ref{lem:v2sblock}, since $v_2 \sim v_s$, $v_2$ is $s$-blocked. 
By Lemma~\ref{mblock1} and Definition~\ref{defn:blockingneighbor} (3) and (4), $v_j$ is $s$-blocked when $4<j<s-1$. By Lemma~\ref{mblock1} and Definition~\ref{defn:blockingneighbor} (1) and (2), $v_4$ is $s$-blocked. By Lemma~\ref{mblock1} and Definition~\ref{defn:blockingneighbor} (1), $v_s$ is $s$-blocked. So
$v_{s-1}$ is the only possible $s$-open vector. If $s<n+3$, the connectivity of $G(S)$ implies that $v_{s-1}$ is $s$-open, so we can use Lemma~\ref{2ssoon} to conclude that $v_{s-1}$ is the only smaller neighbor of $v_{s+1}$, which implies that $v_{s+1}=e_{s-1}+e_{s}-e_{s+1}$, and $|v_j|=2$ for $j>s+1$.
\end{proof}

To summarize, we have the following proposition.

\begin{prop}\label{f=3JR}
When $a_0=2$, one of the following holds:
\begin{enumerate}
    \item $\abs{v_4} = 3$, $\abs{v_i} = 2$ for $i \geq 5$.
    \item $\abs{v_4} = 3$, $\abs{v_5} = 6$, and $\abs{v_i} = 2$ for $i > 5$.
    \item $\abs{v_4} = 5$, $\abs{v_i} = 2$ for $5\leq i\leq s-1$, $\abs{v_s} = s-1$, and $s = n+3$. 
    \item $\abs{v_4} = 5$, $\abs{v_i} = 2$ for $5\leq i\leq s-1$, $\abs{v_s} = s-1$, $\abs{v_{s+1}} = 3$, and $\abs{v_j} = 2$ for $j \geq s+2$. 
\end{enumerate}
\end{prop}

The corresponding changemakers are
\begin{itemize}
    \item $(1,1,1,1,2^{[s]})$, $s>0$,
    \item $(1,1,1,1,4^{[s]},4s+2,4s+6^{[t]})$,
\end{itemize}
where $a^{[s]}$ means $s$ copies of $a$.

\begin{rmk}
We combine degenerate cases when listing changemakers at the end of each section, and in our parametrization $s,t \geq 0$ unless otherwise specified. The parameter $s$ is not the previous $s$ in this section. The same convention is also used in later sections.
\end{rmk}


\section{Classification of $S_f$}\label{sec:basisthroughf}

From now on, we assume $a_0 \geq 3$. In this section, we classify the possible forms of $v_1, \dots, v_f$. Always, $v_1 = e_0-e_1$, and $v_f = e_{f-1}-e_f$.

\begin{lemma}\label{lem:AdjV1}
For $1<i<f-1$, we have $v_i\not\sim v_1$. Moreover, $v_{f-1}\sim v_1$.
\end{lemma}
\begin{proof}
By Lemma~\ref{pairequal},
$\braket{v_i}{v_1} =\braket{v_i}{v_f} =0$ for $1<i<f-1$.
By Lemma~\ref{parity}, $\braket{v_{f-1}}{v_1} \equiv \braket{v_{f-1}}{v_f} = -1 \pmod 2$, so $v_{f-1}\sim v_1$.
\end{proof}

\begin{prop}\label{f=3basistovf}
When $f = 3$, $v_2 = 2e_0 + e_1 - e_2$.
\end{prop}
\begin{proof}
It follows from Lemma~\ref{lem:AdjV1} that $v_2\sim v_1$, so $v_2 = e_1 - e_2$ or $v_2 = 2e_0 + e_1 - e_2$. Since $\langle v_2,v_1\rangle \neq 0$, $x_0 \in A_2$, so $|v_2|\geq |x_0|=a_0 \geq 3$.
\end{proof}

\begin{prop}\label{f>3basistovf}
When $f > 3$, $v_2 = e_0 + e_1 - e_2$, and $v_j$ is just right for all $j < f-1$. 
\begin{enumerate}
	\item If $v_{f-1}$ is tight, then $|v_j| = 2$ for $2 < j < f-1$.
	\item If $v_{f-1}$ is just right, then $v_{f-1} = e_1+...+e_{f-2}-e_{f-1}$, and one of the following holds.
		\begin{enumerate}
			\item $|v_j| = 2$ for $2 < j < f-1$.
			\item $f = 5$, and $|v_3| = 4$.
			\item $f = 6$, $|v_3| = 2$, and $|v_4| = 3$.
		\end{enumerate}
	\item If $v_{f-1}$ is gappy, then $f\ge5$ and $v_{f-1} = e_1 + e_3 + \cdots + e_{f-2} - e_{f-1}$, and one of the following holds.
		\begin{enumerate}
    		\item $|v_j| = 2$ for $2 < j < f-1$.
    		\item $|v_3| = 4$, and $|v_j| = 2$ for $3 < j < f-1$.
    		\item $|v_3| = 2$, $|v_4| = 3$, and $|v_j| = 2$ for $4 < j < f-1$.
    	\end{enumerate}
\end{enumerate}
\end{prop}
\begin{proof}
By Lemma~\ref{lem:AdjV1}, $v_2\not\sim v_1$, so $v_2 = e_0 + e_1 - e_2$. Again using Lemma~\ref{lem:AdjV1}, to avoid a claw, any two smaller neighbors of $v_{f-1}$ other than $v_1$ neighbor each other. Therefore, to avoid a cycle of length 4, (which would violate Lemma~\ref{lem:CycleLength},) $v_{f-1}$ has at most two smaller neighbors other than $v_1$. Furthermore, if $v_{f-1}$ is not tight and has two smaller neighbors, then one of them must have norm 2, otherwise these three vectors would form a heavy triple. (Recall that $v_{f-1}$ is the only possible tight or breakable vector by Lemma~\ref{lem:BrIsTight} and Lemma~\ref{pairequal}.)

When there exists $|v_j| > 2$ for some $2 < j < f-1$, define \[s = \min\{2 < j < f-1:|v_j|\geq 3\}.\] The only possible gappy index for $v_s$ is $1$, which cannot occur since then either $\braket{v_s}{v_1} = -1$ or $\braket{v_s}{v_2} = 2$, contradicting Lemma~\ref{lem:AdjV1} or Lemma~\ref{smallpair}. Thus, $v_s$ is just right. Let 
\[k = \min\supp v_s\le s-2.\] 
To avoid a claw $(v_k;v_{k-1},v_{k+1},v_s)$, we must have $k \leq 2$. Since $\braket{v_s}{v_1} = 0$, \[k\in \{0,2\}.\] So 
\begin{equation}\label{eq:vsv2}
    \braket{v_s}{v_2}=\pm1\ne0.
\end{equation}

Suppose $v_{f-1}$ is tight, then $\braket{v_{f-1}}{v_2} = 2 = |v_2| - 1$. It follows from Lemma~\ref{lem:IntervalProd} and Lemma~\ref{tausign} that $\epsilon_2 = \epsilon_{f-1}$, and $A_2$ shares an end with $A_{f-1}$. Since the left endpoint of $A_{f-1}$ is $v_*$, (see Proposition~\ref{prop:OnlyTau} and the paragraph following it,) $A_2$ and $A_{f-1}$ share right endpoints. By Lemma~\ref{smallpair} and (\ref{eq:vsv2}), $A_s \dagger A_2$. Therefore, either $A_s \dagger A_{f-1}$ and $\braket{v_s}{v_{f-1}} = \pm 1$, or $A_s \pitchfork A_{f-1}$ and $\braket{v_s}{v_{f-1}} = |v_s|-2$. If $k = 0$, then $\braket{v_s}{v_{f-1}} = |v_s|-1$, which does not fall into either case. Hence, $k = 2$, $\braket{v_s}{v_{f-1}} = |v_s|-2$. We claim that $A_s$ and $A_{f-1}$ are not consecutive. Otherwise, $\braket{v_s}{v_{f-1}} =1$ and $\epsilon_s=-\epsilon_{f-1}=-\epsilon_2$. Hence $\braket{v_s}{v_2} =-\braket{[A_s]}{[A_{2}]} =1$, contradicting the fact that $k=2$. So $A_s \pitchfork A_{f-1}$. However, since $v_s \not\sim v_{f-1}$, $(v_2;v_s,v_{f-1},v_3)$ is a claw. Thus, $v_s$ cannot exist when $v_{f-1}$ is tight.

Suppose $v_{f-1}$ is just right. From the first paragraph we know that $v_{f-1}$ has at most one smaller neighbor with norm $\ge3$.
Since $v_{f-1}\sim v_1$, $\min \supp v_{f-1} = 1$. Therefore, $v_{f-1} \sim v_j$ for $2 < j < f-1$ if and only if $|v_j| \ge 3$. This implies $v_1$ and $v_s$ are the only smaller neighbors of $v_{f-1}$, and $|v_i| = 2$ for $s < i < f-1$. Since $\braket{v_{f-1}}{v_s} \le 1$, either $s = 3$ and $v_s = e_0 + e_1 + e_2 - e_3$ or $s = 4$ and $v_s = e_2 + e_3 - e_4$. In either case, unless $s+1 = f-1$, $(v_s;v_2,v_{s+1},v_{f-1})$ forms a claw. Thus, $f = s + 2$.

Suppose $v_{f-1}$ is gappy. By Lemma~\ref{lem:BrIsTight} and Lemma~\ref{pairequal}, $S$ contains no breakable vector. Since $v_{f-1}\sim v_1$, $1 =\min \supp v_{f-1}$. By Lemma~\ref{gappyform},
\begin{equation*}
    v_{f-1} = e_1 + e_j + \cdots + e_{f-2} - e_{f-1},
\end{equation*}
for some $j \ge 3$, and all other standard basis vectors are just right. By Lemma~\ref{lem:j-1}, $f-2\in \supp v_{f-1}$, so $f-2\ge j\ge3$. Since $v_{f-1} \sim v_2$,
by the discussion in the first paragraph of this proof, 
\begin{equation}\label{eq:NeighborNorm2}
\text{if } v_i\sim v_{f-1} \text{ and }2<i<f-1, \text{ then }|v_i|=2.
\end{equation}
In particular, $v_s\not\sim v_{f-1}$.
If  $|v_i| = 2$ for any $2 < i < f-1$, to avoid a claw $(v_j; v_{j-1}, v_{j+1}, v_{f-1})$ we must have either $j=3$ or $j=f-2$. If $j=f-2>3$, we would have a cycle 
\[v_2\sim v_3\sim\cdots\sim v_{f-1}\sim v_2\]
of length $f-2>3$. So $j=3$ if $v_s$ does not exist.
Now assume $v_s$ exists.
Note $v_s \sim v_2$ by (\ref{eq:vsv2}), $v_{f-1} \sim v_2$, and $v_3 \sim v_2$ if $3<s$. To avoid a claw $(v_2; v_3, v_s, v_{f-1})$ either $s = 3$ or $v_{f-1} \sim v_3$. In the first case, $v_3 = e_0 + e_1 + e_2 - e_3$. Since $|v_3|>2$, $v_{f-1} \not \sim v_3$, forcing $j = 3$. If $|v_i| \ge 3$ for some $i$ with $3<i<f-1$, by (\ref{eq:NeighborNorm2}) we have $v_i \not\sim v_{f-1}$. Hence $i = 4$ and $v_4 = e_2 + e_3 - e_4$, which would create a claw $(v_2; v_3, v_4, v_{f-1})$. Therefore, $|v_i| = 2$ for all $3 < i < f-1$. In the second case, by (\ref{eq:NeighborNorm2}) $|v_3| = 2$. Since $v_{f-1} \sim v_3$, $j = 3$. Since $v_s \not\sim v_{f-1}$, $s = 4$ and $v_s = e_2 + e_3 - e_4$. If $|v_l|>2$ for some $l$ with $4<l<f-1$, we would have $\braket{v_l}{v_{f-1}}>0$, contradicting (\ref{eq:NeighborNorm2}).
\end{proof}


\section{$f >3$, $v_{f-1}$ gappy}\label{sec:Gappy}

Recall from Proposition~\ref{f>3basistovf}, when $f >3$ and $v_{f-1}$ is gappy, 
\begin{align*}
v_2 &= e_0+e_1-e_2,\\
v_{f-1} &= e_1+e_3+\cdots+e_{f-2}-e_{f-1}.    
\end{align*}

\begin{prop}\label{vf-1gappy}
When $v_{f-1}$ is gappy, one of the following holds:
\begin{enumerate}
    \item $f = n+3$.
    \item $f = 5$, $|v_3| = 2$, $|v_6| = 5$, and $|v_j| = 2$ for $j > f+1$.
    \item $|v_{f+1}| = 4$, and $|v_j| = 2$ for $j > f+1$. Either $|v_4|=3$, or $|v_3|=4$.
\end{enumerate}
\end{prop}
\begin{proof}
By Lemmas~\ref{lem:BrIsTight} and~\ref{pairequal}, $S$ contains no breakable vector.
Since $v_{f-1}$ is unbreakable, $v_1$ and $v_f$ are $f$-blocked by Lemma~\ref{mblock3}. Since $v_{f-1}\sim v_1$ and $v_2$, $v_{f-1}$ is also $f$-blocked by Lemma~\ref{mblock1} and Definition~\ref{defn:blockingneighbor} (1) and (2). Lemma~\ref{gappyform} implies that $v_j$ is just right for $j > f$. 

When $f > 5$, $v_i$ is $f$-blocked for $i < f-2$ by Lemma~\ref{mblock2}, (the $j$ in the statement of Lemma~\ref{mblock2} is $f-1$,) and the only possible $f$-open vector is $v_{f-2}$. When $f = 5$ and $|v_3| = 4$, 
\[v_3=e_0+e_1+e_2-e_3,\quad v_4=e_1+e_3-e_4.\]
Since $v_3\sim v_2 \sim v_4$ and $v_3 \not\sim v_4$, $v_2$ is $f$-blocked by Lemma~\ref{mblock1} and Definition~\ref{defn:blockingneighbor} (1) and (2), so again the only possible $f$-open vector is $v_{f-2}$. In these two cases, by Lemma~\ref{2ssoon}, $v_{f+1} \sim v_{f-2}$, and $|v_j| = 2$ for $j > f+1$. Since $v_{f-2}$ is the only smaller neighbor of $v_{f+1}$ and $v_{f+1}$ is just right, $|v_{f+1}| = 4$.

When $f = 5$ and $|v_3| = 2$, both $v_2$ and $v_3$ are possibly $f$-open. Since $v_6 \not\sim v_1,v_4$ or $v_5$, $|v_6| \in \{4,5\}$. In either case, $v_2$ and $v_3$ are $6$-blocked by Lemma~\ref{mblock2}, and by Lemma~\ref{2sforever}, $|v_j| = 2$ for $j > f+1$.

To see the second statement in item 3, note when $|v_3| = \dots = |v_{f-2}| = 2$ and $|v_{f+1}| = 4$, 
$(v_{f-1},v_2,v_{f+1})$ forms a heavy triple, so either Case~3~(b) or Case~3~(c) in Proposition~\ref{f>3basistovf} happens.
\end{proof}

The corresponding changemakers are:
\begin{itemize}
    \item $(1,1,2,2,2^{[s]},2s+3,2s+3)$,
    \item $(1,1,2,4,4^{[s]},4s+5,4s+5,8s+14^{[t]})$,
    \item $(1,1,2,2,4^{[s]},4s+3,4s+3,8s+10^{[t]})$.
\end{itemize}

\section{$f > 3$, $v_{f-1}$ just right}\label{sec:JustRight}

Assume $f > 3$ and $v_{f-1}$ is just right. Recall from Proposition~\ref{f>3basistovf} that 
\begin{align*}
v_2 &= e_0+e_1-e_2,\\
v_{f-1} &= e_1+\cdots+e_{f-2}-e_{f-1}.    
\end{align*}
First, we consider Cases~2(b) and 2(c) in Proposition~\ref{f>3basistovf}.

\begin{lemma}\label{lem:f=n+3}
If there exists $|v_j| \ge 3$ for some $2 < j < f-1$, then $f = n+3$. 
\end{lemma}

Recall that in this case $f = j + 2=5 \text{ or }6$ by Proposition~\ref{f>3basistovf}~2(b) and 2(c).
\begin{proof}[Proof of Lemma~\ref{lem:f=n+3}]
By Lemma~\ref{mblock3},
$v_1$ and $v_f$ are $f$-blocked, since $v_{f-1}$ is unbreakable. Therefore, $v_{f-1}$ is $f$-blocked by Lemma~\ref{mblock1}, since it neighbors $v_1$ and $v_j$. Because $G(S_f)$ is connected, any $G(S_g)$ is also connected for $g>f$. So there are no gappy vectors by Lemma~\ref{gappyform}. By Lemma~\ref{mblock2}, $v_{i}$ is $f$-blocked for $i < f-2$. It remains only to be shown that $v_{f-2} = v_j$ is $f$-blocked, but this follows from Lemma~\ref{mblock1} and the fact that it neighbors $v_2$ and $v_{f-1}$, both of which are $f$-blocked but do not neighbor each other.
\end{proof}

From now on in this section, we assume $|v_i| = 2$ for $2 < i < f-1$, i.e. Case~2(a) in Proposition~\ref{f>3basistovf}.

\begin{lemma}\label{f+1}
The vector $v_{f+1}$ is just right. Furthermore,
\begin{enumerate}
	\item If $f = 4$, then $|v_5| \in \{3,4,6\}$.
	\item If $f = 5$, then $|v_6| \in \{3,4,5\}$.
	\item If $f > 5$, then $|v_{f+1}| \in \{3,4\}$.
\end{enumerate}
\end{lemma}
\begin{proof}
The only vectors of norm at least 3 in $S_f$ are $v_2$ and $v_{f-1}$, so the only possible gappy indices for $v_{f+1}$ are $1$ and $f-2$. However, $1$ cannot be a gappy index because then either $\braket{v_{f+1}}{v_1} = -1$ or $\braket{v_{f+1}}{v_2} = 2$, contradicting Lemma~\ref{mblock3} or Lemma~\ref{smallpair} and $f-2$ cannot be a gappy index because then $\braket{v_{f+1}}{v_f} = -1$, contradicting Lemma~\ref{mblock3}. Therefore, $v_{f+1}$ is just right. 

Let $j = \min \supp v_{f+1}$. To avoid pairing with $v_1$ and $v_f$, either $j=0$ or $1<j<f$. When $f = 4$, $j\in\{0,2,3\}$ and $|v_5|\in\{3,4,6\}$. When $f>4$, $j \neq 0$ because otherwise $\braket{v_{f-1}}{v_{f+1}} = f-3 > 1$, contradicting Lemma~\ref{smallpair}. So $1 < j < f$. Since $\braket{v_{f-1}}{v_{f+1}} = f - 2 - j \le 1$, $j \ge f-3$. If $j = f-3$, then $(v_{f-3};v_{f-4},v_{f-2},v_{f+1})$ is a claw unless $f = 5$. Therefore, $j$ is either $f-1$ or $f-2$, unless $f = 5$ in which case $j$ can be $f-3$. 
\end{proof}

\begin{lemma}
If $f = 4$ and $|v_5| = 6$, or $f = 5$ and $|v_6| = 5$, then $f+1 = n+3$.
\end{lemma}
\begin{proof}
In either case, $G(S_{f+1})$ is connected, so $v_{f+2}$ would be just right by Lemma~\ref{gappyform}. Since $\langle v_{f+2},v_{f+1}\rangle \leq 1$ and $v_{f+2} \not\sim v_f$, $\min\supp v_{f+2} \in \{f-1,f+1\}$. If $\min\supp v_{f+2} = f-1$, we have a heavy triple $(v_{f-1},v_{f+1},v_{f+2})$. If $\min\supp v_{f+2} = f+1$, then one obtains the claw $(v_{f+1};v_{2},v_{f-1},v_{f+2})$.
\end{proof}

Now, assume we are not in the above two cases, i.e. 
\begin{equation}\label{eq:vf+1=34}
|v_{f+1}| \in \{3,4\}.    
\end{equation}
Then $G(S_{f+1})$ is disconnected.
 Let $s > f+1$ be the smallest index with $|v_s| \ge 3$. Note that such an index must exist, since otherwise $\hat{G}(S)$ would be disconnected, contradicting Corollary~\ref{indecomposable}.

\begin{lemma}\label{f+1openvectors}
The only possible $(f+1)$-open vectors are $v_{f-2}$, $v_{f-1}$, and $v_{f+1}$. 
\end{lemma}
\begin{proof}
Since $v_{f-1}$ is unbreakable, by Lemma~\ref{mblock3}, $v_1,v_f$ are $f$-blocked. As the statement is obvious for $f = 4$, we may assume $f > 4$ without loss of generality. Since $|v_j| = 2$ for $3 \leq j \leq f-2$, $v_j$ is $f$-blocked for $3 \leq j < f-2$ by Definition~\ref{defn:blockingneighbor} (3) and (4) and Lemma~\ref{mblock1}. The following shows $v_2$ is $(f+1)$-blocked. Suppose $v_j \sim v_2$ for some $j > f+1$. If $2 \not\in \supp v_j$, then $v_j \not\sim v_1$ implies $0,1 \in \supp v_j$, which contradicts $|\braket{v_j}{v_2}| \leq 1$. If $2 \in \supp v_j$, then $\{2,3,...,f-2\} \subset \supp v_j$ because none of $2,...,f-3$ can be a gappy index. We also have $f-1 \in \supp v_j$, because otherwise $\braket{v_j}{v_{f-1}}\geq f-3 >1$. Since $|v_f|=2$, $f-1$ is not a gappy index, so $f \in \supp v_j$.  To satisfy (\ref{eq:vf+1=34}) and $|\braket{v_{f+1}}{v_{j}}|\leq 1$, it must be the case that $|v_{f+1}| = 3$, but this results in a heavy triangle $(v_{f-1},v_{f+1},v_j)$.
\end{proof}

\begin{lemma}
If $v_s$ is just right, then $s = f+2$. Furthermore, either $|v_{f+1}| = 4$ and $|v_{f+2}| = 4$, or $|v_{f+1}| = 3$ and $|v_{f+2}| = 5$. In either case, $f+2 = n+3$.
\end{lemma}
\begin{proof}
Since $|v_s| \ge 3$, $j:=\min \supp v_s < s-1$. By assumption, $|v_{f+2}|=\dots=|v_{s-1}| = 2$. If $j\ge f+2$, we would have a claw $(v_j;v_{j-1},v_{j+1},v_s)$. If $j =f+1$, we would have a claw $(v_{f+1};v_{f+2},v_{s},v_{f-1} \text{ or }v_{f-2})$. Since $\braket{v_s}{v_f} = 0$, in fact $j < f$. In particular, $v_s \sim v_{f+1}$. By (\ref{eq:vf+1=34}), $v_{f+1}$ has exactly one smaller neighbor, which is either $v_{f-1}$ or $v_{f-2}$. In either case, this is not  a neighbor of $v_{s}$ - in the first case, this would form a heavy triple, and in the second, it would make $\braket{v_s}{v_{f+1}} \ge 2$. To avoid a claw $(v_{f+1};v_s,v_{f+2},v_{f-1}\text{ or }v_{f-2})$, $s=f+2$. By Lemma~\ref{f+1openvectors}, $\min \supp v_{f+2} \in \{f-1,f-2\}$. By (\ref{eq:vf+1=34}), $\min \supp v_{f+1} \in \{f-1,f-2\}$. As we showed earlier in this paragraph, the smaller neighbor of $v_{f+1}$ is not adjacent to $v_{f+2}$, so $\min \supp v_{f+2} \neq \min \supp v_{f+1}$. The conclusion about $|v_{f+2}|$ follows. 

To see $v_{f+2}$ is the last standard basis vector, note that 
$G(S_{f+2})$ is connected. By Lemma~\ref{gappyform}, $v_{f+3}$ is just right. Since $|v_{f+2}|\ge4$ and $\braket{v_{f+2}}{v_{f+3}}\le1$, $|v_{f+3}|\in\{2,3,4\}$. Since $v_{f+3}\not\sim v_f$, $v_{f+3}\ne4$. In $G(S_{f+2})$, there is a path containing $4$ vertices $v_{f-2}, v_{f+1}, v_{f+2}, v_{f-1}$. Here $v_{f-2},v_{f-1}$ are the two ends, and $v_{f+1}, v_{f+2}$ are in the interior. Since $|v_{f+3}|\in\{2,3\}$,
 $v_{f+3}$ is adjacent to exactly one of $v_{f+1}, v_{f+2}$, and it does not neighbor $v_{f-2},v_{f-1}$. So we get a claw centered at the neighbor of $v_{f+3}$.
\end{proof}

\begin{lemma}
If $v_s$ is gappy, then $f = 4$, $|v_5| = 4$, and 
\begin{equation*}
v_s = e_2 + e_5 + \dots + e_{s-1} - e_s.
\end{equation*}
Furthermore, if $n+3 > s$, then $|v_{s+1}| = 3$, and $|v_j| = 2$ for all $j > s+1$.
\end{lemma}
\begin{proof}
The possible gappy indices for $v_s$ are $1, f$, and $f-2$. However, $1$ (resp. $f$) cannot be a gappy index, because otherwise $\braket{v_s}{v_1}= -1$ (resp. $\braket{v_s}{v_f} = -1$) or $\braket{v_s}{v_2} = 2$ (resp. $\braket{v_s}{v_{f+1}}\ge2$). Therefore, 
\begin{equation*}
    v_s = e_{f-2} + e_j + \dots + e_{s-1} - e_s,
\end{equation*}
for $f < j < s$. We have $v_s \sim v_{f-2}$ and $v_s \sim v_{f-1}$. If $|v_{f+1}|=3$, then $v_{f+1} \sim v_{f-1}$, so either $(v_{f-1}; v_1, v_{f+1}, v_s)$ is a claw, or $(v_{f-1}, v_{f+1}, v_s)$ is a heavy triple. Therefore, $|v_{f+1}| = 4$, and $v_{f+1} \sim v_{f-2}$. If $j > f+1$, then $(v_s, v_{f-2}, v_{f+1}, \dots, v_j, v_s)$ is a cycle of length $\geq 4$. Therefore, $j = f+1$, and $v_s \not \sim v_{f+1}$. To avoid a claw $(v_{f-2}; v_{f-3}, v_{f+1}, v_s)$, $f = 4$ and $j = 5$.

Since $v_{f-1}\sim v_1,v_s$ and $v_{2}\sim v_{f+1},v_s$, we get that $v_{f-1}$ and $v_{2}$ are both $s$-blocked by Lemma~\ref{mblock1} and Definition~\ref{defn:blockingneighbor} (1) and (2), hence $v_s$ is $s$-blocked by Lemma~\ref{mblock1} and Definition~\ref{defn:blockingneighbor} (1). Lemma~\ref{mblock1} and Definition~\ref{defn:blockingneighbor} (1) and (3) can further imply that $v_j$ is $s$-blocked for $f+1\leq j<s-1$. Thus, $v_{s-1}$ is the only $s$-open vector, and the result follows from Lemma~\ref{2ssoon}. 
\end{proof}

Summarizing, we have:

\begin{prop}\label{f!=3JR}
Suppose that $f > 3$ and $v_{f-1}$ is just right, then $|v_{f-1}| = f-1$, and one of the following holds. (Other than $v_s$ in Case 1(e), all vectors are just right.) 
\begin{enumerate}

\item $|v_j| = 2$ for $2 < j < f-1$.
\begin{enumerate}
\item $f = 4$, $n = 2$, $|v_5| = 6$.
\item $f = 5$, $n = 3$, $|v_6| = 5$.
\item $n = f - 1$, $|v_{f+1}| = 3$, $|v_{f+2}| = 5$.
\item $n = f - 1$, $|v_{f+1}| = 4$, $|v_{f+2}| = 4$.
\item $f = 4$, $|v_5| = 4$, $|v_j| = 2$ for $5 < j < s$ or $j > s+1$, $|v_{s+1}| = 3$, and
\begin{equation*}
    v_s = e_2 + e_5 + \dots + e_{s-1} - e_s.
\end{equation*}

\end{enumerate}
\item $f = 5$, $n = 2$, $|v_3| = 4$. 
\item $f = 6$, $n = 3$, $|v_3| = 2$, $|v_4| = 3$. 
\end{enumerate}
\end{prop}

The corresponding changemakers are:
\begin{itemize}
    \item $(1,1,2,3,3,10)$,
    \item $(1,1,2,2,5,5,14)$,
    \item $(1,1,2^{[s+1]},2s+3,2s+3,4s+6,8s+14)$,
    \item $(1,1,2^{[s+1]},2s+3,2s+3,4s+8,8s+14)$,
    \item $(1,1,2,3,3,8,8^{[s]},8s+10,8s+18^{[t]})$,
    \item $(1,1,2,4,7,7)$,
    \item $(1,1,2,2,4,9,9)$.
\end{itemize}

\section{$f > 3$, $v_{f-1}$ tight}\label{sec:f>3Tight}

In this section, we assume $f > 3$ and $v_{f-1}$ is tight. By Lemma \ref{f>3basistovf}, $v_2 = e_0 + e_1 - e_2$, and $|v_j| = 2$ for $2 < j < f-1$. Since $\braket{v_{f-1}}{v_2} = 2 = |v_2| - 1$, 
\begin{equation}\label{eq:epsilon2=f-1}
\epsilon_2 = \epsilon_{f-1}    \text{ and } A_2 \prec A_{f-1}
\end{equation}
by Lemma~\ref{lem:IntervalProd} and Lemma~\ref{tausign}. Since the left endpoint of $A_{f-1}$ is $x_*$ (which is only contained in $A_1$, $A_{f-1}$ and $A_f$), $A_2$ and $A_{f-1}$ must share the right endpoint. 

\begin{lemma} \label{tightblock}
The vector $v_{f-1}$ is $f$-blocked. For $j > f$, $|\braket{v_j}{v_{f-1}}| \in \{0, |v_j| - 2\}$. The only possible $f$-open vectors are $v_1,v_{f-2},v_f$.
\end{lemma}
\begin{proof}
Suppose for contradiction $v_j \sim v_{f-1}$ for some $j > f$. Since the left endpoint of $A_{f-1}$ is $x_*$, $A_j$ cannot share its left endpoint with $A_{f-1}$. Since $j > f$ and $\braket{v_j}{v_{f-1}} \neq 0$, $|v_j| \geq 3$. Since $z_j \neq z_2$ (Corollary~\ref{uniqueheavy}) and $A_2 \prec A_{f-1}$, $A_j$ does not share its right endpoint with $A_{f-1}$. Thus, $A_j \dagger A_{f-1}$. Since the left endpoint of $A_{f-1}$ is $*$, $A_j$ is distant from $A_1$, so $\braket{v_j}{v_1}=0$.
Since $\epsilon_2=\epsilon_{f-1}$, $\braket{v_j}{v_2} = \braket{v_j}{v_{f-1}} = \pm 1$. This implies $\supp v_j \cap \{0,1,2\} \neq \emptyset$, so $\braket{v_j}{v_{f-1}} \geq 0$. Hence, $\braket{v_j}{v_2} = \braket{v_j}{v_{f-1}} = 1$. However, if $\braket{v_j}{v_2} = 1$, using $\braket{v_j}{v_1}=0$, we get $0,1,2 \in \supp v_j$, and $\braket{v_j}{v_{f-1}} \geq 3$, a contradiction. Thus, for all $j > f$, $v_j \not\sim v_{f-1}$, and immediately $|\braket{v_j}{v_{f-1}}| \in \{0, |v_j| - 2\}$ by Lemma~\ref{lem:IntervalProd}.

Since $|v_i| = 2$ for $2 < i < f-1$, $v_i$ is $f$-blocked for $3 < i < f-2$ by Lemma~\ref{mblock1}. When $3<f-1$, by Definition~\ref{defn:blockingneighbor} (1) and (4), $v_{f-1}$ and $v_3$ are $f$-blocking neighbors of $v_2$. By Lemma~\ref{mblock1}, $v_2$ is $f$-blocked. When $3 < f-2$, $v_3$ is also $f$-blocked by Lemma~\ref{mblock1}, since $v_2$ and $v_4$ are $f$-blocking neighbors. 
\end{proof}

\begin{definition}\label{defn:Vj}
$V_j := \supp v_j \cap \{0,1,...,f\}$.
\end{definition}

\begin{lemma} \label{Vjpossible}
$V_{f+1} = \{1,2,\dots,f-2,f\}$, $V_{f+2} = \{f-2\}$ or $\emptyset$, and $V_j = \emptyset$ for $j > f+2$.
\end{lemma}
\begin{proof}
Suppose $j > f$ and $V_j \neq \emptyset$, then $\min\supp v_j \in \{0,1,f-2,f-1\}$ by Lemma~\ref{tightblock} and the parity condition (Lemma~\ref{parity}). We will discuss each of the four possibilities of $\min\supp v_j$. Since $|v_1| = |v_f| = 2$, $0$ or $f-1$ is not a gappy index. Since $|v_i| = 2$ for $2 < i < f-1$, no $i$ is a gappy index for $2 \le i < f-2$. The following shows $1$ is not a gappy index of $v_j$. For contradiction, suppose $1 \in \supp v_j$ and $2 \not\in \supp v_j$. Since $|\braket{v_j}{v_2}| \le 1$, $0 \not\in \supp v_j$. Since $\braket{v_j}{v_2} = 1$, by Lemma~\ref{smallpair} and (\ref{eq:epsilon2=f-1}), $\epsilon_{j} = -\epsilon_2 = -\epsilon_{f-1}$. Since $\braket{v_j}{v_1} = -1$, $A_j$ has $x_0$ as its left endpoint, so $z_{j} \in A_{f-1}$. However, since $\braket{v_j}{v_f} \equiv \braket{v_j}{v_1} \equiv -1 \pmod 2$, $\supp v_j\cap\{f-1,f\}=\{f\}$, so $\braket{v_j}{v_{f-1}} \geq 0$, contradicting $z_{j} \in A_{f-1}$ and $\epsilon_{j} = -\epsilon_{f-1}$. Thus, the only possible gappy index in $\{0,1,\dots,f-1\}$ is $f-2$.

If $\min \supp v_j = 0$, then $0,1,2, \dots,f-2 \in \supp v_j$, and $\braket{v_j}{v_2} = 1$, so $\epsilon_j = -\epsilon_2$ and $A_j\dagger A_2$ by Lemma~\ref{smallpair}. However, since $\braket{v_j}{v_{f-1}} \ge f-1\ge3$, $\epsilon_j = \epsilon_{f-1}$, contradicting (\ref{eq:epsilon2=f-1}). 

If $\min \supp v_j = f-1$, then $\braket{v_j}{v_{f-1}} = -1$. By Lemma \ref{tightblock}, $|\braket{v_j}{v_{f-1}}| = |v_j|-2$, so $|v_j| = 3$, $A_j\pitchfork A_{f-1}$ and $\epsilon_j=-\epsilon_{f-1}$. However, $V_j = \{f-1,f\}$ by the parity condition, so by Lemma~\ref{lem:j-1}, $j = f+1$. Now $v_{f+1} = -e_{f+1}+e_f+e_{f-1}$, so $v_{f+1}+v_{f-1}$ is irreducible by Lemma~\ref{lem:SumIrr}. 
Since $\braket{v_{f+1}}{v_1}=0$, $A_{f+1}$ does not contain $x_0$, so $\tau([A_{f+1}])=-[A_{f+1}]$.
Since $A_{f+1}\pitchfork A_{f-1}$ and $\epsilon_{f+1} = -\epsilon_{f-1}$, up to applying $\tau$, $v_{f-1} + v_{f+1}$ becomes $[A_{f-1}]-[A_{f+1}]$, which is a signed sum of two distant intervals, contradicting the irreducibility of $v_{f-1} + v_{f+1}$.  

If $\min \supp v_j = 1$, then $1, 2, \dots, f-2 \in \supp v_j$. Since $\braket{v_j}{v_f} \equiv \braket{v_j}{v_1} \equiv -1 \pmod 2$ and $f-1$ is not a gappy index, $f-1 \notin \supp v_j$ and $f \in \supp v_j$. Hence, $\braket{v_j}{v_{f-1}} = f-2$. By Lemma~\ref{tightblock}, $\braket{v_j}{v_{f-1}} = |v_j|-2$, so $|v_j| = f$. By Lemma~\ref{lem:j-1}, $j-1,j \in \supp v_j$. So $|v_j| = f$ implies that $j = f+1$ and $V_{f+1} = \{1,2,\dots,f-2,f\}$.

If $\min \supp v_j = f-2$, then $\braket{v_j}{v_f}=\braket{v_j}{v_1} =0$. Suppose $\braket{v_j}{v_{f-1}} = 0$, then $z_j \not \in A_{f-1}$, so $z_j$ is to the right of $z_2$ by (\ref{eq:epsilon2=f-1}). However, since $v_2\sim v_3$ and $v_3 \not\sim v_{f-1}$ when $f > 4$, $A_3$ is consecutive to $A_2$ on the left, so $A_3, \dots, A_{f-2}$ (all of norm 2) lie inside $A_{f-1}$ to the left of $z_2$. This contradicts $v_{f-2} \sim v_j$ when $f > 4$. When $f = 4$, the contradiction instead comes via a similar argument from the fact that $v_j\sim v_2$, but neither $z_j \in A_{f-1}$ nor $v_j \sim v_{f-1}$. The above shows $\braket{v_j}{v_{f-1}} \neq 0$, hence $f-1 \notin \supp v_j$. Since $\braket{v_j}{v_f}=0$, $V_j = \{f-2\}$. 

By Lemma~\ref{lem:j-1}, $f \in \supp v_{f+1}$, we may conclude $V_{f+1} = \{1,2,\dots,f-2,f\}$.

For $j > f+1$, the above shows either $V_j = \emptyset$ or $V_j = \{f-2\}$. If $V_j = \{f-2\}$, then since $|v_j| - 2= \braket{v_j}{v_{f-1}} = 1$, $|v_j| = 3$ and $v_j = -e_j+e_{j-1}+e_{f-2}$. Since $\braket{v_j}{v_{f-1}} > 0$ and $\braket{v_{f+1}}{v_{f-1}} > 0$, using Lemma~\ref{tightblock}, we have $\epsilon_j = \epsilon_{f-1}$ and $\epsilon_{f+1} = \epsilon_{f-1}$. Unless $j = f+2$, $\braket{v_j}{v_{f+1}} = 1$, thus $A_j\dagger A_{f+1}$ and $\epsilon_j = -\epsilon_{f+1}$, a contradiction. 
\end{proof}

It follows from Lemma~\ref{Vjpossible} that
\begin{equation}\label{eq:v_f+1}
v_{f+1} = e_1+...+e_{f-2}+e_f-e_{f+1},
\end{equation}
 and that \[v_{f+2} = e_{f+1}-e_{f+2} \text{ or } e_{f-2}+e_{f+1}-e_{f+2}.\]

\begin{prop}\label{f!=3tight}
We have $|v_j| = 2$ for $j > f+2$.
\end{prop}
\begin{proof}
By (\ref{eq:v_f+1}), $v_{f+1}\sim v_1,v_f$, so $z_{f+1}$ is the leftmost vertex with norm $\ge3$. In particular,  $z_{f+1}\in A_{f-1}$. We also know that $z_2$ is the rightmost vertex in $A_{f-1}$ with norm $\ge3$ from the beginning of this section. Thus, $A_{f+1} \subset A_{f-1}$.

By Lemma~\ref{Vjpossible}, $V_{f+2} = \{f-2\}$ or $\emptyset$. If $V_{f+2} = \emptyset$, then $v_{f+2} \sim v_{f+1}$. Since $A_{f+1} \subset A_{f-1}$ and $z_2\in A_{f-1}$ is to the right of $A_{f+1}$, $A_{f+2} \subset A_{f-1}$. If $V_{f+2}=\{f-2\}$, we claim that $z_{f+2}\in A_{f-1}$.
When $f>4$,
since $v_3 \not\sim v_{f-1}$, $v_3\sim v_2$, and $A_2$ shares the right endpoint with $A_{f-1}$, $A_3$ must abut $A_2$ on its left. Then since $v_{f+2} \sim v_{f-2} \sim ... \sim v_3$ and $z_{f+2}$ is the unique element with norm $\ge3$ in the intervals $A_{f+2},A_{f-2},\dots,A_3$, $z_{f+2}$ must be to the left of $z_2$. When $f = 4$, this still holds because $v_{f+2} \not\sim v_{f-1}$ and $v_{f+2}\sim v_2$. Hence $z_{f+2}\in A_{f-1}$. In either case, since $z_2\in A_{f-1}$ is to the right of $z_{f+2}$, $A_{f+2} \subset A_{f-1}$.
 
For any $j > f+2$, $V_j = \emptyset$ by Lemma~\ref{Vjpossible}. Since $\braket{v_j}{v_{f-1}} = 0$, Lemma~\ref{lem:IntervalProd} implies that either $A_j$ is distant from $A_{f-1}$ or $|v_j| = 2$. If $A_j$ and $A_{f-1}$ are distant, by the result in the first  paragraph in this proof, $A_{f+1} \subset A_{f-1}$ and $v_j \not\sim v_{f+1}$. When $f > 4$, the argument in the second paragraph shows that $A_{f-2} \subset A_{f-1}$, and when $f = 4$, $A_{f-2} = A_2\prec A_{f-1}$ by (\ref{eq:epsilon2=f-1}). So we also have $v_j \not\sim v_{f-2}$. If $|v_j|=2$, $\braket{v_j}{v_{f+1}}=\braket{v_j}{v_{f-2}}=0$, so $v_j\not\sim v_{f+1},v_{f-2}$. Thus, $v_{f+1}$ and $v_{f-2}$ are $(f+2)$-blocked. Since $\braket{v_1}{v_i} = 0$ for all $i > f+2$, $v_1, v_f$ are $(f+2)$-blocked. The result follows from Lemma~\ref{tightblock} and Lemma~\ref{2sforever}.
\end{proof}

The corresponding changemakers are
\begin{itemize}
    \item $(1,1,2^{[s+1]},2s+5,2s+5,4s+8^{[t]})$,
    \item $(1,1,2^{[s+1]},2s+5,2s+5,4s+8,4s+10^{[t]})$, $t>0$.
\end{itemize}

\section{$f=3$}\label{sec:f=3}

In this case, $a_0 \geq 3$, $v_2 = -e_2+e_1+2e_0$ is tight by Proposition~\ref{f=3basistovf}, and the left endpoint of $A_2$ is $x_*$. 

\begin{lemma}\label{highnormsinT2}
If $A_2$ contains $m$ distinct vertices $x_{i_1},x_{i_2},...,x_{i_m}$ of norm $\geq 3$, then $S$ contains at least $m-1$ vectors $v_j$ such that $j > 3$, $z_j \in A_2$ and $V_j \neq \emptyset$, where $V_j$ is defined in Definition~\ref{defn:Vj}.
\end{lemma}
\begin{proof}
Let $\Delta=\Delta(p,q)$, $\overline{\Delta}$ be the quotient of $\Delta$ by the sublattice spanned by vectors of norm 2, and let $\pi:\Delta\to\overline{\Delta}$ be the quotient map. By the definition of $\tau=\tau_0$ in Definition~\ref{defn:Reflection}, we have
\begin{equation}\label{eq:pi}
\pi\circ\tau=-\pi.    
\end{equation}
Since $S$ spans $\Delta$, it follows from Proposition~\ref{prop:OnlyTau} and (\ref{eq:pi}) that $\{\pi([A_i])\}_{i=1}^{n+3}$ spans $\overline{\Delta}$. In particular, $x_{i_1},x_{i_2},...,x_{i_m}$ are all linear combinations of $\pi([A_i])$'s. Hence there are at least $m-1$ intervals  other than $A_2$ containing one  $x_{i_t}$. Since $A_2$ is the only possibly breakable interval, Corollary~\ref{uniqueheavy} implies that at least $m-1$ vertices among $x_{i_1},x_{i_2},...,x_{i_m}$ are each contained in an interval $A_j$ for some $j > 3$.
For each $z_j \in A_2$, $\braket{v_j}{v_2} \neq 0$ and $V_j \neq \emptyset$.
\end{proof}

\begin{lemma}\label{Vjpossibilities}
$v_4 = -e_4+e_3+e_2$ or $v_4 = -e_4+e_3+e_1$. For all $j \ge 5$, $V_j = \emptyset \text{ or } \{0,1\} \text{ or } \{2,3\}$. \end{lemma}
\begin{proof}
Take any $j \geq 4$. Suppose $0,1 \notin V_j$, then Lemma~\ref{pairequal} implies that $\braket{v_j}{v_3}=0$, so $V_j = \emptyset$ or $V_j = \{2,3\}$.

Suppose $0 \in V_j$, then $1 \in V_j$ because $|v_1|=2$ implies $0$ cannot be a gappy index, so $V_j = \{0,1,2,3\}$ or $\{0,1\}$ by Lemma~\ref{pairequal}. However, if $V_j = \{0,1,2,3\}$, then, using Lemma~\ref{lem:IntervalProd}, $\braket{v_j}{v_2} = 2$ implies $z_j \in A_2$ and $|v_j| = 3$ or $4$, which contradicts $V_j = \{0,1,2,3\}$. Hence, $V_j = \{0,1\}$.

Suppose $0 \notin V_j$ and $1 \in V_j$, then Lemma~\ref{pairequal} implies that $V_j = \{1,3\}$. Since $\braket{v_j}{v_1} = -1$, $x_0$ is the left endpoint of $A_j$, and $z_j = x_0 \in A_2$. Hence, $|\braket{v_j}{v_2}| = |v_j|-2$, and $|v_j| = 3$, which only happens when $j = 4$ by Lemma~\ref{lem:j-1}. 

Since $3 \in \supp v_4$ by Lemma~\ref{lem:j-1}, $v_4 = -e_4+e_3+e_2$ or $v_4 = -e_4+e_3+e_1$. 
\end{proof}

\begin{lemma}\label{Vj23}
There is at most one $j \geq 4$ such that $V_j = \{2,3\}$. If $V_j = \{2,3\}$, then $A_j \dagger A_2$ and $\epsilon_j = \epsilon_2$. So $v_j \sim v_2$ and $z_j \notin A_2$.
\end{lemma}
\begin{proof}
Suppose $V_j = \{2,3\}$. Since $\braket{v_j}{v_2} = -1$, either $A_j \dagger A_2$ and $\epsilon_j = \epsilon_2$, or $A_2 \pitchfork A_j$, $|v_j| = 3$ and $\epsilon_j = -\epsilon_2$. If $j > 4$, the latter is impossible because $|v_j| \geq 4$ by Lemma~\ref{lem:j-1}.  If $j = 4$ and $A_2 \pitchfork A_4$, notice that $A_4$ does not contain $x_0$ since $\braket{v_4}{v_1}=0$, so $\tau([A_{4}])=-[A_{4}]$.
Up to applying $\tau$, $v_4 + v_2$ becomes $([A_2] - [A_4])$ which is reducible, contradicting Lemma~\ref{lem:SumIrr}.

Since $j = \min\{i: |z_i|\geq 3\text{ and } z_i \notin A_2\}$, there is at most one such $j$.
\end{proof}

\begin{lemma}\label{2sinbetween}
If $V_j = \emptyset$ for $4 < j \leq m$, then $|v_j| = 2$ for $4 < j \leq m$. 	
\end{lemma}
\begin{proof}
By assumption, $v_1,v_2,v_3$ are $(4,m)$-blocked. Unless $m = 4$ (the lemma is vacuously true when $m = 4$), $|v_5| = 2$ and $v_5 \sim v_4$ by Lemma~\ref{lem:j-1}. If $v_4 = -e_4+e_3+e_2$, then $v_2$ and $v_5$ are $(5,m)$-blocking neighbors of $v_4$ by Definition~\ref{defn:blockingneighbor} (1) and (4), so $v_4$ is $(5,m)$-blocked by Lemma~\ref{mblock1}. If $v_4 = -e_4+e_3+e_1$, then $v_4\sim v_1$, hence $z_4 = x_0$, and $v_i \not\sim v_4$ unless $|v_i| = 2$ or $z_i \in A_2$. If $|v_i| = 2$, $v_i\not\sim v_4$ for $i>5$.
For $5 < i \le m$, since $\braket{v_i}{v_2} = 0$, $z_i \notin A_2$. Hence, $v_4$ is also $(5,m)$-blocked in the case $v_4 = -e_4+e_3+e_1$. The result follows from Lemma~\ref{2sforever}.
\end{proof}

\begin{lemma}\label{lem:v4=-431}
If $v_4 = -e_4+e_3+e_1$, then $z_4 = x_0$, and $A_2$ contains at least two vertices with norm $\ge3$. Moreover,  $\epsilon_2=\epsilon_4$ and $v_2 \not\sim v_4$.
\end{lemma}
\begin{proof}
We have $\braket{v_4}{v_1}\ne0$, so $A_4$ is an interval containing $x_0\in A_2$, $z_4 = x_0$, and $|v_4|=a_0=3$. In particular, $A_2$ and $A_4$ are not consecutive.
Since $|v_2|=6>a_0$, $A_2$ contains another vertex with norm $\ge3$. 
Since $\braket{v_4}{v_2}=1=|v_4|-2$, $\epsilon_2=\epsilon_4$,
$A_2 \pitchfork A_4$ and $v_2 \not\sim v_4$.
\end{proof}

\begin{lemma}\label{Vj01}
If $V_j =\{0,1\}$ for some $j > 4$, we must have $j=5$ and $v_5 = -e_5+e_4+e_1+e_0$.
\end{lemma}
\begin{proof}
Suppose $V_j = \{0,1\}$ for some $j > 4$, then $\braket{v_j}{v_2} = 3$, and $\epsilon_j = \epsilon_2$. By Lemma~\ref{newirreducibility}, $v_j-v_2$ is irreducible, so $A_j \prec A_2$. Hence $|\braket{v_j}{v_2}| = |v_j| - 1$, and $|v_j| = 4$. Depending on whether $v_4 = -e_4+e_3+e_1$ or $v_4 = -e_4+e_3+e_2$, we can use Lemma~\ref{lem:v4=-431} or Lemma~\ref{Vj23} to conclude that $\epsilon_j = \epsilon_2 = \epsilon_4$. 
So $\braket{v_j}{v_4} \leq 0$ by Lemma~\ref{smallpair}, and hence $4 \in \supp v_j$. Since $|v_j| = 4$, Lemma~\ref{lem:j-1} implies that $v_j$ must be $v_5$.
\end{proof}

First, we consider the case where $v_4 = -e_4+e_3+e_2$. Lemma~\ref{Vj23} implies that $A_2 \dagger A_4$. We split the case $v_4 = -e_4+e_3+e_2$ into two subcases according to whether $A_2$ contains multiple high norm vertices.

\begin{lemma}\label{f=3unbreakable}
Suppose $v_4 = -e_4+e_3+e_2$ and $a_0 = 6$. Then, $|v_i| = 2$ for all $i > 4$.
\end{lemma}
\begin{proof}
Since $a_0=|x_0| = |v_2|$, $x_0$ is the only vertex of norm $\geq 3$ in $A_2$. By Lemmas~\ref{Vjpossibilities} and~\ref{Vj23}, $v_5 = -e_5+e_4$ or $-e_5+e_4+e_1+e_0$. If $v_5 = -e_5+e_4+e_1+e_0$, then $\braket{v_5}{v_2} = 3$, contradicting $z_5 \notin A_2$. Therefore, $v_5 = -e_5+e_4$. By Lemmas~\ref{Vjpossibilities}, \ref{Vj23} and~\ref{Vj01}, $V_j = \emptyset$ for all $j > 4$. The result follows from Lemma~\ref{2sinbetween}.
\end{proof}

\begin{lemma}\label{f=3v4justright}
Suppose $v_4 = -e_4+e_3+e_2$ and $a_0 < 6$. Then, $v_5 = -e_5+e_4+e_1+e_0$, and $|v_i| = 2$ for all $i > 5$, except possibly one $v_m = -e_m+e_{m-1}+\cdots+e_4$. 
\end{lemma}
\begin{proof}
Since $|x_0| < |v_2|$, there are at least 2 vertices of norm $\geq 3$ in $A_2$. By Lemma~\ref{Vj23}, $z_4 \notin A_2$. Therefore, by Lemma~\ref{highnormsinT2}, there exists $V_i \neq \emptyset$ for some $i > 4$. By Lemmas~\ref{Vjpossibilities}, \ref{Vj23} and \ref{Vj01}, $v_5 = -e_5+e_4+e_1+e_0$, and $V_i = \emptyset$ for all $i > 5$. This implies $v_2$ is $5$-blocked by Corollary~\ref{cor:AdjacencyNon0}. Observe that $v_4 \sim v_5$. Since $\braket{v_5}{v_2}=3$, 
$z_5 \in A_2$. By Lemma~\ref{Vj23}, $A_2\dagger A_4$, so $z_4 \notin A_2$. If $|v_i| = 2$ for some $i>6$, then $\braket{v_i}{v_5} = 0$, so $v_5 \not\sim v_i$. If $|v_i|\ge3$, since $z_5 \in A_2$ and $A_2 \dagger A_4$, $v_5 \sim v_i$ would imply that $z_i \in A_2$. However, for $i>6$, $V_i = \emptyset$ and $\braket{v_i}{v_2} = 0$, so $z_i \notin A_2$. Thus, $v_5$ is $6$-blocked.

Let $m > 5$ be the minimal index such that $|v_m| \ge 3$, if such $m$ exists. If $m = 6$, $V_m = \emptyset$ implies $v_6 = -e_6+e_5+e_4$. Assume $m > 6$, then $v_5$ is $(m-1)$-blocked. By Lemma \ref{mblock1}, $v_i$ is $(m-1)$-blocked for any $5 < i < m-1$. Thus, $\min\supp v_m = 4$. Since $v_m \not\sim v_5$, $4$ is not a gappy index of $v_m$. Since $|v_i| = 2$ for $5 < i < m$, $i$ is not a gappy index for $5 \leq i < m-1$. Hence, $v_m$ is just right.

Since $v_2$ and $v_m$ are $m$-blocking neighbors of $v_4$ by Definition~\ref{defn:blockingneighbor} (1) and (2), $v_4$ is $m$-blocked by Lemma \ref{mblock1}. We claim $v_{m-1}$ is also $m$-blocked. Suppose $v_i \sim v_{m-1}$ for some $i > m$, then $|v_i| \ge 3$. Observe that $v_i$ is connected to $v_5$ through a path of norm 2 vectors ($v_i \sim v_{m-1} \sim v_{m-2} \sim ... \sim v_5$). Since $z_5 \in A_2$ and $A_2\dagger A_4$, we must have $z_i \in A_2$, which contradicts $V_i = \emptyset$ and $\braket{v_i}{v_2} = 0$. Thus, $v_i$ is $m$-blocked for all $i < m$. By Lemma \ref{2sforever}, $|v_i| = 2$ for all $i > m$.
\end{proof}

Next, we consider the case where $v_4 = -e_4+e_3+e_1$.  

\begin{lemma}\label{onlyV5canbe01}
Suppose $v_4 = -e_4+e_3+e_1$. If $v_5 = -e_5+e_4+e_1+e_0$, then $|v_i| = 2$ for all $i > 5$.
\end{lemma}
\begin{proof}
Assume $v_5 = -e_5+e_4+e_1+e_0$. By Lemmas~\ref{Vjpossibilities} and~\ref{Vj01}, $V_i = \emptyset$ or $\{2,3\}$ for $i>5$. If $V_i = \{2,3\}$, then by Lemma~\ref{Vj23}, $v_i \sim v_2$, $\epsilon_i = \epsilon_2$ and $z_i \notin A_2$. Since $\braket{v_5}{v_2}=3=|v_5|-1$, $A_2$ and $A_5$ share the right endpoint and $\epsilon_5=\epsilon_2$, so $v_i \sim v_5$. Since $z_i \notin A_2$ and $z_4 = x_0$ (which follows from Lemma~\ref{lem:v4=-431}), $v_i \not\sim v_4$. Since $\braket{v_i}{v_4} = 0$, $4 \in \supp v_i$. Since $v_i \sim v_5$, $5 \not\in \supp v_i$ and $\braket{v_i}{v_5} = 1$, so $\epsilon_5 = -\epsilon_i$. However, this contradicts $\epsilon_5 = \epsilon_2$ and $\epsilon_i = \epsilon_2$. Thus, $V_i = \emptyset$ for $i>5$. Therefore, $v_1,v_2,v_3$ are $5$-blocked by Corollary~\ref{cor:AdjacencyNon0}. For $i>5$, since $\braket{v_i}{v_2}=0$, $z_i \notin A_2$. Since $z_4 = x_0$, $v_4 \not\sim v_i$ for $i > 5$. Since $v_1,v_2,v_3,v_4$ are $5$-blocked, $|v_i| = 2$ for all $i > 5$ by Lemma~\ref{2sforever}.
\end{proof}

\begin{lemma}\label{f=3v4gappyonly23}
Suppose $v_4 = -e_4+e_3+e_1$. If there exists $V_m = \{2,3\}$, then $v_m = -e_m+e_{m-1}+\cdots+e_2$ and $|v_i| = 2$ for all $i > 4$, $i \neq m$.
\end{lemma}
\begin{proof}
By Lemmas~\ref{Vjpossibilities}, \ref{Vj23}, \ref{Vj01} and \ref{onlyV5canbe01}, $V_i = \emptyset$ for all $i > 4$, $i \neq m$. By Lemma~\ref{2sinbetween}, $|v_i| = 2$ for $4 < i < m$. Hence, $i$ is not a gappy index for $4 \leq i < m-1$. Recall that $z_4 = x_0$ from Lemma~\ref{lem:v4=-431}. By Lemma~\ref{Vj23}, $z_m \notin A_2$, so $v_m \not\sim v_4$, and $4 \in \supp v_m$. Therefore, $v_m = -e_m+e_{m-1}+...+e_2$.

Since $V_i = \emptyset$ for $i > m$, $v_1,v_2,v_3$ are $m$-blocked. Take any $4 \leq j < m$. Observe that $v_j$ is connected to $v_4$ through a path of norm 2 vectors, and recall that $z_4 = x_0$, so each $A_j$ is contained in the interior of $A_2$. For any $i > m$, $v_i \not\sim v_j$ unless $|v_i| = 2$ or $z_i \in A_2$. If $|v_i|=2$, we have $\braket{v_i}{v_j}=0$, so $v_i\not\sim v_j$. If $z_i \in A_2$, then $\braket{v_i}{v_2}\ne0$, so $V_i\ne\emptyset$, a contradiction.
So $v_j$ is $m$-blocked for $4 \leq j < m$. By Lemma~\ref{2sforever}, $|v_i| = 2$ for all $i > m$.
\end{proof}

\begin{prop}\label{f=3tight}
When $f = 3$ and $v_2$ is tight, one of the following holds:
\begin{enumerate}
    \item $v_4 = -e_4+e_3+e_2$. $|v_i| = 2$ for all $i > 4$. 
    \item $v_4 = -e_4+e_3+e_2$. $v_5 = -e_5+e_4+e_1+e_0$. $|v_i| = 2$ for all $i > 5$, except possibly one $v_m = -e_m+e_{m-1}+\cdots+e_4$.
    \item $v_4 = -e_4+e_3+e_1$. $v_5 = -e_5+e_4+e_1+e_0$. $|v_i| = 2$ for all $i > 5$.
    \item $v_4 = -e_4+e_3+e_1$. $|v_i| = 2$ for all $i > 4$, except possibly one $v_m = -e_m+e_{m-1}+\cdots+e_2$.
\end{enumerate}
\end{prop}
\begin{proof}
When $v_4 = -e_4+e_3+e_2$ and $a_0=6$, item 1 holds by Lemma~\ref{f=3unbreakable}. When $v_4 = -e_4+e_3+e_2$ and $a_0<6$, item 2 holds by Lemma~\ref{f=3v4justright}. When  $v_4 = -e_4+e_3+e_1$ and $V_5 = \{0,1\}$, item 3 holds by Lemma~\ref{onlyV5canbe01}. When $v_4 = -e_4+e_3+e_1$ and $V_5 = \emptyset$ or $\{2,3\}$, item 4 holds by Lemma~\ref{f=3v4gappyonly23} (if there exsits $V_m = \{2,3\}$) or Lemma~\ref{2sinbetween} (if there is no $V_m = \{2,3\}$).
\end{proof}

The corresponding changemakers are:
\begin{itemize}
    \item $(1,1,3,3,4^{[s]},4s+6^{[t]})$, $s+t>0$,
    \item $(1,1,3,3,6,8^{[s]},8s+6^{[t]})$,
    \item $(1,1,3,3,4,6^{[s]})$, $s>0$.
\end{itemize}

\section{Determining $p$ and $q$}\label{sec:pandq}
\begin{table}\centering
\caption{$\mathcal{P}^{-}$, table of $P(p,q)$ that are realizable, $q<0$}\label{table:Types}
   \begin{tabular}{@{}lll@{}} \toprule

    Type & \begin{tabular}{l}$P(p,q)$\end{tabular} &\begin{tabular}{l}Range of parameters  \\
    ($p$ and $r$ are always odd, $p>1$)\end{tabular}\\
    \midrule
    \\
  {\bf 1A}   &\begin{tabular}{l}$P\left(p, -\frac{1}{2}(p^2 - 3p + 4)\right)$ \end{tabular} &\begin{tabular}{l}$p\ge7$ 
  \end{tabular} \bigskip \bigskip \\ \\ 
  {\bf 1B} &\begin{tabular}{l}$P\left(p, -\frac{1}{22}(p^2 - 3p + 4)\right)$\end{tabular} & \begin{tabular}{l}$p \equiv 17$ or $19$ $\pmod{22}$\\ $p>22$\end{tabular}\bigskip\\ \\  \bigskip \bigskip 
  {\bf 2} &\begin{tabular}{l}$P\left(p, -\frac{1}{|4r+2|}(r^2p + 1)\right)$ \end{tabular} &\begin{tabular}{l}$r \equiv -1 \pmod4$\\ $r\ne-1,3$ \\
  $p \equiv 2r-3 \pmod{4r+2}$\end{tabular}  \\ \\ \bigskip \bigskip 
  {\bf 3A} &\begin{tabular}{l}
  $P\left(p, -\frac{1}{2r}(p+1)(p+4)\right)$\end{tabular} &\begin{tabular}{l}$r\ge1$ \\
  $p \equiv -1\pmod{2r}$\\ $p\ge4r-1$ \end{tabular}\\ \\ \bigskip \bigskip 
  {\bf 3B} &\begin{tabular}{l}$P\left(p, -\frac{1}{2r}(p+1)(p+4)\right)$\end{tabular} &\begin{tabular}{l}$r\ge5$ \\
  $p \equiv r-4\pmod{2r}$\\ $p\ge 3r-4$\end{tabular}\\ \\ \bigskip \bigskip 
  {\bf 4} &\begin{tabular}{l}$P\left(p, -\frac{1}{2r^2}\left((2r+1)^2p + 1\right)\right)$\end{tabular} &\begin{tabular}{l}$r\ne-3,-1,1$ \\
  $p \equiv 4r-1 \pmod{2r^2}$\\ $p\ge4r-1$\end{tabular}\\ \\ \bigskip \bigskip 
  {\bf 5} &\begin{tabular}{l}$P\left(p, -\frac{1}{r^2 - 2r - 1}(r^2p + 1)\right)$\end{tabular} &\begin{tabular}{l}$r \neq 1$ \\
  $p \equiv 2r - 5$ $\pmod{r^2 - 2r - 1}$
  \\ $p\ge2r-5$
  \end{tabular}\\ \\ \bigskip \bigskip 
  {\bf Sporadic} &\begin{tabular}{l}$P(11,-30)$,  $P(17,-31)$,\\  $P(13,-47)$, $P(23,-64)$\end{tabular}&\\
  \bottomrule
\end{tabular}

\end{table}

Having classified all the $(n+3)$--dimensional $D$--type lattices which are changemaker lattices, we now aim to concretely compute the pairs $(p,q)$ for these lattices. More precisely, we will give a list $\mathcal P^-$ of prism manifolds with the property that if positive surgery on some knot in $S^3$ results in a prism manifold $P(p,q)$ with $q<0$, then $P(p,q)\in \mathcal P^-$.  

Recall from Section~\ref{sec:Background} that a prism manifold $P(p,q)$ is the boundary of a sharp $4$--manifold $X(p,q)$ when $q<0$. The homology group $H_2(X(p,q))$ equipped with the inner product $-Q_{X(p,q)}$ becomes a lattice isomorphic to $\Delta(p,q)$. From the integers $a_0, a_1, \cdots a_n$ in~\eqref{eq:InnerProduct}, we can recover $p,q$ using (\ref{eq:q/pContFrac}).	 

Our strategy to determine $p$ and $q$ is as follows. Let $S = \{v_1,\dots,v_{n+3}\}$ be the standard basis for a changemaker lattice that is isomorphic to a $D$--type lattice. Convert $S$ into a vertex basis, denoted $S^* = \{v_1^*,\dots, v_{n+3}^*\}$. Using the results of Section~\ref{sec:Changemaker}, the $a_i$ will be recovered. We can then get the pair $(p,q)$ using (\ref{eq:q/pContFrac}). The following lemma helps to simplify some involved computations of continued fractions that we will face. Theses properties are the first two items of~\cite[Lemma~9.5]{greene:LSRP}.  

\begin{lemma}\label{greene9.5}
For integers $r,s,t \ge 2$,
\begin{itemize}
    \item[1.] $[\dots,r,2^{[s]},t,\dots]^- = [\dots, r-1, -(s+1), t-1,\dots]^-$, and
    \item[2.] $[\dots, s, 2^{[t]}]^- = [\dots, s-1, -(t+1)]^-$.
\end{itemize}
\end{lemma}

\begin{ex}
We illustrate how to obtain the pair $(p,q)$ in the case $f = 3$, $v_{2}$ is tight, and $|v_4| = 3$ [c.f. Proposition~\ref{f=3tight}~(2)]. Let $S = \{v_1,...,v_{n+3}\}$ denote the standard basis for the changemaker lattice $L= (\sigma)^{\perp}$. One easily computes  \[\sigma = (1,1,3,3,6,8^{[n-1]}),\quad n\geq 2.\]
From the pairing graph it follows that $S$ is not a vertex basis. Taking
\[v_2^* = v_2 - v_1 - v_5 - ... - v_{n+3},\]
we get that
\[S^* = \{v_1,v_3\}\cup\{v_2^*,v_{n+3},\dots,v_5,v_4\}\]
is indeed a vertex basis with norms giving by the tuple
\[V^* = (2,2,4,2^{[n-2]},4,3),\quad n\geq 2.\]
Using Proposition~\ref{greene9.5} together with (\ref{eq:q/pContFrac}), we get
\begin{equation}\label{SampleCalculation}
\frac{-q}{p} = [4-1,2^{[n-2]},4,3]^- =\frac{16(n-1) + 14}{8(n-1) + 3}.
\end{equation}
Observe that $-q = 2(p + 4)$. Moreover, looking at the denominator of the right hand side of~\eqref{SampleCalculation}, we get that $p \equiv 3\pmod{8}$. Since $n\ge 2$, we have $p \ge 11$. 
\end{ex}

Similar computations for the $D-$type changemaker lattices give prism manifolds $P(p,q)$, so that up to re-parametrization, each falls into one of the families in Table~\ref{table:Types}. We shall denote the collection of these families by $\mathcal{P}^-$. Here we divide the families so that each changemaker vector corresponds to a unique family. The detailed correspondence between the changemaker vectors and $P(p,q)$ can be found in Table~\ref{BigSummary}.
It should be noted that, as discussed in Section~\ref{sec:Background}, the positive integer $p$ is always odd.

\section{Primitive/Seifert-fibered knots admitting prism manifold surgeries}\label{sec:BergeKangPrism}

In the previous sections, we provided a list of changemaker vectors in $\mathbb Z^{n+4}$ whose orthogonal complements are isomorphic to $D$--type lattices. This list gave rise to a collection $\mathcal P^-$ of prism manifolds. Let $\sigma$ be a changemaker vector that corresponds to $P(p,q)\in \mathcal P^-$. We will find a knot $K_{\sigma}\subset S^3$, on which some surgery yields $P(p,q)$.
This is the theme of this section. To start, we need to recall some definitions.

\begin{definition}
Let $H$ be a genus two handlebody. A simple closed curve $c\subset \partial H$ is {\it primitive} in $H$ if $H[c]$, the manifold obtained by adding a two-handle to $H$ along $c$, is a solid torus.  A simple closed curve $c\subset \partial H$ is {\it Seifert-fibered} in $H$ if $H[c]$ is a Seifert fibered space.
\end{definition}

\begin{definition}\label{def:P/SF}
Let $\Sigma$ denote the genus two Heegaard surface of the standard genus two Heegaard spliting of $S^3 = H \cup_{\Sigma} H'$. A knot $K\subset S^3$ is called {\it primitive/Seifert-fibered}, denoted P/SF, if it has a presentation as a simple closed curve on $\Sigma$, such that $K$ is primitive in $H'$ and Seifert-fibered in $H$. The isotopy class in $\partial \nu(K)$ of the curves in $\partial \nu(K) \cap \Sigma$ is called the {\it surface slope of $K$ with respect to $\Sigma$.} Below, we will use $\gamma\in \mathbb Z$ to denote the surface slope.
\end{definition}

P/SF knots are generalizations of doubly primitive knots (i.e. knots that can be isotoped to lie on $\Sigma$ and are primitive in both handlebodies $H$ and $H'$) studied by Berge~\cite{Berge}. In~\cite{Dean2003}, Dean studied the Dehn surgery on P/SF knots along their surface slopes. He proved that the surface slope surgery on a P/SF knot results in either a Seifert fibered manifold 
or a connected sum of lens spaces~\cite[Proposition~2.3]{Dean2003}. Also, Eudave-Mu\~noz proved that the latter case does not happen for hyperbolic P/SF knots~\cite[Theorem~4]{Eudave1992}. 

Berge and Kang, in~\cite{BergeKang}, classify all P/SF knots. Furthermore, given a P/SF knot $K$, they specify the indices of the singular fibers of the Seifert fibered manifold obtained from the surface slope surgery on $K$. In the case that $H[K]$ is Seifert fibered over a disk with two singular fibers, there are five families as follows.
\begin{itemize}
\item[(1)] Knots in solid tori, denoted KIST;
\item[(2)] Knots with torus knot meridians, denoted TKM;
\item[(3)] Knots with cable knot meridians, denoted CKM;
\item[(4)] Knots in once punctured tori, denoted OPT;
\item[(5)] Sporadic.
\end{itemize}
Each of these families of knots are parametrized by a collection of integers with certain constraints, to be made precise below. Types KIST, TKM, and OPT each divides into five subfamilies enumerated I--V.

\begin{rmk}
Berge and Kang also classify all P/SF knots such that $H[K]$ is Seifert fibered over the M\"obius band with one singular fiber. From the data provided in~\cite{BergeKang}, however, it is straightforward to show that the result of the surface slope surgery on such knots is never a prism manifold.
\end{rmk}

Recall that, for a pair of relatively prime integers $p>1$ and $q$, a prism manifold $P(p,q)$ is a Seifert fibered manifold with base orbifold $S^2$ and three singular fibers of index \begin{equation} \label{S2fibers}2,2,p.\end{equation} 

\noindent In what follows, we present an example of a family of P/SF knots that admit prism manifold surgeries. The strategy is to determine which values, if any, for the parameters defining the family in~\cite{BergeKang}, give knots on which the surface slope surgeries result in Seifert fibered manifolds with singular fibers of indices as~\eqref{S2fibers}. Using~\eqref{eq:PrismHomology}, if $P(p,q)$ arises from the surface slope surgery on a P/SF knot, then
\begin{equation}\label{eq:Slope=4q}
\gamma=\text{Surface Slope}=\pm4q.
\end{equation}
First, we make some notational conventions. As in Definition~\ref{def:P/SF}, take $\Sigma$ to be the genus two Heegaard surface of $S^3$. Let $H_1(\Sigma)$ be generated by the class of simple closed curves $A,B,X$, and $Y$, where $A,B$ bound disjoint disks in $H'$, and $X,Y$ bound disjoint disks in $H$. Moreover, these curves are oriented so that
\begin{equation}\label{eq:HomologyClass}A\cdot X=B\cdot Y=-1,\quad A\cdot B=A\cdot Y=X\cdot B=X\cdot Y=0.\end{equation}

\begin{ex}\label{KISTPrism}
The family of KIST~IV knots are characterized by five integer parameters \[(J_1,J_2,\epsilon,n,\widetilde p),\] where  \[|J_1| > 1, |J_2 + 1| > 1, \epsilon = \pm 1, \text{ and } \widetilde p + \epsilon > 1.\] 
(The parameter $\widetilde p$ is the parameter $p$ in \cite{BergeKang}. We use the notation $\widetilde p$ here because $p$ has been used.)
Let $K$ be a P/SF knot of Type KIST~IV. Using the the generators of $H_1(\Sigma)$ as in~\eqref{eq:HomologyClass}, the homology class of a simple closed curve $c$ representing $K$ in $H_1(\Sigma)$ can be written as
\[ [c] = aA + bB + xX + yY,\]
where 
\begin{eqnarray*}
                                  a &=& (J_1J_2 + J_2 + 1),\\ 
                                  b &=& (J_2 + 1)\widetilde p +J_1J_2(\widetilde p+\epsilon), \\
                                  x &=&     -1    \\
                                  y &=&     (J_2+1)(n\widetilde p+\epsilon) + J_1J_2(n(\widetilde p+\epsilon) + \epsilon).                                          
\end{eqnarray*}
(Our $(a,b,x,y)$ are the numbers $(A,B,a,b)$ in \cite{BergeKang}. Note that $A,B$ in \cite{BergeKang}  denote both the simple closed curves on $\Sigma$ and the corresponding coefficients in the homology class $[c]$.)

\noindent The surface slope of $K$ is
\begin{equation}\label{eq:KISTSlope}
\gamma=ax+by,
\end{equation}
and moreover, the surface slope surgery on $K$ yields a Seifert fibered manifold over $S^2$ with singular fibers of multiplicities
\begin{equation}\label{eq:KISTfibers}
|J_1|,|J_2+1|, |\epsilon \widetilde p((J_2+1)(n\widetilde p + \epsilon) + J_1J_2(n(\widetilde p+\epsilon)+\epsilon)) - (J_1(n(\widetilde p+\epsilon)+\epsilon) + \epsilon)|.
\end{equation}
If the result of the surgery is a prism manifold $P(p,q)$, using~\eqref{S2fibers}, we get that two of the three terms in~\eqref{eq:KISTfibers} must be $2$, and the other term will be $p$. It could be that \[|J_1| = |J_2 + 1| = 2.\] 
Let, for instance, $(J_1,J_2,\epsilon) = (2,1,1)$. If $n > 0$, comparing~\eqref{eq:KISTSlope} to~\eqref{eq:Slope=4q} and using~\eqref{eq:KISTfibers}, the resulting prism manifold will be
\begin{equation}\label{eq:SamplePrism}P(n(4\widetilde p^2-2)+4\widetilde p-3,\pm (n(1+2\widetilde p)^2+4\widetilde p+1)).\end{equation}
Taking $p = |n(4\widetilde p^2-2)+4\widetilde p-3|$ and $r = 2\widetilde p+1$, the manifold in~\eqref{eq:SamplePrism} will have equivalent parameters as
\[P(p, \pm \frac{1}{r^2-2r-1}(r^2p+ 1)).\]
When $n < 0$, we get 
\[P(p, \pm \frac{1}{r^2-2r-1}(r^2p- 1)),\]
as a result of the surface slope surgery on $K$.
A similar analysis for all other possible values of $(J_1,J_2,\epsilon,n,\widetilde p)$ gives the results illustrated in Table~\ref{BK Prism with slope sign}. Note that the data in Table~\ref{BK Prism with slope sign} are sign refined version of the results of the computations presented. See Subsection~\ref{SignIssue}.   

\end{ex}
As it will be observed, in order to match all the changemaker vectors, whose orthogonal complements are isomorphic to $D$--type lattices, with P/SF knots it will be sufficient to look into the cases KIST~I, KIST~IV, TKM~II, TKM~V, and OPT~I--V. Now, we give a tabulation of these knots with data similar to those of Example~\ref{KISTPrism}. 
\begin{itemize}
\item KIST~I: Characterized by
\[
 (J,h,h',k,k'),
\]
satisfying
\[
J>0, 2h-k>1, \text{  } k'\ne 0, \text{ and }h'k-hk'=1.
\]
Surgery along
\[
\gamma= -(J+1)+(2h+Jk)(2h'+Jk')
\]
results in a manifold with exceptional fibers of index
\[
J + 1, 2h- k, \text{ and  } |2h' + (2J + 1)k'|,
\]
provided $|2h' + (2J + 1)k'|>1$.

\item KIST~IV: See Example~\ref{KISTPrism}.

\item TKM~II: Characterized by
\[
 (J_1, J_2, \epsilon, n, \widetilde{p}),
\]
satisfying
\[
|J_2|>1, \epsilon=\pm 1, |\widetilde p+1|>1,  \text{ and } |n(\widetilde p+1)+\epsilon|>n>0.
\]
Take $(h,h')=(n\widetilde p+\epsilon,\epsilon\widetilde p),(k,k')=(n(\widetilde p+1)+\epsilon,\epsilon(\widetilde p+1))$. Surgery along
\[
\gamma= J_2(1-J_1J_2-J_1)+((J_1J_2+J_1)k-h)((J_1J_2+J_1)k'-h')
\]
results in a manifold with exceptional fibers of index
\[
|J_1k-h|, |J_2|, \text{ and  } |(J_1J_2+J_1-1)\widetilde p+J_1|,
\]
provided $|(J_1J_2+J_1-1)\widetilde p+J_1|>1$.

\item TKM~V: Characterized by
\[
 (J_1, J_2, \delta, \epsilon, n, \widetilde{p}),
\]
satisfying
\[
|J_1|>1, J_2\ne 0, \delta =\pm 1, \epsilon=\pm 1, \widetilde p>1, n\ge 0 \text{(if } n=0, \epsilon = 1\text{) and } |J_1 (n\widetilde p+\epsilon)+n|>1.
\]
Take $(h, h')=(n\widetilde p+\epsilon, \epsilon \widetilde p)$, and $(k,k')=(n(\widetilde p+1)+\epsilon, \epsilon(\widetilde p+1))$. Surgery along
\[
\gamma= (\delta J_1+1)(J_2-\delta)+(-\delta J_1 h+J_2(J_1-1)h+J_2k)(-\delta J_1 h'+J_2(J_1-1)h'+J_2k')
\]
results in a manifold with exceptional fibers of index
\[
|\delta J_1+1|, |(J_1J_2+1)h+J_2n|, \text{ and  } |(J_2-\delta)\widetilde p-\delta J_2|,
\]
provided all indices are greater than $1$.

\vspace{10pt} 

\hspace{-25pt}{\it Below in the OPT family, we always assume $m\ge0$ and $\gcd(m,n)=1$.}

\item OPT~I: Characterized by
\[
 (m,n,s),
\]
satisfying
\[
|s|>1.
\]
Surgery along
\[
\gamma= m^2+mn+sn^2
\]
results in a manifold with exceptional fibers of index
\[
m, |s|, \text{ and } |m+n|,
\]
provided 
\[
m>1, |s|>1, \text{ and } |m+n|>1.
\]

\item OPT~II: Characterized by
\[
 (m,n,j, \epsilon, \widetilde p),
\]
satisfying
\[
\widetilde p>0, s>1, j\ge 0\text{ and }\epsilon =\pm 1,
\]
where $s=j(2\widetilde p+1)+2\epsilon$. Take $u=-j\widetilde p-\epsilon$. Surgery along
\[
\gamma= m(m+2n)-(-mu+ns)\epsilon[m\widetilde p+n(2\widetilde p+1)]
\]
results in a manifold with exceptional fibers of index
\[
m, s, \text{ and } |(\widetilde p+1)m+(2\widetilde p+1)n|,
\]
provided 
\[
m>1, \text{ and } |(\widetilde p+1)m+(2\widetilde p+1)n|>1.
\]

\item OPT~III: Characterized by
\[
 (m,n,j, \epsilon, \widetilde p),
\]
satisfying
\[
\widetilde p>0, s>1, j\ge 0\text{ and }\epsilon =\pm 1,
\]
where $s=j(2\widetilde p+3)+2\epsilon$. Take 
\[
u=-j(\widetilde p+1)-\epsilon, t=s+u, s'=-\epsilon(2\widetilde p+3),  t'=-\epsilon(\widetilde p+2), \text{ and } u'=\epsilon(\widetilde p+1).
\]
 Surgery along
\[
\gamma= -mn+(m(t+u)+ns)(m(t'+u')+ns')
\]
results in a manifold with exceptional fibers of index
\[
m, |s|, \text{ and } |2m+(2\widetilde p+3)n|,
\]
provided 
\[
m>1, \text{ and } |2m+(2\widetilde p+3)n|>1.
\]

\item OPT~IV: Characterized by
\[
 (m,n).
\]
Surgery along
\[
\gamma= m^2 + 9mn + 22n^2
\]
results in a manifold with exceptional fibers of index
\[
2, m, \text{ and } |3m+11n|,
\]
provided 
\[
m>1, \text{ and }|3m+11n|>1.
\]

\item OPT~V: Characterized by
\[
 (m,n).
\]
 Surgery along
\[
\gamma= 2m^2 + 13mn + 22n^2
\]
results in a manifold with exceptional fibers of index
\[
2, m, \text{ and } |4m+11n|,
\]
provided 
\[
m>1, \text{ and }|4m+11n|>1.
\]
\end{itemize}

\subsection{The sign of the orbifold Euler number}\label{SignIssue}
Let $K$ be a P/SF knot on which the surface slope surgery gives a prism manifold. The surface slope given in \cite{BergeKang} is not necessarily positive, while we only consider positive surgery. To get a knot with positive prism manifold surgery, we need to take the mirror image of the Berge--Kang knot if the surface slope is negative. Suppose the surface slope surgery on $K$ results in $P(p,q)$, then $p$ is the index of the singular fiber with odd index, and $q$ satisfies (\ref{eq:Slope=4q}). To achieve the goal of this section, however, we need to know the sign of $q$. This subsection addresses the issue of sign.

\begin{lemma}\label{lem:EulerPos}
Let $M$ be a Seifert fibered space over $D^2(p_1,p_2)$, $\rho,\iota,\phi\in H_1(\partial M)$ be homology classes corresponding to three distinct slopes on $\partial M$. Here we suppose that
 $\phi$ is the homology class of a regular fiber, and $\iota=0$ in $H_1(M;\mathbb Q)$. We equip $\partial M$ with the outward orientation. Let $M_{\rho}$ be the Seifert fibered space obtained by Dehn filling along the slope corresponding to $\rho$. Then the sign of the orbifold Euler number of $M_{\rho}$ is opposite to the sign of
\begin{equation}\label{eq:TripleDot}
\mathcal I(\iota,\rho,\phi):=(\iota\cdot \rho)(\rho\cdot \phi)(\phi\cdot \iota).
\end{equation}
\end{lemma}
\begin{proof}
Let $M$ be the complement of the singular fiber with index $p_3$ in the Seifert fibered space given by Figure~\ref{fig:SFS} where there are three singular fibers. Without loss of generality, we can assume $e=0$. Then the Euler number of the Seifert fibration of the closed manifold is
\[
e_{\mathrm{orb}}=\frac{q_1}{p_1}+\frac{q_2}{p_2}+\frac{q_3}{p_3}.
\]
The slope of $\iota$ is $1/(\frac{q_1}{p_1}+\frac{q_2}{p_2})$, the unique slope which makes $e_{\mathrm{orb}}(M_{\iota})=0$, and the slope of $\phi$ is $0$. Let $\rho$ correspond to the slope $-\frac{p_3}{q_3}$.
Note that the standard meridian $\mu$ and longitude $\lambda$ on a knot $K\subset Y$ give the inward orientation on $\partial(Y\setminus\nu^{\circ}(K))$. We hence have 
\[(a[\mu]+b[\lambda])\cdot(c[\mu]+d[\lambda])=-\det\begin{pmatrix}
a &b\\
c &d
\end{pmatrix},
\]
if we use the outward orientation on $\partial M$.
Now $\mathcal I(\iota,\rho,\phi)$ has the same sign as
\[-\det
\begin{pmatrix}
p_1p_2&q_1p_2+p_1q_2\\
-p_3 &q_3
\end{pmatrix}\cdot
\det
\begin{pmatrix}
-p_3 &q_3\\
0 &1
\end{pmatrix}\cdot
\det
\begin{pmatrix}
0 &1\\
p_1p_2&q_1p_2+p_1q_2
\end{pmatrix},\]
which is equal to $-e_{\mathrm{orb}}(p_1p_2p_3)^2$.
Thus our conclusion holds.
\end{proof}

In the above lemma, $\mathcal I(\iota,\rho,\phi)$ does not change if we change the sign of any one of $\iota,\rho,\phi$. In fact, the sign of $\mathcal I(\iota,\rho,\phi)$ does not change if we scale any homology class by a nonzero rational number.

We immediately get the following corollary of Lemma~\ref{lem:EulerPos}.

\begin{cor}\label{cor:Sign}
Let $K$ be a P/SF knot with a prism manifold surgery, let $\gamma$ be the surface slope, and let $K'$ be $K$ or its mirror image such that a positive surgery on $K'$ is a prism manifold $Y$. Then the orbifold Euler number of $Y$ is negative if and only if
\[
\mathrm{sign}(\gamma)=\mathrm{sign}(\mathcal I(\iota,\rho,\phi)).
\]
\end{cor}

The proof of the following lemma is left to the reader.

\begin{lemma}\label{lem:SubHomo}
Let $\Sigma$ be a closed oriented surface, $c\subset \Sigma$ be a simple closed curve. Let $\Sigma'$ be the closed surface obtained by surgery on $c$. Define 
\[[c]^{\perp}=\{\alpha\in H_1(\Sigma):\alpha\cdot[c]=0\}\] Then $H_1(\Sigma')$ can be identified with
$[c]^{\perp}/\langle[c]\rangle$.
\end{lemma}

Now, let $H\cup_{\Sigma}H'$ be a genus--$2$ Heegaard splitting of $S^3$, and let $c\subset \Sigma$ be a simple closed curve which is Seifert-fibered in $H$ and primitive in $H'$. Since $ c$ is primitive in $H'$, there exists a
simple closed curve $R\subset (\Sigma\setminus c)$ which bounds a disk in $H'$. Let $K\subset S^3$ be the knot represented by $ c$, then the Dehn surgery on $K$ with the surface slope is obtained by attaching two $2$--handles along $ c$ and $R$ to $H$, then adding a $3$--handle.

Let $T$ be the torus obtained from $\Sigma$ by surgery on $ c$. 
The homology class $[c]$ in terms of $A,B,X,Y$ is given in~\cite{BergeKang}. Suppose 
\[
[c]=aA+bB+xX+yY.
\]
Then $[c]^{\perp}$ is generated (over $\mathbb Q$) by
\[
aA+xX, bB+yY, \text{ and }yA-xB.
\]
By Lemma~\ref{lem:SubHomo}, $H_1(T)$ can be identified with
$[c]^{\perp}/\langle[c]\rangle$, which is generated (over $\mathbb Q$) by
\[
U=-(aA+xX)=bB+yY,\quad V=yA-xB.
\]

Suppose that the homology class $\phi$ is given. The homology class $[\iota]$ is $0$ in $H_1(H[ c];\mathbb Q)=\langle A,B|aA+bB=0\rangle$. Notice that both $X$ and $Y$ are zero in $H_1(H)$. Suppose $\iota=uU+vV$, then $u(-aA)+v(yA-xB)$ is a rational multiple of $aA+bB$. Up to a rational scalar, we may choose $u=ax+by$, $v=ab$. Thus
\begin{equation}\label{eq:s}
\iota=-(ax+by)(aA+xX)+ab(yA-xB)=-a^2xA-abxB-x(ax+by)X.
\end{equation}

The homology class $\rho=[R]\in\ker (H_1(\partial H')\to H_1(H'))$, so it is a linear combination of $A$ and $B$. Since $[R]\cdot[ c]=0$, we get 
\begin{equation}\label{eq:r}
\rho=yA-xB,
\end{equation}
up to a scalar. Since we can determine $a,b,x,y$ and 
$\phi$ directly from~\cite{BergeKang}, we can compute $\mathcal{I}(\phi,\iota,\rho)$ using (\ref{eq:HomologyClass}), (\ref{eq:s}), (\ref{eq:r}). 

\begin{ex}
In Example~\ref{KISTPrism}, we showed that P/SF knots of Type~KIST~IV admit $P(p,q)$ surgeries, leaving whether $q>0$ or $q<0$. 
Let $K$ be a KIST~IV P/SF knot. Using the notation of Example~\ref{KISTPrism}, take
\[
(h, h') = (\widetilde p, n\widetilde p + \epsilon),  (k, k') = (\widetilde p + \epsilon, n(\widetilde p + \epsilon) + \epsilon), 
\]
and
\[(l,l')=(k-h,k'-h')=(\epsilon,n\epsilon).\]

From \cite{BergeKang} one sees that
\[[c] = (J_1J_2+J_2+1)A+((J_2+1)h+J_1J_2k)B-X+((J_2+1)h'+J_1J_2k')Y.\]
and also
\[\phi = J_1A+(J_1k-\epsilon\widetilde pl+h)B-\epsilon\widetilde pX+(J_1k'-\epsilon\widetilde pl'+h')Y.\]

Then from~(\ref{eq:s}) and~(\ref{eq:r}) we obtain
\begin{align*}
    \iota = &(J_1J_2+J_2+1)^2A +(J_1J_2+J_2+1)((J_2+1)h+J_1J_2k)B \\
    & + (-(J_1J_2+J_2+1) + ((J_2+1)h+J_1J_2k)((J_2+1)h'+J_1J_2k'))X,\\
\rho = &((J_2+1)h'+J_1J_2k')A + B.
\end{align*}

\begin{table}[t!]\centering
\caption{$\mathcal{P}^{+}$, table of $P(p,q)$ that are realizable, $q>0$}\label{table:Types+}
   \begin{tabular}{@{}lll@{}} \toprule

    Type & \begin{tabular}{l}$P(p,q)$\end{tabular} &\begin{tabular}{l}Range of parameters  \\
    ($p$ and $r$ are always odd, $p>1$)\end{tabular}\\
    \midrule
    \\
  {\bf 1A}   &\begin{tabular}{l}$P\left(p, \frac{1}{2}(p^2 + 3p + 4)\right)$\end{tabular} & \bigskip  \\ \\ 
  {\bf 1B} &\begin{tabular}{l}$P\left(p, \frac{1}{22}(p^2 + 3p + 4)\right)$\end{tabular} & \begin{tabular}{l}$p\equiv5$ or $3 \pmod{22}$\end{tabular}\\\bigskip \\   
  {\bf 2} &\begin{tabular}{l}$P\left(p, \frac{1}{|4r+2|}(r^2p - 1)\right)$\end{tabular} &\begin{tabular}{l}$r \equiv -1\pmod4$\\
$p \equiv -2r+3\pmod{4r+2}$\end{tabular} \bigskip  \\ \\   
  {\bf 3A} &\begin{tabular}{l}
  $P\left(p, \frac{1}{2r}(p-1)(p-4)\right)$\end{tabular} &\begin{tabular}{l}$p\equiv 1\pmod{2r}$ \end{tabular}\bigskip\\ \\  
  {\bf 3B} &\begin{tabular}{l}$P\left(p, \frac{1}{2r}(p-1)(p-4)\right)$\end{tabular} &\begin{tabular}{l}$p\equiv r+4\pmod{2r}$\end{tabular} \bigskip\\ \\  \bigskip
  {\bf 4} &\begin{tabular}{l}$P\left(p, \frac{1}{2r^2}\left((2r+1)^2p - 1\right)\right)$\end{tabular} &\begin{tabular}{l}$p \equiv 4r-1\pmod{2r^2}$\end{tabular}\\ \\ \bigskip
  {\bf 5} &\begin{tabular}{l}$P\left(p, \frac{1}{r^2 - 2r - 1}(r^2p - 1)\right)$\end{tabular} &\begin{tabular}{l}$r \neq 1$\\ $p \equiv -2r + 5\pmod{r^2 - 2r - 1}$\end{tabular}\\ \\ \bigskip
  {\bf Sporadic} &\begin{tabular}{l}$P(11,19)$,  $P(11,30)$, \\  $P(13, 34)$\end{tabular}&\\
  \bottomrule
\end{tabular}

\end{table}

We also have 
\begin{align*}
\gamma= &-(J_1J_2+J_2+1)+((J_2+1)h+J_1J_2k)((J_2+1)h'+J_1J_2k')\\
=&n((J_2+1)h+J_1J_2k)^2+(((J_2+1)h+J_1J_2k)\epsilon-1)(J_1J_2+J_2+1).
\end{align*}
Consider the case $J_1=2$, $J_2=-1\pm2$, then
\[
|(J_2+1)h+J_1J_2k|\ge2(h+k)=2(2\widetilde p+\epsilon)\ge6,
\]
so we get when  $n \neq 0$,
\[\mathrm{sign}(\gamma) = \mathrm{sign}(n),\] and when $n = 0$, 
\[\mathrm{sign}(\gamma) = \mathrm{sign}(\epsilon).\]

When $(J_1,J_2)=(2,1)$, we have
\begin{align*}
        \iota\cdot\rho&=8(-1 + (2\widetilde p + \epsilon) (2 n\widetilde p + (n + 2) \epsilon))(2 n\widetilde p + (n + 2) \epsilon),\\
    \rho\cdot\phi&=2\epsilon n(2\widetilde p^2-1)+4\widetilde p-3\epsilon,\\
    \phi\cdot\iota&=8(2\epsilon n\widetilde p+n+2),
\end{align*}
so 
$\iota\cdot\rho>0$, and
\[
\mathrm{sign}( \rho\cdot\phi)=\mathrm{sign}(\phi\cdot\iota)=\mathrm{sign}(\epsilon n)\quad\text{when }n\ne0,
\]
and both $\rho\cdot\phi$ and $\phi\cdot\iota$ are positive when $n=0$.
Thus $\mathcal{I}(\iota,\rho,\phi)>0$.

Similarly, when $(J_1,J_2)=(2,-3)$, we have $\mathcal{I}(\iota,\rho,\phi)>0$.
Now we can use Corollary~\ref{cor:Sign} to determine the orientation of the corresponding prism manifolds.
\end{ex}

We collect similar results in Table~\ref{BK Prism with slope sign}, which lists (possibly all) P/SF knots that admit prism manifold surgeries, together with datum on the sign of $q$. Here we take the mirror image of the knot if the surface slope is negative. We point out that if two (likely isotopic) knots match with a changemaker vector,
we report only one of the two in the table. For example, OPT I knots are not in the table, because they correspond to the family 1A, which is already covered by TKM II knots. It is highly possible that the OPT I knots are isotopic to the corresponding TKM II knots.

The list $\mathcal P^+$ of realizable $P(p,q)$ with $q>0$ is summarized in Table~\ref{table:Types+}. 
Unlike Table~\ref{table:Types}, we do not have the notion of ``changemaker vectors'' here. We divide these prism manifolds into families in such a way that it reflects the ``symmetry" between the two tables, noting that a knot may correspond to more than one family in Table~\ref{table:Types+}. We take $\mathcal P$=$\mathcal P^- \cup \mathcal P^+$.

\begin{table}\centering
\ra{1.3}
\caption{Prism Manifolds arising from multiple changemaker vectors
}\label{Overlap}
   \begin{tabular}{@{}lllll@{}} \toprule
	
	Prism manifold	& Type & Changemaker & P/SF knot  & Braid word \\ \midrule
		& 
		{\small \begin{tabular}{l}$3A$\end{tabular}} & {\small $(1,1,2,5,5)$} &\begin{tabular}{l} 
		   {\small {\bf TKM II}}  \\
		   {\small (1,2,1,2,1)}
		          \end{tabular} & {\small $(\sigma_{1}\cdots \sigma_{11})^{5}(\sigma_1)^{-2}$}  \\  {\small \begin{tabular}{l}$P(3,-14)$\end{tabular}} &&&& \\
	&	 {\small \begin{tabular}{l}$5$\end{tabular}} & {\small $(1,1,3,3,6)$} & \begin{tabular}{l}
	
		    {\small {\bf KIST IV}} \\
		    {\small (2,1,1,1,1)}
		    \end{tabular} & \begin{tabular}{l}
		  {\tiny$(\sigma_3\sigma_4\sigma_5\sigma_2\sigma_3\sigma_4\sigma_1\sigma_2\sigma_3)^3(\sigma_1\sigma_2)^{10}$}\\
		      ({\small $(19,3)$--cable of $T(3,2)$})
		    \end{tabular}  \\          
		          
\bottomrule
         & {\small \begin{tabular}{l}$2$\end{tabular}}& {\small $ (1,1,2,4,5,5)$} 
	&\begin{tabular}{l} {\small {\bf KIST I}} \\ {\small $(1,-3,-8,1,3)$} \end{tabular}& {\small $(\sigma_{1} \cdots             \sigma_{13})^5\sigma_{1}\sigma_{2}$} \\ {\small \begin{tabular}{l}$P(11,-18)$\end{tabular}}&&&& \\
& {\small \begin{tabular}{l}$3B$\end{tabular}} & {\small $(1,1,3,3,4,6)$} 
		&\begin{tabular}{l}{\small {\bf OPT III}}\\{\small $(2,-3,0,1,1)$}\end{tabular} &{\small $(\sigma_1 \cdots \sigma_5)^{10}(\sigma_1 \sigma_2 \sigma_3)^3 $}\\ \bottomrule

		 & 
		{\small \begin{tabular}{l}$3A$\end{tabular}} & {\small $(1,1,2,5,5,8^{[s]})$} & \begin{tabular}{l}
		{\small {\bf OPT II}}\\
		{\small $(2,-5,0,1,s)$} 
		\end{tabular} &  {\small $(\sigma_1 \cdots \sigma_7)^{5+8s}(\sigma_7 \cdots \sigma_1)^{2}$}
 \\  \\	
\begin{tabular}{l} {\small $P(8s+3,-(16s+14))$} \\ {\small $s \ge 1$} \end{tabular} &{\small \begin{tabular}{l}$4$\end{tabular}}& $(1,1,3,3,6,8^{[s]})$ & \begin{tabular}{l}
		{\small {\bf KIST IV}} \\
		{\small $(2,-3,1,0,s)$}
		\end{tabular} & {\small $(\sigma_{1} \cdots \sigma_{13})^{8}(\sigma_1\cdots\sigma_7)^{8s-7}$} 	\\ \\
		&\begin{tabular}{l}
		     {\small Spor}  \\
		     {\small $s=1$} 
		\end{tabular}
		& {\small $(1,1,2,4,7,7)$} & \begin{tabular}{l} 
                        {\small {\bf TKM II}} \\ 
		                {\small $(2,-2,1,3,-3)$} 
		          \end{tabular} &  {\small $(\sigma_{1} \cdots \sigma_{17})^{7} (\sigma_1 \sigma_2)^{-2}$}   \\

\bottomrule 

\end{tabular}
\caption*{\small The parameters beneath the P/SF knot types are ordered as in Table~\ref{BK Prism with slope sign}.}
\end{table}

\subsection{Manifolds corresponding to distinct changemaker vectors}\label{subsect:Overlap}

As we mentioned before, the families in Table~\ref{table:Types} are divided so that every changemaker vector corresponds to a unique family. However, a lattice $\Delta(p,q)$ may be isomorphic to the orthogonal complements of different changemaker vectors. Thus the corresponding prism manifold $P(p,q)$ lies in different families, and not just one. In Table~\ref{Overlap}, we list all such $P(p,q)$. Each of these prism manifolds is contained in two families, except that $P(11,-30)$ is contained in three families. If a P/SF knot admits a surgery to such a $P(p,q)$, we require extra information in order to detect the changemaker vector that corresponds to this knot. The information we will collect is the Alexander polynomial.

Let $\sigma$ be a changemaker vector such that $\Delta(p,q)\cong(\sigma)^{\perp}$. Assume that $\sigma$ corresponds to a knot $K$ admitting a surgery to $P(p,q)$. Using Lemma~\ref{lem:AlexanderComputation}, we can compute the Alexander polynomial $\Delta_K(T)$. 
We will explicitly exhibit a P/SF knot $K_0$ admitting a surgery to $P(p,q)$, and directly
compute $\Delta_{K_0}(T)$ to check that it is equal to the predicted Alexander polynomial $\Delta_K(T)$. So $K_0$ matches with $\sigma$. 


\subsection{Proofs of the main results}
This subsection is devoted to the proof of Theorems~\ref{thm:Classification}, and~\ref{thm:lattice}.
\begin{proof}[Proof of Theorem~\ref{thm:lattice}] If $\Delta(p,q)$ is isomorphic to a changemaker lattice $L$ then it belongs to the families classified in Sections~\ref{sec:a0=2}--\ref{sec:f=3}. As in Section~\ref{sec:pandq}, we can find a pair $(p',q')$ such that $L$ is isomorphic to $\Delta(p',q')$, and $P(p',q')$ is in $\mathcal P^-$. Using Proposition~\ref{pp}, the result follows.
\end{proof} 

Table~\ref{BigSummary} collects all results, matching each of the changemakers deduced from Propositions~\ref{f=3JR}, \ref{vf-1gappy}, \ref{f!=3JR}, \ref{f!=3tight}, and~\ref{f=3tight} to the Berge--Kang families that admit the corresponding prism manifold surgeries. In all families of changemaker vectors, the parameters $s,t$ are assumed to be non-negative, unless otherwise indicated.

\begin{proof}[Proof of Theorem~\ref{thm:Classification}]
It follows from the results in this section together with Theorem~\ref{thm:lattice} that if $P(p,q)\cong S^3_{4|q|}(K)$, then there exists a Berge--Kang knot $K'$ such that $P(p,q)\cong S^3_{4|q|}(K')$ and $\Delta_K(t)=\Delta_{K'}(t)$. By \cite{OSzLens}, $K$ and $K'$ have isomorphic knot Floer homology groups.
\end{proof}

\begin{table}\centering
\ra{1.2}
\caption{Matching P/SF knots admitting prism manifold surgeries with elements of $\mathcal P$, Part I}

   \begin{tabular}{@{}llllll@{}} \toprule
P/SF parameters   & Conditions  & $\mathcal P$ type   &&& Prism manifold parameters \\

\midrule
&\multicolumn{4}{c}
		 {\begin{tabular}{c}
		 {\bf KIST I} $(J,h,k,h',k')$\\
		 $J>0$, $2h-k>1$,\\ $k'\ne0$, $h'k-hk'=1$ 
		 \end{tabular}
		 }&\\
\cmidrule(lr){2-5}
	\addlinespace

            \begin{tabular}{ll}
                         & $(J,h,k,h',k')$\\ =&$(1,2a+1,4a,$ \\
                         &$2ab+a+b,$\\
                         &$4ab+2a-1)$\\
		&	$a\ne0,-1$ 
          \end{tabular}
                      &
                        \begin{tabular}{l}
                                 $a(2b+1)>0$ \\ \\ 
				 $a(2b+1)<0$
          \end{tabular} 
                       &
                   \begin{tabular}{l}
				$\mathcal P^-$, $2$ \\ \\ 
                                $\mathcal P^+$, $2$
			\end{tabular}
                        &&&
                     \begin{tabular}{l}
				$p = \abs{16ab+8a+2b-3}$ \\
				$|q|=|16a^2b+8a^2+8ab+b-1|$ \\
				$r=-4a-1$ 
			\end{tabular} \\ \addlinespace
\bottomrule
&\multicolumn{4}{c}
		 {\begin{tabular}{c}
		 {\bf KIST IV $(J_1,J_2,\epsilon,n,\widetilde p)$} \\
			$|J_1|>1, |J_2+1|>1,$\\ $\epsilon=\pm1, \widetilde p+\epsilon>1$
		 \end{tabular}
		 }&\\
\cmidrule(r){2-5}
\addlinespace
	
            \begin{tabular}{ll}
                         & $(J_1,J_2,\epsilon)$\\ =&$(2,1,1)$
          \end{tabular}
                      &
                        \begin{tabular}{c}
                                 $n \geq 0$ \\ \\ 
				 $n<0$
          \end{tabular} 
                       &
                   \begin{tabular}{l}
				$\mathcal P^-$, $5$ \\ \\ 
                                $\mathcal P^+$, $5$
			\end{tabular}
                        &&&
                     \begin{tabular}{l}
				$p = \abs{n(4\widetilde p^2-2)+4\widetilde p-3}$ \\
				$|q|=|n(2\widetilde p+1)^2+4\widetilde p+1|$\\
				$r = 2\widetilde p + 1$
			\end{tabular} \\ 
\addlinespace			
			\hdashline
\addlinespace			
			
\begin{tabular}{ll}
         =&$(2,1,-1)$
          \end{tabular}
                      &
                        \begin{tabular}{c}
                                 $n > 0$ \\ \\ 
				 $n\leq 0$
          \end{tabular} 
                       &
                   \begin{tabular}{l}
				$\mathcal P^-$, $5$ \\ \\ 
                                $\mathcal P^+$, $5$
			\end{tabular}
                        &&&
                     \begin{tabular}{l}
				$p = \abs{-n(4\widetilde p^2-2)+4\widetilde p+3}$ \\
				$|q|=|n(2\widetilde p-1)^2-4\widetilde p+1|$\\
				$r = 1-2\widetilde p$
			\end{tabular} \\ 
			\addlinespace			
			\hdashline
\addlinespace
			
\begin{tabular}{ll}
                         =&$(2,-3,1)$
          \end{tabular}
                      &
                        \begin{tabular}{c}
                                 $n\geq 0$ \\ \\ 
				 $n< 0$
          \end{tabular} 
                       &
                   \begin{tabular}{l}
				 $\mathcal P^-$, $4$ \\ \\ 
                                $\mathcal P^+$, $4$
			\end{tabular}
                        &&&
                     \begin{tabular}{l}
				$p = \abs{8n\widetilde p^2+8n\widetilde p+8\widetilde p+2n+3}$ \\
				$|q|=|n(4\widetilde p+3)^2+16\widetilde p+14|$\\
				$r = 2\widetilde p + 1$
			\end{tabular} \\ 
\addlinespace			
			\hdashline
\addlinespace			
			
\begin{tabular}{ll}
                         =&$(2,-3,-1)$
          \end{tabular}
                      &
                        \begin{tabular}{c}
                                 $n> 0$ \\ \\ 
				 $n\leq 0$
          \end{tabular} 
                       &
                   \begin{tabular}{l}
				 $\mathcal P^-$, $4$ \\ \\ 
                                $\mathcal P^+$, $4$
			\end{tabular}
                        &&&
                     \begin{tabular}{l}
				$p = \abs{8n\widetilde p^2-8n\widetilde p-8\widetilde p+2n+3}$ \\
				$|q|=|n(4\widetilde p-3)^2-16\widetilde p+14|$\\
				$r = 1-2\widetilde p$
			\end{tabular} \\ 
			
\addlinespace
\bottomrule

\end{tabular}

\end{table}

\begin{table} \centering

\ra{1.2}
\addtocounter{table}{-1}
\caption{Matching P/SF knots admitting prism manifold surgeries with elements of $\mathcal P$, Part II}
   \begin{tabular}{@{}llllll@{}} \toprule
P/SF parameters   & Conditions  & $\mathcal P$ type   &&& Prism manifold parameters \\

\midrule

           &\multicolumn{4}{c}
				{\begin{tabular}{c}
				{\bf TKM II $(J_1,J_2,\epsilon,n,\widetilde p)$}\\
				$|J_2|>1,\epsilon=\pm1,$\\$|\widetilde p+1|>1,n>0$
				\end{tabular}
				}& \\
\cmidrule(lr){2-5}
\addlinespace

              \begin{tabular}{ll}
					&$(J_1,J_2,\epsilon,n)$ \\
                    =&$(1,2,1,2)$
	      \end{tabular}
                      &
                        \begin{tabular}{l}
                                $\widetilde p > 0$ \\ \\ 
				 $\widetilde p <-2$
          \end{tabular} 
                       &
                   \begin{tabular}{l}
				$\mathcal P^-$, $3A$ \\ \\ 
                                $\mathcal P^+$, $3A$
			\end{tabular}
                        &&&
                     \begin{tabular}{l}
				$p = \abs{2\widetilde{p} + 1}$\\
					$|q|=|\frac{1}2(4\widetilde p^2+14\widetilde p+10)|$\\
					$r=1$
			\end{tabular} \\ 
             
			\addlinespace
\hdashline 
           
\addlinespace
              \begin{tabular}{ll}
					=&$(1,2,-1,2)$
	      \end{tabular}
                      &
                        \begin{tabular}{l}
                                $\widetilde p > 0$ \\ \\ 
				 $\widetilde p <-2$
          \end{tabular} 
                       &
                   \begin{tabular}{l}
				$\mathcal P^+$, $1A$ \\ \\ 
                                $\mathcal P^-$, $1A$
			\end{tabular}
                        &&&
                     \begin{tabular}{l}
				$p = \abs{2\widetilde{p} + 1}$\\
			        $|q|=|2\widetilde p^2+5\widetilde p+4|$
			\end{tabular}\\ 
\addlinespace			
			\hdashline
\addlinespace			

              \begin{tabular}{ll}
					 	&$(J_1,J_2,\epsilon,n,\widetilde{p})$\\
						=&$(2,2,-1,1,-3)$
					\end{tabular} & &  $ \mathcal P^+$, Spor &&&
                                          \begin{tabular}{l}
					$p=13$\\ $q=34$
					\end{tabular}\\ 
\addlinespace			
			\hdashline
\addlinespace					
					
					 \begin{tabular}{ll}
						=&$(2,-2,-1,1,-3)$
					\end{tabular} &  & $\mathcal P^+$, Spor &&&
					\begin{tabular}{l}
					$p=11$\\ $q=19$
					\end{tabular}\\ 
\addlinespace			
			\hdashline
\addlinespace					
					 \begin{tabular}{ll}
						=&$(2,2,1,1,-5)$
					\end{tabular} & &  $\mathcal P^-$, Spor &&&
                                        \begin{tabular}{l}
					$p=23$ \\ $q=-64$
					\end{tabular}\\  
\addlinespace			
			\hdashline
\addlinespace					
					 \begin{tabular}{ll}
						=&$(2,-2,1,1,-5)$
					\end{tabular} & &  $\mathcal P^-$, Spor &&&
                                         \begin{tabular}{l}
					$p=17$\\ $q=-31$
					\end{tabular}\\ 
\addlinespace			
			\hdashline
\addlinespace					
					
					 \begin{tabular}{ll}
						=&$(2,2,1,3,-3)$
					\end{tabular} &  &$\mathcal P^-$, Spor &&&
                                         \begin{tabular}{l}
					$p=13$ \\ $q=-47$
					\end{tabular}\\  
\addlinespace			
			\hdashline
\addlinespace					
					 \begin{tabular}{ll}
						=&$(2,-2,1,3,-3)$ 
					\end{tabular} &  & $\mathcal P^-$, Spor &&&
                                         \begin{tabular}{l}
					$p=11$\\ $q=-30$
					\end{tabular}\\
\addlinespace
\bottomrule
                &\multicolumn{4}{c}{\begin{tabular}{c}
                {\bf TKM V $(J_1,J_2,\delta,\epsilon,n,\widetilde p)$} 
                \end{tabular}}&\\
 \cmidrule(lr){2-5}
\addlinespace
					 \begin{tabular}{ll}
						&$(J_1,J_2, \delta,\epsilon, n,\widetilde p)$ \\=&$(-3,4,1,1,0,2)$ 
					\end{tabular} &  & $\mathcal P^+$, Spor &&&
                                        \begin{tabular}{l}
					$p=11$\\ $q=30$
					\end{tabular}\\
\addlinespace
\bottomrule

\end{tabular}

\end{table}
\clearpage

\begin{table} \centering

\ra{1.2}
\addtocounter{table}{-1}
\caption{Matching P/SF knots admitting prism manifold surgeries with elements of $\mathcal P$, Part III}\label{BK Prism with slope sign}
   \begin{tabular}{@{}llllll@{}} \toprule
P/SF parameters   & Conditions  & $\mathcal P$ type   &&& Prism manifold parameters \\

\midrule

                    &\multicolumn{4}{c}{
				\begin{tabular}{c}
				{\bf OPT II $(m,n,j,\epsilon,\widetilde p)$}\\
				$m\ge0$, $\gcd(m,n)=1$,\\
				$\widetilde p > 0$, $j\ge0$, $\epsilon=\pm1$
				\end{tabular}
				}&\\ 
 \cmidrule(lr){2-5}
\addlinespace
						
						\begin{tabular}{l}
							$(m,j,\epsilon) = (2,0,1)$\\ 
							$n$ odd
						\end{tabular} & 
                                                 \begin{tabular}{l}
							$n > 0$\\ \\
                                                        $n < -2$
						\end{tabular} &
                                                 \begin{tabular}{l}
							$\mathcal P^+$, $3A$\\ \\
                                                        $\mathcal P^-$, $3A$
						\end{tabular} & & & 
                                               \begin{tabular}{l}
						$p = \abs{2(\widetilde p+1) + n(2\widetilde p + 1)}$\\
						$|q|=|\widetilde p(n+1)^2+\frac{(n+1)(n-2)}2|$\\
						$r= 2\widetilde p+1$ 
						\end{tabular}
\\ 
\addlinespace
\bottomrule
                                             &\multicolumn{4}{c}{
			        \begin{tabular}{c}
						{\bf OPT III $(m,n,j,\epsilon,\widetilde p)$} \\
						$m\ge0$, $\gcd(m,n)=1$\\
						$\widetilde p > 0$, $j\ge0$, $\epsilon=\pm1$
					\end{tabular}	}&\\ 
\cmidrule(lr){2-5}
\addlinespace

						 \begin{tabular}{l}
						 
						 $(m,j,\epsilon) = (2,0,1)$ \\ 
						  $n$ odd
						\end{tabular} & 
                                                \begin{tabular}{l}
						           $n>0$\\\\
                                                            $n<-2$
						\end{tabular} & 

                                                 \begin{tabular}{l}
						           $\mathcal P^+$, $3B$\\ \\
                                                            $\mathcal P^-$, $3B$
						\end{tabular} &&&

\begin{tabular}{l}
						$p = \abs{4 + (2\widetilde p +3)n}$\\
						$|q|=|\frac{n^2(2\widetilde p + 3) + 3n}2|$\\
						$r=2\widetilde p +3$
		    			\end{tabular}  \\ 
\addlinespace
\bottomrule

 &\multicolumn{4}{c}{
				\begin{tabular}{c}
				{\bf OPT IV $(m,n)$}\\
				$m\ge0$, $\gcd(m,n)=1$
				\end{tabular}
				}&\\ 
 \cmidrule(lr){2-5}
\addlinespace

						\begin{tabular}{l}
						$m = 2$ \\
						$n$ odd
						
					\end{tabular} &
                                          \begin{tabular}{l}
						$n>1$\\ \\
						$n<0$
					\end{tabular} &
                                          \begin{tabular}{l}
						$\mathcal P^-$, $1B$\\ \\
						$\mathcal P^+$, $1B$
					\end{tabular} &&&

                                 \begin{tabular}{l}
				$p=|11n+6|$\\
				$|q|=|\frac{11n^2+9n+2}2|$
					\end{tabular} \\ 
\addlinespace
\bottomrule

&\multicolumn{4}{c}{
				\begin{tabular}{c}
				{\bf OPT V $(m,n)$}\\
				$m\ge0$, $\gcd(m,n)=1$
				\end{tabular}
				}&\\ 
 \cline{2-5}
\addlinespace

						\begin{tabular}{l}
						$m = 2$ \\
						$n$ odd
						
					\end{tabular} &
                                          \begin{tabular}{l}
						$n>1$\\ \\
						$n<0$
					\end{tabular} &
                                          \begin{tabular}{l}
						$\mathcal P^-$, $1B$\\ \\
						$\mathcal P^+$, $1B$
					\end{tabular} &&&

                                 \begin{tabular}{l}
					$p=|11n+8|$\\
				$|q|=|2+\frac{13n+11n^2}{2}|$
					\end{tabular} \\
\addlinespace
\bottomrule 
\end{tabular}

\end{table}

\begin{table}\centering
\caption{$D-$type changemakers vs Prism manifolds, Part I}
\resizebox{\textwidth}{!} {%

   \begin{tabular}{@{}lll@{}} \toprule

\multicolumn{1}{l}{{\small Prop$^{\text{n}}$.}} &
\multicolumn{1}{l}{{\small Changemaker vector}} & \multicolumn{1}{l}{{\small Vertex basis (with $x_*,x_{**}$ omitted) $\{ x_0,\cdots,x_n \}$}}  \\ 

\midrule

 \\
	&	{\small \begin{tabular}{l}$(1,1,2,4,7,7)$\end{tabular}}&
		  $\{v_4, -v_3, v_2\}$  \\ \\

	&		{\small \begin{tabular}{l}$(1,1,2,3,3,10)$\end{tabular}}
			 & $\{v_3, -v_5, v_2\}$   \\ \\

	 &   	{\small \begin{tabular}{l}$(1,1,2,2,4,9,9)$\end{tabular}}
	     & $\{v_5, -v_4, -v_2, -v_3\}$   \\ \\

 {\small \ref{f!=3JR}}	& 
   	{\small \begin{tabular}{l}$(1,1,2,2,5,5,14)$\end{tabular}}
	     & $\{v_4, -v_6, -v_2, -v_3\}$   \\ \\

&	   	{\small \begin{tabular}{l}$(1,1,2^{[s+1]},(2s+3)^{[2]},4s+6,8s+14)$\end{tabular}}
	     & $\{v_{s+3},v_{s+5},-v_{s+6}, -v_{s+2}, \cdots, -v_2\}$ 
		     \\ \\

&{\small \begin{tabular}{l}$(1,1,2^{[s+1]},(2s+3)^{[2]},4s+8,8s+14)$\end{tabular}}
	     & $\{v_{s+3},v_{s+6}, -v_{s+5}, -v_{s+2}, \cdots, -v_2\}$ 
		 \\ \\

	    &	{\small \begin{tabular}{l}$(1,1,2,3,3,8,8^{[s]},8s+10,(8s+18)^{[t]})$\end{tabular}}
	     & $\{v_3, -v_{s+6}, -v_2, -v_5, \cdots, -v_{s+5}, -v_{s+7}, \cdots, -v_{s+t+6}\}$ 
		    \\ \\ 
\bottomrule
	    
 \\
   &	{\small \begin{tabular}{l}$(1,1,1,1,2^{[s]})$,\\
   $s>0$
   \end{tabular}}
	    	 & $\{v_2, v_4, v_5, \cdots, v_{s+3} \}$ 
	   \\ {\small \ref{f=3JR}} \\
	
	    &	{\small \begin{tabular}{l}$(1,1,1,1,4^{[s]},4s+2,(4s + 6)^{[t]})$\end{tabular}} 
	    	&  $\{v_2, v_{s+4}, -v_4, \cdots, -v_{s+3}, -v_{s+5}, \cdots,  -v_{t+s+4}\}$  \\ \\ \bottomrule
	    
%
%
\\
	   	& {\small \begin{tabular}{l}$(1,1,2^{[s+1]},(2s+5)^{[2]}, (4s+8)^{[t]}),$\\ $t>0$ \end{tabular}} 
	    & $\{v_{s+5}, \cdots , v_{t+s+4}, v_{s+3}-v_{[1,s+2]}-v_{[s+5,t+s+4]}, v_{s+2},\cdots, v_2\}$
	    \\ \\
	
	{\small \ref{f!=3tight}} &
  {\small \begin{tabular}{l}$(1,1,2^{[s+1]},(2s+5)^{[2]})$\end{tabular}} 
	    & $\{v_{s+3}-v_{[1,s+2]}, v_{s+2}, \cdots , v_2\}$
	    \\ \\

	  & {\small \begin{tabular}{l} $(1,1,2^{[s+1]},(2s+5)^{[2]},4s+8,(4s+10)^{[t]}),$\\$t>0$\end{tabular}} 
	   	 & {\small $\{v_{s+5}, v_{s+3}-v_{[1,s+2]}-v_{[s+5,s+t+5]}, v_{s+t+5}, \cdots, v_{s+6},v_{s+2}, \cdots , v_2\}$} \\ \\
	   	\bottomrule \\
%
%
& {\small \begin{tabular}{l}$(1,1,3,3,4^{[s]},(4s+6)^{[t]}),$\\$s>0$\end{tabular}}
	   	 & $\{v_4, \cdots , v_{s+3}, v_2-v_1-v_{[4,s+3]}, v_{s+4}, \cdots, v_{t+s+3} \}$  \\ \\
	
	  &  \begin{tabular}{l}$(1,1,3,3,6^{[t]}),$\\$t>0$
\end{tabular}  & $\{v_2-v_1, v_4, \cdots , v_{t+3} \}$ \\ {\small \ref{f=3tight}} \\

 & 	{\small \begin{tabular}{l}$(1,1,3,3,4,6^{[s]}),$\\ $s>0$\end{tabular}} 
& $\{v_4, v_2-v_1-v_{[4,s+4]}, v_{s+4}, \cdots, v_5\}$  \\ \\

&   	{\small \begin{tabular}{l}$(1,1,3,3,6,8^{[s]},(8s+6)^{[t]}),$\\$s>0$\end{tabular}}  
& $\{v_2-v_1-v_{[5,s+4]}, v_{s+4}, \cdots, v_4, v_{s+5}, \cdots , v_{t+s+4}\}$\\ \\

   	 \bottomrule \\
	   	
%
	&  {\small	\begin{tabular}{l}$(1,1,2,2,2^{[s]},(2s+3)^{[2]})$,\\
	$s>0$\end{tabular}} 
	   	 & $\{v_{s+4},v_3,v_4,\cdots , v_{s+3}, -v_{[2,s+3]}\}$ 
		   	 \\ \\

	 &   {\small	\begin{tabular}{l}$(1,1,2, 4,4^{[s]}, (4s+5)^{[2]},(8s+14)^{[t]})$\end{tabular}}
	    & $\{v_{s+4}, -v_2, v_3,\cdots, v_{s+3},v_{s+6},\cdots,v_{t+s+5}\}$ 
	   	\\ {\small \ref{vf-1gappy}} \\

	&   {\small \begin{tabular}{l} $(1,1,2,2,4^{[s]},(4s+3)^{[2]},(8s+10)^{[t]}),$\\ $s>0$\end{tabular}}
	    & $\{v_{s+4}+v_{3},-v_3,-v_2,-v_4,\cdots,-v_{s+3},-v_{s+6},\cdots, -v_{t+s+5}\}$ 
	   	\\ \\
	   	
&  	{\small \begin{tabular}{l}$(1,1,2,2,3^{[2]},10^{[t]})$\end{tabular}}
	    & $\{v_{4}+v_{3},-v_3,-v_2,-v_6, \cdots, -v_{t+5}\}$ 
	   	\\ 	\\  	
\bottomrule
\end{tabular}%
}

\end{table}


\begin{table}
\addtocounter{table}{-1}
\caption{$D-$type changemakers vs Prism manifolds, Part II}
{\small
\centering
\begin{adjustbox}{max width=\textwidth}
\ra{1.4}

   \begin{tabular}{@{}llll@{}} \toprule

\multicolumn{1}{l}{Prop$^{\text{n}}$.} &
\multicolumn{1}{l}{Vertex norms $\{a_0,\dots,a_n\}$} & \multicolumn{1}{l}{Prism manifold parameters} &\multicolumn{1}{l}{$\mathcal P^-$ type}  \\
\midrule

	&	$\{4,4,3\}$ &
		
		\begin{tabular}{l}	$p=11$\\ $-q=30$ \end{tabular}
			 & Spor \\  \cdashline{2-4}

        & $\{3,6,3\}$ & 
	\begin{tabular}{l}		$p=17$\\ $-q=31$ \end{tabular}
			 & Spor \\ \cdashline{2-4}

	    & $\{5,3,3,2\}$ & 
	    \begin{tabular}{l}$p=13$\\ $-q=47$ \end{tabular}
        & Spor \\ \cdashline{2-4}

	\ref{f!=3JR} &    $\{4,5,3,2\}$ & 
	    \begin{tabular}{l} $p=23$\\ $-q=64$\end{tabular} & Spor \\ \cdashline{2-4}
	    
	    & $\{s+3,3,5,2^{[s]},3\}$ &
	    	\begin{tabular}{l} $p=22s+39$\\ $-q=22s^2+75s+64$ \end{tabular}
	     & $1B$  \\ \cdashline{2-4}
	    
	    & $\{s+3,4,4,2^{[s]},3\}$ &
	    
	    \begin{tabular}{l} $p=22s+41$\\ $-q=22s^2+79s+71$ \end{tabular}
	     & $1B$  \\ \cdashline{2-4}
	    
	    & $\{3,s+3,3,4,2^{[s]},3,2^{[t-1]}\}$ &
	   \begin{tabular}{l}
	    $p = 2r^2t +2r^2 +4r -1$ \\
	    $-q= (2r+1)^2(t+1)+8r+6$ \\ $r = -(2s + 5)$  
	   \end{tabular}  & $4$
	     \\ 
	   \bottomrule
	    
%
%
%
	\begin{tabular}{l}
	\\
	\\
	\ref{f=3JR}
	\end{tabular}
	&   $\{2,3,2^{[s-1]}\}$ &
	    
	 \begin{tabular}{l}   	$p=2s+1$\\ $-q=s+1$\\ $r=-1$ \end{tabular}
	    	 & $5$  
	   \\ \cdashline{2-4}

	&    $\{2,s+3,5,2^{[s-1]},3,2^{[t-1]}\}$ &
	    \begin{tabular}{l}
	   $p = r^2 - 6 + (r^2 - 2r - 1)t$\\
            $-q= r^2(t+1)+2r-1$\\
             $r = -(2s + 3)$ 
\end{tabular}& $5$
\\
   
         \bottomrule 

	 &  	   $\{s+4,2^{[t-1]},4,2^{[s]},3\}$&
	   
	   \begin{tabular}{l}	
	   $p=(2t+1)(2s+4)-1$\\
	   $-q=(s+2)((2t+1)(2s+4)+3)$ \\
	   $r=2t+1$ 
	   \end{tabular}
     & $3A$ \\ \cdashline{2-4}
	   	
	\ref{f!=3tight}&   	$\{s+6,2^{[s]},3\}$ &
	   
	  \begin{tabular}{l} 	
	  $p=2s+3$\\
	  $-q=2s^2+11s+14$\\
	  $r=1$
	  \end{tabular} 	 & $3A$ \\  \cdashline{2-4}
	  &   	$\{s+4, 3, 2^{[t-1]}, 3, 2^{[s]}, 3\}$ &
	   \begin{tabular}{l}
	   	$p = (2t+3)(2s+5)-4$\\
	    $-q=(2s+5)((2t+3)s+5t+6)$ \\
        $r=2t+3$ 
       \end{tabular}
       & $3B$ \\ \bottomrule

         	\end{tabular}%
\end{adjustbox}
}

\end{table}	

\begin{table}
\addtocounter{table}{-1}
\caption{$D-$type changemakers vs Prism manifolds, Part III}\label{BigSummary}

{\small
\centering
\begin{adjustbox}{max width=\textwidth}
\ra{1.4}

   \begin{tabular}{@{}llll@{}} \toprule

\multicolumn{1}{l}{Prop$^{\text{n}}$.} &
\multicolumn{1}{l}{Vertex norms $\{a_0,\dots,a_n\}$} & \multicolumn{1}{l}{Prism manifold parameters} &\multicolumn{1}{l}{$\mathcal P^-$ type}  \\ 
\midrule 
%
%
%
%

%
%

	&   $\{3,2^{[s-1]},5, s+3,2^{[t-1]}\}$ &
	 
	\begin{tabular}{l}
	   $p=(r^2-2r-1)t+2r-5$\\
	   $-q=r^2t+2r-1$\\
	   $r=2s+3$
	\end{tabular}   	
	   	 & $5$  \\ \cdashline{2-4}

\begin{tabular}{l}
	\\
	\\
	\ref{f=3tight}
\end{tabular}	
	&   	$\{6, 3, 2^{[t-1]}\}$ &
	 \begin{tabular}{l}
	 $p=2t+1$\\
	 $-q=9t+5$\\
	 $r=3$
	 \end{tabular}
	 & $5$ \\ \cdashline{2-4}

	&   	$\{3,3,2^{[s-1]},4\}$ &
	  \begin{tabular}{l} 
	  $p=6s+5$\\
	  $-q=9s+9$\\
	  $r=2s+3$
	  \end{tabular}
	  & $3B$ \\ \cdashline{2-4}

	 &  	$\{4, 2^{[s-1]},4,3, s+2,2^{[t-1]}\}$ &\begin{tabular}{l}
	        $p=2r^2t+4r-1$\\
	        $-q=(2r+1)^2t+8r+6$\\
	        $r=2s+1$
	 \end{tabular}
	 & $4$ \\
   
         \bottomrule
	&   $\{s+3,2^{[s+1]},3\}$ &
	  \begin{tabular}{l}
	    $p=2s+5$\\
	    $-q=2s^2+7s+7$      
	  \end{tabular} 
	   	
	   	 & $1A$  \\ \cdashline{2-4}
		
\begin{tabular}{l}
	\\
	\\
	\ref{vf-1gappy}
\end{tabular} & 		$\{s+3,3,4,2^{[s]},4,2^{[t-1]}\}$ &	   	
	   \begin{tabular}{l}
	        $p=16ts+30t+8s+11$\\
	        $-q=16ts^2+56ts+8s^2+24s+49t+18$\\
	        $r=4s+7$
	   \end{tabular}
	   & $2$ \\ \cdashline{2-4}
	   	
	&	$\{s+3,2,3,3,2^{[s-1]},4,2^{[t-1]}\}$ &   
	   
\begin{tabular}{l}
  $p=16ts+18t+8s+5$\\
  $-q=16ts^2+40ts+8s^2+25t+16s+7$\\
  $r=-4s-5$
\end{tabular}	   	
	   	& $2$ \\ \cdashline{2-4} 
	   	
	&   	$\{3,2,3,5,2^{[t-1]}\}$ &   
	 \begin{tabular}{l}
	   $p=18t+5$\\
	   $-q=25t+7$\\
	   $r=-5$
	 \end{tabular}  
	  	
	   	& $2$ \\

	   	\bottomrule          
         
         	\end{tabular}%
\end{adjustbox}
}
\caption*{\small In this table, $v_{[a,b]}$ means $v_a+v_{a+1}+\cdots+v_{b}$ for $a<b$. All vertex bases are presented in the form $\{ x_0,x_1,\cdots,x_n \}$.
A superscript $^{[-1]}$ at an element in the sequence of vertex norms means that the sequence is truncated at this element and the element preceding it. For example, the sequence $\{3,s+3,3,4,2^{[s]},3,2^{[t-1]}\}$ becomes $\{3,s+3,3,4,2^{[s]}\}$ when $t=0$.}
\end{table}

\clearpage


\bibliographystyle{abbrv}
\bibliography{bibliography}

\end{document}